\documentclass{article}
\usepackage{amsmath, amssymb, amsthm} 
\usepackage{geometry}
\usepackage{enumerate}
\usepackage{hyperref}

\allowdisplaybreaks
\numberwithin{equation}{section}

\theoremstyle{plain} 
\newtheorem{theorem}{Theorem}[section] 
\newtheorem{lemma}[theorem]{Lemma}     
\newtheorem{corollary}[theorem]{Corollary}
\newtheorem{proposition}[theorem]{Proposition}

\theoremstyle{definition}
\newtheorem{definition}[theorem]{Definition}
\newtheorem{example}[theorem]{Example}
\newtheorem{remark}[theorem]{Remark}

\title{Embeddings of Generalized Weighted Morrey Spaces}
\author{Zhen Liu}
\date{\today}

\begin{document}

\maketitle

\begin{abstract}
We study continuous embeddings from two classes of generalized weighted Morrey spaces into generalized Morrey spaces. 
The Komori--Shirai-type spaces use the weighted measure of cubes in their normalization, 
whereas the Samko-type spaces use the Lebesgue measure.
We first consider generalized Morrey spaces on bounded domains. 
In the non-radial setting, we establish embedding results under local radial comparability assumptions on the Morrey functions. 
As a consequence, we characterize the corresponding embeddings in the radial setting for the full range $0<p_1,p_2<\infty$. 
We then study embeddings from generalized Komori--Shirai-type weighted Morrey spaces associated with piecewise power weights and general Muckenhoupt weights. 
For the Samko-type scale, we treat both generalized and classical weighted Morrey spaces, 
again considering piecewise power weights and general Muckenhoupt weights. 
Finally, assuming that the weight is comparable to a positive constant on some cube or ball, 
we show that the continuous embeddings under consideration are not compact.
\end{abstract}

\vspace{0.2cm}
\vspace{0.2cm}


\section{Introduction}\label{sec:intro}

The classical Morrey spaces $M_u^p(\mathbb R^d)$,
$0<p\leq u<\infty$, were introduced by Morrey \cite{Mor38}. 
They consist of all $f\in L^p_{\mathrm{loc}}(\mathbb R^d)$ such that
\[
    \|f\mid M_u^p(\mathbb R^d)\|
    :=
    \sup_{x\in\mathbb R^d,\, r>0} r^{d/u}
    \Bigl(\frac{1}{|Q(x,r)|}
    \int_{Q(x,r)}|f(y)|^p\,dy\Bigr)^{1/p}
    <\infty.
\]
When $u=p$, one recovers $L^p(\mathbb R^d)$, with equivalent quasi-norms. 
Morrey spaces have become an important class of function spaces in harmonic analysis, potential theory, and the theory of partial differential equations; see, for example, \cite{RSS13,SDH20a}.

A natural generalization is obtained by replacing the power function $r^{d/u}$ in the above quasi-norm by a more general function $\varphi$. 
The generalized Morrey spaces $M_\varphi^p(\mathbb R^d)$ were studied independently by Mizuhara \cite{Miz91} and Nakai \cite{N94}. 
Their counterparts on scales of smoothness spaces have subsequently been investigated in, for example, \cite{NNS16,HL23,HLMS25}. 
We refer to \cite{AGNS17,SDH20b} and the references therein for further information.

There are also two well-studied weighted versions of Morrey spaces.
The first, introduced by Komori and Shirai \cite{KS09}, uses the weighted measure of a cube in the normalization, 
whereas the second, introduced by Samko \cite{Sam09}, uses its Lebesgue measure.
Generalizations of these two weighted scales were introduced by Guliyev \cite{Gul12} and Nakamura \cite{Nak16}, respectively.
Weighted Morrey spaces and their generalizations have also been studied in connection with partial differential equations; see, for example, \cite{DFN20,LM24,SFZ13,BS15,CZ26,GMOS16,SK24}.

In the notation of this paper, the Komori--Shirai-type space is denoted by $\mathcal M$, 
whereas the Samko-type space is denoted by $\mathfrak M$. 
For a cube $Q$ and a measurable function $f$, 
the corresponding local quantities in the generalized setting are
\[
 \frac{\varphi(Q)}{w(Q)^{1/p}}
 \Bigl(\int_Q |f|^p w\Bigr)^{1/p}
 \qquad\text{and}\qquad
 \frac{\varphi(\ell(Q))}{|Q|^{1/p}}
 \Bigl(\int_Q |f|^p w\Bigr)^{1/p},
\]
respectively. 
This distinction in normalization is essential: the two scales generally require different assumptions and give rise to different embedding conditions. 
It is also reflected in the transformation $f\mapsto fw^{1/p}$, under which the Samko-type norm reduces to an unweighted generalized Morrey norm. 
In general, no analogous reduction is available for the Komori--Shirai-type norm; see Remark~\ref{rem:Mfvw-transform}.

The main purpose of this paper is to characterize continuous embeddings between generalized Morrey-type spaces. 
More precisely, we examine the interaction among the integrability exponents $p$, the Morrey functions $\varphi$, and the weights $w$ in embeddings from generalized weighted Morrey spaces into generalized unweighted Morrey spaces.

We first study embeddings of generalized Morrey spaces on bounded domains. 
Let $\Omega\subset\mathbb R^d$ be a bounded domain, and consider embeddings of the form
\[
M_{\varphi_1}^{p_1}(\Omega)
 \hookrightarrow
M_{\varphi_2}^{p_2}(\Omega).
\]
We treat both non-radial Morrey functions $\varphi_i=\varphi_i(x,r)$ and radial Morrey functions $\varphi_i=\varphi_i(r)$, $i=1,2$. 
In the non-radial setting, we establish an embedding result under local radial comparability assumptions on the Morrey function $\varphi_1$. 
As a consequence, in the radial case, under the natural assumptions $\varphi_i\in\mathcal G_{p_i}$, $i=1,2$, and $\varphi_1(r)\to0$ as $r\to0^+$, we obtain the characterization
\[
M^{p_1}_{\varphi_1}(\Omega)
\hookrightarrow
M^{p_2}_{\varphi_2}(\Omega)
\qquad\text{if and only if}\qquad
p_2\leq p_1
\quad\text{and}\quad
\varphi_2\lesssim\varphi_1.
\]
In particular, the condition $p_2\leq p_1$ is necessary throughout the full range $0<p_1,p_2<\infty$. We prove this necessity by a direct block-packing construction that does not rely on duality; see Theorem~\ref{thm:nece-Omega} and Proposition~\ref{pro:rad-Omega}. 
This extends the corresponding result of \cite{HMS26} from $1<p_1,p_2<\infty$ to $0<p_1,p_2<\infty$ and resolves the question raised in \cite[Remark~5.14]{HMS26}.

We also compare two natural definitions of generalized Morrey spaces on bounded domains. 
In the first definition, the cubes are required to be contained in $\Omega$, whereas in the second, they are allowed to intersect the boundary of $\Omega$. 
For bounded Lipschitz domains, we give sufficient conditions under which these two spaces coincide, thereby recovering the corresponding result for classical Morrey spaces proved in \cite{HSS18}; see Proposition~\ref{prop:nor-cap}.
We further characterize these spaces as restrictions of generalized Morrey spaces on $\mathbb R^d$; see Proposition~\ref{prop:whole-space-restriction}. 
This provides a link between the domain and whole-space theories and may serve as a starting point for extension and localization arguments, as well as for the study of boundary value problems.

To transfer the bounded-domain results to the derivation of necessary embedding conditions on $\mathbb R^d$, we use cubes on which the weight is bounded above and below by positive constants. 
This leads us to revisit Proposition~2.6 and Corollary~2.7 in \cite{HS17}, which concern the existence of such regularity cubes. 
We construct an $\mathcal A_1$-weight whose singular set is dense in $\mathbb R^d$, showing that the closure of the singular set need not have measure zero, and thus that no regularity cube exists for such a weight.
We then prove that every cube contains a regularity cube, provided that the singular set of $w\in\mathcal A_\infty$ is closed and both $w$ and $w^{-1}$ are locally essentially bounded outside this set; see Proposition~\ref{pro:regularity-cubes}. 
On a regularity cube, the weight is comparable to a positive constant. 
Consequently, the bounded-domain results yield the necessity of $p_2\leq p_1$ in several of the weighted embedding theorems below. Regularity cubes also provide the local setting for our non-compactness argument.

We then turn to Komori--Shirai-type generalized weighted Morrey spaces on $\mathbb R^d$. The principal embeddings considered in this part are
\[
 \mathcal M_{\varphi_1}^{p_1}(\mathbb R^d;w)
 \hookrightarrow
 M_{\varphi_2}^{p_2}(\mathbb R^d),
\]
where the source space is weighted and the target space is unweighted. 
Particular attention is given to the piecewise power weights
\[
 w_{\alpha,\beta}(x)
 =
 \begin{cases}
  |x|^\alpha, & |x|<1,\\
  |x|^\beta, & |x|\geq 1,
 \end{cases}
 \qquad \alpha,\beta>-d.
\]
These weights provide a useful model in which the behavior near the origin and at infinity can be treated separately. 
We obtain sufficient embedding conditions expressed in terms of dyadic summability and derive corresponding necessary conditions; see Theorems~\ref{thm:Mfwwab-suff} and \ref{thm:Mfwwab-nece}. 
We also consider general Muckenhoupt weights; see Theorem~\ref{thm:wMfa-Mfb}. 
Corresponding results for the Komori--Shirai-type weighted Morrey spaces $\mathcal M_u^p(\mathbb R^d;w)$ were obtained by Haroske and Skrzypczak \cite{HS17}.

The Samko-type case is treated separately, with the Morrey functions assumed to be radial. 
We investigate embeddings of the form
\[
 \mathfrak M_{\varphi_1}^{p_1}(\mathbb R^d;w)
 \hookrightarrow
 M_{\varphi_2}^{p_2}(\mathbb R^d).
\]
Since the normalization involves $|Q|$ rather than $w(Q)$, the resulting conditions differ from those in the Komori--Shirai-type setting. 
For the piecewise power weights $w_{\alpha,\beta}$, we obtain explicit conditions involving the parameters $\alpha$ and $\beta$ and the growth of the Morrey functions; see Theorems~\ref{thm:Mf1wab-nece} and \ref{thm:Mf1wab-suff}. 
By taking $\varphi(t)=t^{d/u}$, where $0<p\leq u<\infty$, we further obtain embedding results for the classical Samko-type weighted Morrey spaces $\mathfrak M_u^p(\mathbb R^d;w)$.
Finally, if the weight is comparable to a positive constant on some cube, we prove non-compactness by using oscillatory functions supported in that cube; see Proposition~\ref{prop:compact}.

The paper is organized as follows.
Section~\ref{sec:pre} introduces the notation and definitions and collects preliminary results on generalized Morrey spaces, Muckenhoupt weights, the two classes of generalized weighted Morrey spaces, and estimates for characteristic functions.
Section~\ref{sec:Mf-Omega} is devoted to embeddings of generalized Morrey spaces on bounded domains. It treats both the radial and non-radial settings and compares the two definitions of generalized Morrey spaces on bounded domains.
Section~\ref{sec:KS} concerns embeddings of Komori--Shirai-type generalized weighted Morrey spaces, with particular emphasis on piecewise power weights and general Muckenhoupt weights. 
Section~\ref{sec:S} treats embeddings of Samko-type weighted Morrey spaces in both the generalized and classical settings, including piecewise power weights and general Muckenhoupt weights, and concludes with the non-compactness result.

\section{Preliminaries}\label{sec:pre}

We begin by fixing the notation used throughout the paper.
Let $\mathbb N=\{1,2,\ldots\}$ and $\mathbb N_0=\mathbb N\cup\{0\}$.
For $d\in\mathbb N$, let $\mathbb R^d$ denote the $d$-dimensional Euclidean space and $\mathbb Z^d$ the integer lattice.
For $a\in\mathbb R$, set $a_+:=\max\{a,0\}$.

For $x\in\mathbb R^d$ and $r>0$, let $Q(x,r)$ denote the open cube centered at $x$, with side length $r$ and sides parallel to the coordinate axes. 
Let $\mathcal Q$ be the collection of all such cubes.
For $Q\in\mathcal Q$, we denote its center, volume, and side length by $c(Q)$, $|Q|$, and $\ell(Q)=|Q|^{1/d}$, respectively.
For a function $\varphi:\mathbb R^d\times(0,\infty)\to(0,\infty)$, we use the abbreviation $\varphi(Q):=\varphi(c(Q),\ell(Q))$.

The symbol $C$ denotes a positive constant whose value may change from line to line. Subscripts indicate dependence on the corresponding parameters. 
We write $A\lesssim B$ if $A\le CB$ for a constant $C>0$ independent of the relevant variables, and $A\sim B$ if $A\lesssim B$ and $B\lesssim A$. 
Given two (quasi-)Banach spaces $X$ and $Y$, we write $X \hookrightarrow Y$ to indicate that $X \subset Y$ and that the natural embedding is continuous. 

\subsection{Generalized Morrey spaces}\label{ssec:Mf}

\begin{definition}\label{def:Mf}
    Let $0<p<\infty$ and $\varphi:\mathbb R^d\times(0,\infty)\to(0,\infty)$ be a measurable function. 
    The \emph{generalized Morrey space} $M_\varphi^p(\mathbb R^d)$ is defined to be the set of all $f\in L^p_{\mathrm{loc}}(\mathbb R^d)$ such that
    \[
    \|f \mid M_\varphi^p(\mathbb R^d)\|
    :=
    \sup_{x\in\mathbb R^d,\, r>0}
    \varphi(x,r)
    \Bigl(\frac{1}{|Q(x,r)|}\int_{Q(x,r)} |f(y)|^p\,dy \Bigr)^{1/p}
    <\infty.
    \]
    If $\varphi$ depends only on $r$, we write simply $\varphi(r)$.
\end{definition}

\begin{remark}\label{rem:Mf}
    Generalized Morrey spaces were introduced independently by Mizuhara \cite{Miz91} and Nakai \cite{N94}.
    We shall consider both the non-radial case $\varphi=\varphi(x,r)$ and the radial case $\varphi=\varphi(r)$.
    We refer the reader to \cite{AGNS17} and \cite[Chapter~12]{SDH20b} for more details about these spaces. 

    If $\varphi(x,r)= r^{d/u}$ or $\varphi(r)=r^{d/u}$, where $0<p\le u<\infty$, then $M_\varphi^p(\mathbb R^d)= M_u^p(\mathbb R^d)$ with equivalent quasi-norms. 
    If we consider the radial function $\varphi(r)=r^{-\sigma}\chi_{(0,1)}(r)$, where $-{d}/{p}\le \sigma < 0$, then $M_\varphi^p(\mathbb R^d)$ coincides with the local Morrey spaces $\mathcal L^p_\sigma$, cf. \cite{Tri13}.
    If $\varphi_0\equiv1$, then $M_p^{\varphi_0}(\mathbb R^d)=L^\infty(\mathbb R^d)$ by Lebesgue's differentiation theorem.

    In the non-radial setting, it is natural to postulate the following compatibility condition: 
    \begin{equation}\label{eq:def-C}
    \varphi(x,r)\sim \varphi(y,r)
    \end{equation}
    for all $x,y \in \mathbb R^d$, $r>0$ and $|x-y|\le C\,r$ with $C \ge 1$; see \cite[p.~22]{AGNS17} for example.
\end{remark}

For the space $M_\varphi^p(\mathbb R^d)$, it is usually required that $\varphi$ belongs to the class $\mathcal G_p$.

\begin{definition}\label{def:Gp}
Let $0<p<\infty$. The class $\mathcal G_p$ consists of all functions
$\varphi:\mathbb R^d\times(0,\infty)\to(0,\infty)$ 
such that $\varphi(x,\cdot)$ is almost increasing and 
$r\mapsto r^{-d/p}\varphi(x,r)$ is almost decreasing, uniformly in $x$.
That is, there
exist constants $c_1,c_2>0$, independent of $x$, $t$, and $s$,
such that
\[
\varphi(x,t)\le c_1\varphi(x,s) \qquad \text{and}
\qquad t^{-d/p}\varphi(x,t)\ge c_2s^{-d/p}\varphi(x,s)
\]
for all $x\in\mathbb R^d$ and $0<t\le s<\infty$. 

If $\varphi$ is independent of $x$, we simply write
$\varphi=\varphi(r)$. In this case the above conditions become
\[
\varphi(t)\lesssim\varphi(s),
\qquad\text{and}\qquad
t^{-d/p}\varphi(t)\gtrsim s^{-d/p}\varphi(s),
\]
for $0<t\le s<\infty$.
\end{definition}

\begin{remark}\label{rem:Gp}
    The class $\mathcal G_p$ is natural in this setting; see \cite[Theorem~2.5]{AGNS17} for the non-radial case and \cite[Lemma~2.2]{NNS16} (see also \cite{N00}) for the radial case. 
    More precisely, under the assumptions above, and assuming \eqref{eq:def-C} in the non-radial case, there exists a function $\psi\in\mathcal G_p$ such that $M_\varphi^p(\mathbb R^d)=M^p_\psi(\mathbb R^d)$.
    
    One can easily verify that $\mathcal{G}_{p_1}\subset \mathcal{G}_{p_2}$ for $0< p_2\le p_1< \infty$. 
    For further details and specific examples of the class $\mathcal G_p$, we refer the reader to \cite[Section~12.1.2]{SDH20b}.
    Notice that the class $\mathcal G_p$ is a special case of the more general class $\mathcal G(\delta_1,\delta_2;w)$ introduced in \cite{BYYZ26}.  

    Furthermore, in the radial case, if $0<p_i<\infty$, $\varphi_i\in\mathcal G_{p_i}$ for $i=1,2$, and $\inf_{t>0}\varphi_1(t)=0$, then
    \begin{equation}\label{eb:Mfa-Mfb}
    M^{p_1}_{\varphi_1}(\mathbb R^d)
    \hookrightarrow 
    M^{p_2}_{\varphi_2}(\mathbb R^d)
    \qquad\text{if and only if }\qquad
    p_2\le p_1\quad\text{and}\quad \varphi_2\lesssim \varphi_1.
    \end{equation}
    A non-radial counterpart will be established in Corollary~\ref{cor:Mfab-nr}.
\end{remark}

\subsection{Muckenhoupt weights}\label{ssec:Ap}

We briefly recall the basic notation and properties of the Muckenhoupt classes $\mathcal A_p$, introduced in \cite{Muc72,Muc74}. 
For comprehensive accounts of the theory of Muckenhoupt weights, we refer the reader to \cite{Gra24,Ste93,Tor86}.

Throughout the paper, a \emph{weight} refers to a locally integrable function $w \in L^1_{\mathrm{loc}}(\mathbb R^d)$ such that $w(x)>0$ for almost every $x \in \mathbb R^d$. 
For a bounded measurable set $E \subset \mathbb R^d$, we denote its weighted measure by
\[
w(E):=\int_E w(x)\,dx.
\]
We denote the classical Hardy--Littlewood maximal operator by $\mathbf M$, defined by
\[
\mathbf M f(x)
:=\sup_{B \ni x}
\frac{1}{|B|}\int_{B} |f(y)|\,dy,
\qquad x\in\mathbb R^d,
\]
where $B$ is an open ball given by
\[
B=B(a,R):=\{b\in \mathbb R^d:\, |b-a|<R\},\qquad R>0,
\]
and the supremum is taken over all open balls $B \subset \mathbb R^d$ containing $x$.

\begin{definition}\label{def:Ap}
Let $w$ be a weight on $\mathbb R^d$.

\begin{enumerate}[\bfseries\upshape (i)]
\item
Let $1<p<\infty$. We say that $w$ belongs to the
Muckenhoupt class $\mathcal A_p$ if
\[
[w]_{\mathcal A_p}
:=
\sup_{B\subset\mathbb R^d}
\Bigl(
\frac{1}{|B|}
\int_B w(x)\,dx
\Bigr)
\Bigl(
\frac{1}{|B|}
\int_B w(x)^{-1/(p-1)}\,dx
\Bigr)^{p-1}
<\infty,
\]
where the supremum is taken over all balls $B\subset\mathbb R^d$.

\item
We say that $w$ belongs to the class $\mathcal A_1$ if there exists a constant
$0<A<\infty$ such that
\[
\mathbf Mw(x)\le A\,w(x)
\]
for almost all $x\in\mathbb R^d$.

\item
The class $\mathcal A_\infty$ is defined by
\[
\mathcal A_\infty:=\bigcup_{p\ge1}\mathcal A_p.
\]
\end{enumerate}
\end{definition}

\begin{remark}\label{rem:Aq}
We shall use the following characterization of $\mathcal A_p$: 
for $1\le p<\infty$, a weight $w$ belongs to $\mathcal A_p$ if and
only if there exists $C>0$ such that 
\[
    \frac{1}{|B|}\int_B f(y)\,dy
    \le
    \Bigl(\frac{C}{w(B)}
    \int_B f(x)^p w(x)\,dx \Bigr)^{1/p}
\]
for all non-negative measurable functions $f$ and all balls $B\subset\mathbb R^d$. 
In particular, by taking $f=\chi_E$, where $E\subset B$ is measurable, one obtains
\begin{equation}\label{eq:EB}
    \frac{|E|}{|B|}
    \le
    C \Bigl(\frac{w(E)}{w(B)}\Bigr)^{1/p},
    \qquad E\subset B,\quad w\in\mathcal A_p,\quad 1\le p<\infty.
\end{equation}

Note that the smallest constant $A$ in Definition~\ref{def:Ap}(ii) is exactly $[w]_{\mathcal A_1}$, and in this case
\begin{equation}\label{eq:pro-A1}
    \frac{1}{|Q|}\int_Q w(y)\, dy \le [w]_{\mathcal A_1} \operatorname*{ess\,inf}_{y\in Q} w(y),
\end{equation}
cf. \cite[(8.2.18)]{Gra24} or \cite[(2.4)]{Tor86}.

We shall also use the openness of the Muckenhoupt classes:  
if $w\in\mathcal A_p$ for some $p>1$, then there exists $q<p$ such that $w\in\mathcal A_q$. 
This is a consequence of the reverse H\"older property of $\mathcal A_p$-weights. 
Following \cite{HS08}, we define the number 
\[
q_w:=\inf\{q\ge 1: w\in\mathcal A_q\},\qquad 
w\in\mathcal A_\infty.
\]
Then $1\le q_w<\infty$. 
Moreover, by the openness property, if $w\in\mathcal A_{q_w}$, then $q_w=1$.

\end{remark}

\begin{example}\label{ex:wab}
\begin{enumerate}[\bfseries\upshape (i)]
\item Let $w_\alpha(x)=|x|^\alpha$. Then, for $1<p<\infty$, $w_\alpha\in\mathcal A_p$ if and only if $-d<\alpha<d(p-1)$, while $w_\alpha\in\mathcal A_1$ if and only if $-d<\alpha\le 0$.
A natural extension is the piecewise power weight
\begin{align}\label{eq:def-wab}
w_{\alpha,\beta}(x)=
\begin{cases}
|x|^\alpha, & |x|<1,\\
|x|^\beta, & |x|\ge 1,
\end{cases}
\end{align}
where $\alpha,\beta>-d$. Then, for $1<p<\infty$,
$w_{\alpha,\beta}\in\mathcal A_p$ if and only if $-d<\alpha,\beta<d(p-1)$, while 
$w_{\alpha,\beta}\in\mathcal A_1$ if and only if $-d<\alpha,\beta\le 0$.
In particular, $q_{w_{\alpha,\beta}}=1+\frac{\max\{\alpha,\beta,0\}}{d}$.

\item A logarithmic refinement of the above piecewise power weights,
considered in \cite{HS11b}, is given by
\[
w_{(\boldsymbol\alpha,\boldsymbol\beta)}(x)
:=
\begin{cases}
|x|^{\alpha_1}(1-\log |x|)^{\alpha_2},
& |x|\le 1,\\
|x|^{\beta_1}(1+\log |x|)^{\beta_2},
& |x|>1,
\end{cases}
\]
where $\boldsymbol\alpha=(\alpha_1,\alpha_2)$, $\alpha_1>-d$, $\alpha_2\in\mathbb R$,
and $\boldsymbol\beta=(\beta_1,\beta_2)$, $\beta_1>-d$, $\beta_2\in\mathbb R$.
Further examples can be found in \cite[Example~2.2]{HS17}.

\end{enumerate}
\end{example}

Let $\mathcal D := \{Q_{j,k}:=2^{-j}([0,1)^d+k): j\in\mathbb N_0,\ k\in\mathbb Z^d\}$
be the collection of dyadic cubes under consideration in $\mathbb R^d$. 
We recall the notion of the \emph{set of singularities} of $\mathcal A_\infty$-weights, introduced in \cite{HS11}. For brevity, we refer to it as the singular set.

\begin{definition}\label{def:sing}
Let $w\in\mathcal A_\infty$. The \emph{singular set} of $w$ is defined by
\[
\mathbf S_{\mathrm{sing}}(w):=\mathbf S_0(w)\cup \mathbf S_\infty(w),
\]
where
\begin{align*}
    \mathbf S_0(w)
    &:=
    \biggl\{
    x\in\mathbb R^d:
    \inf_{Q_{j,k}\ni x}\frac{w(Q_{j,k})}{|Q_{j,k}|}=0
    \biggr\},\\
    \mathbf S_\infty(w)
    &:=
    \biggl\{
    x\in\mathbb R^d:
    \sup_{Q_{j,k}\ni x}\frac{w(Q_{j,k})}{|Q_{j,k}|}=\infty
    \biggr\},
\end{align*}
and the infimum and supremum are taken over all dyadic cubes $Q_{j,k}$ containing $x$.
\end{definition}

\begin{remark}\label{rem:sing}
This notion is a special case of $\mathbf S_{\mathrm{sing}}(w_1,w_2)$ introduced in \cite{HS11},
corresponding to the choice $w_1=w$ and $w_2\equiv1$; see also \cite[Remark~2.4]{HS17}.

For the weight $w_{\alpha,\beta}$ defined in \eqref{eq:def-wab}, the singular set is given by
\[
\mathbf S_{\mathrm{sing}}(w_{\alpha,\beta})
=
\begin{cases}
\{0\}, & \alpha\neq0,\\
\emptyset, & \text{otherwise},
\end{cases}
\]
as detailed in \cite[Example~2.5]{HS17}.
\end{remark}

We shall also use the following terminology. 
A cube $R\subset\mathbb R^d$ is called a \emph{regularity cube} for $w$ if there exist constants $0<c_R\leq C_R<\infty$ such that
\[
c_R\leq w(x)\leq C_R
\qquad\text{for a.e. }x\in R.
\]
A \emph{regularity ball} is defined analogously.

It is known that, for $w\in\mathcal A_\infty$,
\[
\bigl|\mathbf S_{\mathrm{sing}}(w)\bigr|=0;
\]
see \cite[Proposition~4.5]{HS11}. 
This does not, however, imply that the closure of the singular set has measure zero. 
In fact, the conclusion $|\overline{\mathbf S_{\mathrm{sing}}(w)}|=0$ stated in \cite[Proposition~2.6]{HS17} does not hold without an additional assumption, as the following example shows.

\begin{example}\label{ex:no-regularity-cube}
There exists a weight $w^\ast\in\mathcal A_1$ such that
\[
\overline{\mathbf S_{\mathrm{sing}}(w^\ast)}
=\mathbb R^d.
\]
Moreover, $w^\ast$ is essentially unbounded on every cube and hence
has no regularity cube.

Indeed, let $\{q_n\}_{n\geq1}$ be a countable dense subset of
$\mathbb R^d$, fix $0<\gamma<d$, and define
\[
w^\ast(x)
:=
\sum_{n=1}^{\infty}2^{-n}|x-q_n|^{-\gamma}.
\]
First, we verify that $w^\ast$ is locally integrable. Let
$K\subset\mathbb R^d$ be compact. For all $a\in\mathbb R^d$, we
have
\[
\begin{aligned}
\int_K |x-a|^{-\gamma}\,dx
&=
\int_{K\cap B(a,1)}|x-a|^{-\gamma}\,dx
+
\int_{K\setminus B(a,1)}|x-a|^{-\gamma}\,dx
\\
&\leq
\int_{B(a,1)}|x-a|^{-\gamma}\,dx+|K|
\\
&=
\int_{B(0,1)}|y|^{-\gamma}\,dy+|K|.
\end{aligned}
\]
Since $0<\gamma<d$, the last expression is finite and independent of $a$. Thus we set 
\[
C_K
:=
\sup_{a\in\mathbb R^d}
\int_K |x-a|^{-\gamma}\,dx
<\infty.
\]
By Tonelli's theorem,
\[
\begin{aligned}
\int_K w^\ast(x)\,dx
=
\sum_{n=1}^{\infty}
2^{-n}\int_K|x-q_n|^{-\gamma}\,dx
\leq
C_K\sum_{n=1}^{\infty}2^{-n}
=C_K<\infty.
\end{aligned}
\]
Thus $w^\ast\in L^1_{\mathrm{loc}}(\mathbb R^d)$.

Moreover, $|\cdot-a|^{-\gamma}\in\mathcal A_1$, 
with an $\mathcal A_1$ constant independent of $a$. 
Hence
\[
\begin{aligned}
\mathbf M w^\ast
\leq
\sum_{n=1}^{\infty}
2^{-n}\mathbf M\bigl(|\cdot-q_n|^{-\gamma}\bigr)
\leq
C_{d,\gamma}
\sum_{n=1}^{\infty}
2^{-n}|\cdot-q_n|^{-\gamma}
=
C_{d,\gamma}w^\ast
\qquad\text{a.e.},
\end{aligned}
\]
and therefore $w^\ast\in\mathcal A_1$.

We next show that every $q_N$ belongs to
$\mathbf S_{\mathrm{sing}}(w^\ast)$. Fix $N\in\mathbb N$, and for
each $j\in\mathbb N_0$ let $Q_j(q_N)\in\mathcal D$ be a dyadic cube
of side length $2^{-j}$ containing $q_N$. 
Since $w^\ast(x)\geq2^{-N}|x-q_N|^{-\gamma}$ and $|x-q_N|\leq\sqrt d\,2^{-j}$, $x\in Q_j(q_N)$, 
we obtain
\[
w^\ast(x)
\geq
d^{-\gamma/2}2^{-N}2^{j\gamma}
\qquad\text{for a.e. }x\in Q_j(q_N).
\]
It follows that
\[
\frac{w^\ast(Q_j(q_N))}{|Q_j(q_N)|}
\geq
d^{-\gamma/2}2^{-N}2^{j\gamma}
\longrightarrow\infty
\qquad\text{as }j\to\infty.
\]
Hence $q_N\in\mathbf S_{\infty}(w^\ast)\subset\mathbf S_{\mathrm{sing}}(w^\ast)$. 
By the density of $\{q_N:N\in\mathbb N\}$,
\[
\overline{\mathbf S_{\mathrm{sing}}(w^\ast)}
\supset
\overline{\{q_N:N\in\mathbb N\}}
=\mathbb R^d.
\]
Consequently,
\[
\overline{\mathbf S_{\mathrm{sing}}(w^\ast)}
=\mathbb R^d.
\]
Thus, even for a weight in
$\mathcal A_1\subset\mathcal A_\infty$, the closure of the singular
set need not have measure zero.

Finally, we show that $w^\ast$ is essentially unbounded on every
cube. Let $Q\subset\mathbb R^d$ be a cube with nonempty interior.
By the density of $\{q_n\}_{n\geq1}$, there exists $q_N\in Q^\circ$.
We may therefore choose $\rho>0$ such that $B(q_N,\rho)\subset Q$. 
Given $L>0$, define
\[
r_L
:=
\min\Bigl\{
\rho,
\Bigl(\frac{2^{-N}}{L}\Bigr)^{1/\gamma}
\Bigr\}.
\]
The ball $B(q_N,r_L)$ has positive measure, and for a.e. 
$x\in B(q_N,r_L)$, $|x-q_N|<r_L
\leq
({2^{-N}}/{L})^{1/\gamma}$. 
Thus 
\[
w^\ast(x)
\geq
2^{-N}|x-q_N|^{-\gamma}
>L.
\]
Note that $L>0$ is arbitrary, we have 
\[
\operatorname*{ess\,sup}_{x\in Q}w^\ast(x)=\infty,
\]
and hence no cube can be a regularity cube for $w^\ast$. 
In particular, the regularity-cube conclusion stated in
\cite[Corollary~2.7]{HS17} also requires an additional assumption.
\end{example}

However, the closure of the singular set of $w\in\mathcal A_\infty$ has measure zero is valid if the singular set is closed.

\begin{proposition}\label{pro:closed-sing}
Let $w\in\mathcal A_\infty$. If $\mathbf S_{\mathrm{sing}}(w)$ is closed, then 
\[
\bigl|
\overline{\mathbf S_{\mathrm{sing}}(w)}
\bigr|=0.
\]
\end{proposition}

\begin{proof}
By \cite[Proposition~4.5]{HS11}, $|\mathbf S_{\mathrm{sing}}(w)|=0$. 
Since $\mathbf S_{\mathrm{sing}}(w)$ is closed, $\overline{\mathbf S_{\mathrm{sing}}(w)}
=
\mathbf S_{\mathrm{sing}}(w)$
and the conclusion follows.
\end{proof}

We finally give a sufficient condition for the existence of regularity
cubes and balls.

\begin{proposition}\label{pro:regularity-cubes}
Let $w\in\mathcal A_\infty$. Assume that
$\mathbf S_{\mathrm{sing}}(w)$ is closed and that
\begin{equation}\label{eq:LEB}
w,w^{-1}
\in
L^\infty_{\mathrm{loc}}
\bigl(
\mathbb R^d\setminus\mathbf S_{\mathrm{sing}}(w)
\bigr).
\end{equation}
Then every cube $Q\subset\mathbb R^d$ contains a regularity cube
$\widetilde Q$ for $w$ such that $\overline{\widetilde Q}\subset Q$. 
The result remains valid for balls.
\end{proposition}

\begin{proof}
Set $\Omega
:=
\mathbb R^d\setminus\mathbf S_{\mathrm{sing}}(w)$.
By $\mathbf S_{\mathrm{sing}}(w)$ is closed, $\Omega$ is open.

Let $Q\subset\mathbb R^d$ be a cube. Since $Q^\circ$ has positive
measure, it cannot be contained in
$\mathbf S_{\mathrm{sing}}(w)$. Hence, we may choose $x\in Q^\circ\cap\Omega$. 
Since $Q^\circ\cap\Omega$ is open, there exists a cube
$\widetilde Q$ containing $x$ such that 
\[
\overline{\widetilde Q}
\subset
Q^\circ\cap\Omega.
\]
In particular, $\overline{\widetilde Q}$ is a compact subset of
$\Omega$.

By \eqref{eq:LEB}, both $w$ and $w^{-1}$ are essentially bounded on
every compact subset of $\Omega$. Applying this to
$\overline{\widetilde Q}$ gives $w,w^{-1}\in L^\infty(\overline{\widetilde Q})$, 
and therefore $w,w^{-1}\in L^\infty(\widetilde Q)$. 
Consequently, 
\[
0<
\frac{1}{\|w^{-1}\|_{L^\infty(\widetilde Q)}}
\leq w(y)\leq
\|w\|_{L^\infty(\widetilde Q)}
<\infty
\qquad\text{for a.e. }y\in\widetilde Q.
\]
Thus $\widetilde Q$ is a regularity cube for $w$, and $\overline{\widetilde Q}\subset Q$. 

The proof for balls is identical: one may choose
$x\in B^\circ\cap\Omega$ and a smaller ball $\widetilde B$
such that $\overline{\widetilde B}
\subset
B^\circ\cap\Omega$.
\end{proof}

\begin{remark}\label{rem:regularity-cubes-applications}
The assumptions of Proposition~\ref{pro:regularity-cubes} are not
automatic for an arbitrary weight $w\in\mathcal A_\infty$.
Thus whenever the existence of regularity cubes is used
below to derive a necessary condition, these assumptions must either
be imposed explicitly or verified for the weight under
consideration.

However, for the weight $w_{\alpha,\beta}$ defined in \eqref{eq:def-wab}, where 
$\alpha,\beta>-d$, no additional assumptions are needed. 
Indeed, by \cite[Example~2.5]{HS17},
\[
\mathbf S_{\mathrm{sing}}(w_{\alpha,\beta})
=
\begin{cases}
\{0\}, & \alpha\neq0,\\
\emptyset, & \alpha=0.
\end{cases}
\]
In particular, $\mathbf S_{\mathrm{sing}}(w_{\alpha,\beta})$ is
closed.

We next verify \eqref{eq:LEB}. 
Let $K\subset\mathbb R^d\setminus\{0\}$ be compact. Then, by the Weierstrass extreme value theorem, there exist constants $0<c_K\leq C_K<\infty$ such that $c_K\leq w_{\alpha,\beta}(x)\leq C_K$ for all $x\in K$. 
Consequently, $w_{\alpha,\beta},
w_{\alpha,\beta}^{-1}\in L^\infty(K)$.

It remains only to consider the origin when $\alpha=0$, in which case the singular set is empty. 
For $0<|x|<1$, we have $w_{0,\beta}(x)=|x|^0=1$.
Thus both $w_{0,\beta}$ and $w_{0,\beta}^{-1}$ are essentially bounded
in a neighborhood of the origin. We have therefore verified \eqref{eq:LEB}. 
Hence $w_{\alpha,\beta}$ satisfies all the assumptions of
Proposition~\ref{pro:regularity-cubes}.
\end{remark}

\subsection{Weighted spaces}\label{ssec:M/Mfvw}

Let $0<p<\infty$ and $w\in\mathcal A_\infty$ be a Muckenhoupt weight. 
The weighted Lebesgue space $L^p(\mathbb R^d; w)$ is defined to be the set of all measurable functions $f$ such that 
\[
\|f\mid L^p(\mathbb R^d;w)\|:=\Bigl(\int_{\mathbb R^d}|f(x)|^p w(x)\,dx\Bigr)^{1/p}<\infty.
\]
The corresponding weighted local Lebesgue space $L^p_{\mathrm{loc}}(\mathbb R^d; w)$ contains all measurable functions $f$ such that $\int_E |f(x)|^p w(x)\,dx < \infty$ for all compact sets $E\subset\mathbb R^d$. 
Under the usual convention for $p=\infty$, the weight does not affect the norm, so that $L^\infty(\mathbb R^d;w)=L^\infty(\mathbb R^d)$.
We therefore restrict attention to $p<\infty$ in the weighted setting.

We recall the two well-studied classes of weighted Morrey spaces. 

\begin{definition}\label{def:Mvw}
    Let $0<p\le u<\infty$ and $w$ be a weight. 
    \begin{enumerate}[\bfseries\upshape (i)]
    \item The \emph{weighted Morrey space} $\mathcal M_u^p(\mathbb R^d; w)$ is defined to be the set of all $f\in L^p_{\mathrm{loc}}(\mathbb R^d; w)$ such that
    \[
    \|f\mid \mathcal M_u^p(\mathbb R^d; w)\|:=\sup_{x\in\mathbb R^d,\, r>0}w(Q(x,r))^{1/u-1/p}
    \Bigl(\int_{Q(x,r)}|f(y)|^p w(y)\,dy\Bigr)^{1/p}<\infty.
    \]

    \item The \emph{weighted Morrey space} $\mathfrak M_u^p(\mathbb R^d; w)$ is defined to be the set of all $f\in L^p_{\mathrm{loc}}(\mathbb R^d; w)$ such that
    \[
    \|f\mid \mathfrak M_u^p(\mathbb R^d; w)\|:=\sup_{x\in\mathbb R^d,\, r>0}|Q(x,r)|^{1/u-1/p}
    \Bigl(\int_{Q(x,r)}|f(y)|^p w(y)\,dy\Bigr)^{1/p}<\infty.
    \]
    \end{enumerate}
\end{definition}

\begin{remark}\label{rem:Mvw}
    The first space is the Komori--Shirai-type weighted Morrey space introduced in \cite{KS09}, whose normalization is determined by $w(Q)$; the second is the Samko-type weighted Morrey space introduced in \cite{Sam09}, whose normalization is determined by $|Q|$.

    When $u=p$, both spaces coincide with the weighted Lebesgue space $\mathcal M_p^p(\mathbb R^d;w)=\mathfrak M_p^p(\mathbb R^d;w)=L^p(\mathbb R^d;w)$. 
    When $w\equiv 1$, both recover the classical Morrey space $\mathcal M_u^p(\mathbb R^d;1)=\mathfrak M_u^p(\mathbb R^d;1)=M_u^p(\mathbb R^d)$. 
\end{remark}

We now recall the corresponding generalized weighted Morrey spaces. 

\begin{definition}\label{def:Mfvw}
    Let $0<p<\infty$ and $w$ be a weight. 
    \begin{enumerate}[\bfseries\upshape (i)]
    \item Let $\varphi=\varphi(x,r)$ be a positive measurable function on
$\mathbb R^d\times(0,\infty)$. The \emph{generalized weighted Morrey space} $\mathcal M_\varphi^p(\mathbb R^d; w)$ is defined to be the set of all $f\in L^p_{\mathrm{loc}}(\mathbb R^d; w)$ such that
    \[
    \|f\mid \mathcal M_\varphi^p(\mathbb R^d; w)\|:=\sup_{x\in\mathbb R^d,\, r>0}\frac{\varphi(x,r)}{w(Q(x,r))^{1/p}}
    \Bigl(\int_{Q(x,r)}|f(y)|^p w(y)\,dy\Bigr)^{1/p}<\infty.
    \]

    \item Let $\varphi=\varphi(r)$ be a positive measurable function on
$(0,\infty)$. The \emph{generalized weighted Morrey space} $\mathfrak M_\varphi^p(\mathbb R^d; w)$ is defined to be the set of all $f\in L^p_{\mathrm{loc}}(\mathbb R^d; w)$ such that
    \[
    \|f\mid \mathfrak M_\varphi^p(\mathbb R^d; w)\|:=\sup_{x\in\mathbb R^d,\, r>0}\frac{\varphi(r)}{|Q(x,r)|^{1/p}}
    \Bigl(\int_{Q(x,r)}|f(y)|^p w(y)\,dy\Bigr)^{1/p}<\infty.
    \]
\end{enumerate}
\end{definition}

\begin{remark}\label{rem:Mfvw}
The Komori--Shirai-type space $\mathcal{M}^{p}_{\varphi}(\mathbb{R}^{d};w)$ was introduced by Guliyev~\cite{Gul12}, while the Samko-type space $\mathfrak{M}^{p}_{\varphi}(\mathbb{R}^{d};w)$ was introduced by Nakamura~\cite{Nak16}.

When $w\equiv1$ and $\varphi$ is radial, both spaces reduce to the unweighted generalized Morrey space
\[
\mathcal M_\varphi^p(\mathbb R^d;1)=\mathfrak M_\varphi^p(\mathbb R^d;1)= M_\varphi^p(\mathbb R^d).
\]
Furthermore, let $0<p\le u<\infty$, then  
\[
\begin{cases}
    \mathcal M_\varphi^p(\mathbb R^d;w)=\mathcal M_u^p(\mathbb R^d;w),&\quad\text{if } \varphi(x,r)= w(Q(x,r))^{1/u};\\[2mm]
    \mathfrak M_\varphi^p(\mathbb R^d;w)=\mathfrak M_u^p(\mathbb R^d;w),&\quad\text{if }  \varphi(r)= r^{d/u}.
\end{cases}
\]
\end{remark}

Characteristic function estimates will be used repeatedly in the proof of our main results. We begin with the Komori–Shirai-type scale, and recall the weighted class $\mathcal G_p^w$.

\begin{definition}\label{def:Gpw}
Let $0<p<\infty$. A function
$\varphi:\mathbb R^d\times(0,\infty)\to(0,\infty)$ belongs to $\mathcal G_p^w$ if it is almost increasing in the
second variable, and, for all
$Q_0\in\mathcal Q$,
\[
\sup_{\substack{Q\in\mathcal Q,\ Q\cap Q_0\neq\emptyset\\
\ell(Q)\le \ell(Q_0)}}\varphi(Q)\lesssim \varphi(Q_0)
\]
and
\[
\sup_{\substack{Q\in\mathcal Q,\ Q\cap Q_0\neq\emptyset\\
\ell(Q)\ge \ell(Q_0)}}
\varphi(Q)w(Q)^{-1/p}
\lesssim
\varphi(Q_0)w(Q_0)^{-1/p}.
\]
\end{definition}

\begin{remark}\label{rem:Gpw}
This weighted class was introduced in \cite{DGH18}; see also \cite{GE21}.

If $w\in\mathcal A_\infty$ and $0<p<u<\infty$, then the function $\varphi(x,r)=w(Q(x,r))^{1/u}$
belongs to $\mathcal G_p^w$. 
Indeed, this follows from the doubling property of $w$ and the monotonicity of the weighted measure with respect to set inclusion.

When $w\equiv1$, $\mathcal G_p^1$ reduces to a cube-based version of
$\mathcal G_p$. In the radial case $\varphi=\varphi(r)$, it coincides with
the usual class $\mathcal G_p$.
\end{remark}

\begin{lemma}\label{lem:1Qww}
    Let $0<p<\infty$, $\varphi\in\mathcal G_p^w$ and $Q_0=Q(x_0,r_0) \in \mathcal Q$. Then
    \begin{equation}\label{eq:norm-1Qww}
        \|\chi_{Q_0} \mid \mathcal M_\varphi^p(\mathbb R^d;w)\|\sim \varphi(x_0,r_0).
    \end{equation}
\end{lemma}

\begin{proof}
    The proof is similar to those of \cite[Lemma~3.2]{DGH18} and \cite[Lemma~31]{GE21}. We provide the details here for the reader's convenience.

    Let $Q=Q(x,r)\in\mathcal Q$. A direct computation gives
    \begin{align*}
        \|\chi_{Q_0} \mid \mathcal M_\varphi^p(\mathbb R^d;w)\|
        &=\sup_{Q\in\mathcal Q}\varphi(Q)w(Q)^{-1/p}\bigl(w(Q\cap Q_0)\bigr)^{1/p}\\
        &\geq \varphi(Q_0)w(Q_0)^{-1/p}w(Q_0)^{1/p}=\varphi(Q_0).
    \end{align*}
    Thus, it remains to prove the upper bound $\|\chi_{Q_0} \mid  \mathcal M_\varphi^p(\mathbb R^d;w)\|\lesssim \varphi(Q_0)$. If $\ell(Q)\leq \ell(Q_0)$, we know $\varphi(Q)\lesssim \varphi(Q_0)$, leading to 
    \[
    \varphi(Q)w(Q)^{-1/p}\bigl(w(Q\cap Q_0)\bigr)^{1/p} \le \varphi(Q) \lesssim \varphi(Q_0).
    \]
    If $\ell(Q)\geq \ell(Q_0)$, then the condition $\varphi(Q)w(Q)^{-1/p}\lesssim \varphi(Q_0)w(Q_0)^{-1/p}$ implies
    \[
    \varphi(Q)w(Q)^{-1/p}\bigl(w(Q\cap Q_0)\bigr)^{1/p} \lesssim \varphi(Q_0)w(Q_0)^{-1/p}w(Q_0)^{1/p} = \varphi(Q_0).
    \]
    Taking the supremum over all $Q\in\mathcal Q$ completes the proof.
\end{proof}

We next consider characteristic functions in the Samko–type scale. For this purpose, we recall the class $\mathcal{B}_{\varphi,p}$ introduced in \cite{Nak16}.

\begin{definition}\label{def:Bvp}
    Let $0<p<\infty$ and $\varphi\in\mathcal G_p$. 
    A weight $w$ belongs to the \emph{class $\mathcal{B}_{\varphi,p}$} if there exists a constant $C_{\varphi,p}> 0$ such that 
    \begin{equation}\label{eq:def-Bvp}
        \sup_{Q\in\mathcal Q,\, Q\subset Q_0}\Phi_{\varphi,p,w}(Q)\leq C_{\varphi,p}\,\Phi_{\varphi,p,w}(Q_0)
    \end{equation}
    for any $Q_0\in\mathcal Q$, where $\Phi_{\varphi,p,w}$ is given by
    \[
    \Phi_{\varphi,p,w}(Q):=\varphi(\ell(Q))\Bigl(\frac{w(Q)}{|Q|}\Bigr)^{1/p},\qquad Q\in\mathcal Q.
    \]
\end{definition}

\begin{remark}\label{rem:Bvp}
    For $\varphi(t) = t^{d/u}$, $0<p\leq u<\infty$, Definition~\ref{def:Bvp} reduces to the class $\mathcal{B}_{u,p}$. Specifically, a weight $w$ is in the \emph{class $\mathcal{B}_{u,p}$} if for any $Q_0\in\mathcal Q$,
    \begin{equation}\label{eq:def-Bup}
        \sup_{Q\in\mathcal Q,\, Q\subset Q_0}\Phi_{u,p,w}(Q)\leq C_{u,p}\, \Phi_{u,p,w}(Q_0)
    \end{equation}
    holds, where $\Phi_{u,p,w}(Q):=|Q|^{1/u-1/p}w(Q)^{1/p}$. Further information about the class $\mathcal{B}_{u,p}$ can be found in \cite[Section~15.2.1]{SDH20b}.
\end{remark}

\begin{example}\label{ex:Bp}
    Let $1\le p<\infty$, $\varphi\in \mathcal G_p$, and $\alpha>-d$. The power function $w(x)=|x|^\alpha$ belongs to $\mathcal B_{\varphi,p}$ if and only if  
    \[
    \varphi(R_0)R_0^{\alpha/p} \lesssim \varphi(R_1)R_1^{\alpha/p},\qquad \text{for } R_0\le R_1.
    \]
    In particular, for $\varphi(t)=t^{d/u}$ with $p\le u$, we obtain
    \[
    |x|^\alpha\in \mathcal B_{u,p} \quad\text{if and only if }\quad \alpha\ge -\frac{dp}{u}.
    \]
    We refer to \cite[Example~2.3]{Nak16} or \cite[p.~158]{SDH20b} for the proofs.
\end{example}

The formulation for $\|\chi_{Q} \mid \mathfrak M_\varphi^p(\mathbb R^d;w)\|$ has been established in \cite[Lemma~2.1]{Nak16}.

\begin{lemma}\label{lem:1Q1w}
    Let $0<p<\infty$ and $\varphi\in\mathcal G_p$. For all $Q\in\mathcal Q$, we have
    \[
    \varphi(\ell(Q))\Bigl(\frac{w(Q)}{|Q|}\Bigr)^{1/p} \leq \|\chi_{Q}\mid \mathfrak M_\varphi^p(\mathbb R^d;w)\| 
    \lesssim \sup_{E\in\mathcal Q,\, E\subset Q}\varphi(\ell(E))\Bigl(\frac{w(E)}{|E|}\Bigr)^{1/p}.
    \]
    If, additionally, \eqref{eq:def-Bvp} holds, then
    \begin{equation}\label{eq:norm-1Q1w}
        \|\chi_{Q}\mid \mathfrak M_\varphi^p(\mathbb R^d;w)\|\sim\varphi(\ell(Q))\Bigl(\frac{w(Q)}{|Q|}\Bigr)^{1/p}.
    \end{equation}
\end{lemma}

\begin{remark}\label{rem:M1w-chiQ}
    Set $\varphi(t)= t^{d/u}$ with $p\le u$ yields
    \begin{equation}\label{eq:norm-1Q1w'}
        \|\chi_Q\mid \mathfrak M_u^p(\mathbb R^d;w)\| \sim |Q|^{1/u-1/p}w(Q)^{1/p}
    \end{equation}
    if \eqref{eq:def-Bup} holds; see \cite[p.~158]{SDH20b} for a detailed proof.
\end{remark}

\begin{remark}\label{rem:Mfvw-transform}
It is well known that, for $0<p<\infty$, $\|f\mid L^p(\mathbb R^d;w)\| = \|fw^{1/p}\mid L^p(\mathbb R^d)\|$. 
The two generalized weighted Morrey scales behave differently under
the transformation $f\mapsto fw^{1/p}$. For the Samko--type scale,
\[
\|f\mid\mathfrak M_\varphi^p(\mathbb R^d;w)\|
=
\|fw^{1/p}\mid M_\varphi^p(\mathbb R^d)\|.
\]

By contrast,
\[
\|f\mid\mathcal M_\varphi^p(\mathbb R^d;w)\|
=
\sup_{Q\in\mathcal Q}
\varphi(Q)w(Q)^{-1/p}
\Bigl(\int_Q|fw^{1/p}|^p\Bigr)^{1/p},
\]
and the factor $w(Q)^{-1/p}$ cannot in general be replaced by
$|Q|^{-1/p}$. Thus multiplication by $w^{1/p}$ does not, in general,
identify the Komori--Shirai-type scale with its unweighted
counterpart.
Moreover, the two norms are comparable if
\[
0<\inf_{Q\in\mathcal Q}\frac{w(Q)}{|Q|}
\le
\sup_{Q\in\mathcal Q}\frac{w(Q)}{|Q|}
<\infty.
\]
By the Lebesgue differentiation theorem, this condition is essentially
equivalent to $w\sim1$ a.e.. It therefore fails for many
nontrivial weights.

In the classical case $\varphi(t)=t^{d/u}$, $0<p\le u<\infty$, one has
$\|f\mid\mathfrak M_u^p(\mathbb R^d;w)\|
=
\|fw^{1/p}\mid M_u^p(\mathbb R^d)\|$. 
On the other hand, the Komori--Shirai-type space
$\mathcal M_u^p(\mathbb R^d;w)$ corresponds to
$\varphi(Q)=w(Q)^{1/u}$ and cannot generally be obtained from the
unweighted Morrey space merely by multiplication by $w^{1/p}$; see
\cite[Remark~2.12]{HS17}.
\end{remark}

\section{Embeddings of generalized Morrey spaces on bounded domains}
\label{sec:Mf-Omega}

Let $\Omega\subset\mathbb R^d$ be a bounded domain. In this section,
we study generalized Morrey spaces on $\Omega$ defined through
averages over cubes $Q(x,r)\subset\Omega$, where the Morrey function
$\varphi(x,r)$ may depend on both the center $x$ and the scale $r$;
see Definition~\ref{def:Mf}. We first study embeddings of the form
\[
M_{\varphi_1}^{p_1}(\Omega)
\hookrightarrow
M_{\varphi_2}^{p_2}(\Omega).
\]
We treat both the non-radial case
$\varphi_i=\varphi_i(x,r)$ and the radial case
$\varphi_i=\varphi_i(r)$, $i=1,2$. In the final part of the section,
we restrict attention to radial Morrey functions. For bounded
Lipschitz domains, we compare the definition based on cubes contained
in $\Omega$ with the alternative definition allowing cubes that
intersect the boundary. We also characterize the resulting domain
spaces as restrictions of generalized Morrey spaces on
$\mathbb R^d$.

Let $0<p\leq u<\infty$, the Morrey space $M_u^p(\Omega)$
consists of all measurable functions $f$ on $\Omega$ such that
\[
\|f\mid M_u^p(\Omega)\|
:=
\sup_{Q\subset\Omega}
|Q|^{1/u-1/p}
\Bigl(\int_Q |f(y)|^p\,dy\Bigr)^{1/p}
<\infty.
\]
This space can be traced back to Piccinini
\cite{Pic69a,Pic69b}. Recall that, for $0<p_i\le u_i<\infty$, $i=1,2$,
\begin{equation}\label{eb:Ma-Mb-Omega}
M^{p_1}_{u_1}(\Omega)
\hookrightarrow
M^{p_2}_{u_2}(\Omega)
\qquad\text{if and only if}\qquad
p_2\le p_1
\quad\text{and}\quad
u_2\le u_1,
\end{equation}
cf.~\cite{Pic69a,Pic69b,HSS18}.

As in Section~\ref{ssec:Mf}, we generalize the power function
$r^{d/u}$ in the above Morrey quasi-norm and introduce the following spaces.

\begin{definition}\label{def:morrey-Omega}
Let $0<p<\infty$ and $ \varphi:\Omega\times(0,\operatorname{diam}(\Omega)]
    \to(0,\infty)$. 
    The \emph{generalized Morrey space $M_\varphi^p(\Omega)$} is defined
    to be the set of all measurable functions $f$ on $\Omega$ such that
    \[
    \|f\mid M_\varphi^p(\Omega)\|
    :=
    \sup_{Q(x,r)\subset\Omega}
    \varphi(x,r)
    \Bigl(
        \frac{1}{|Q(x,r)|}
        \int_{Q(x,r)}|f(y)|^p\,dy
    \Bigr)^{1/p}
    <\infty.
    \]
    If $\varphi$ depends only on $r$, we simply write $\varphi(r)$.
\end{definition}

Note that when $\varphi(x,r)= r^{d/u}$ or $\varphi(r)=r^{d/u}$, where $0<p\le u<\infty$, then
\[
M_\varphi^p(\Omega)= M_u^p(\Omega).
\]

\begin{remark}\label{rem:cube-ball-MO}

Unlike in the whole-space setting considered in Section~\ref{sec:pre}, replacing cubes by balls in Definition~\ref{def:morrey-Omega} need not yield an equivalent quasi-norm on a bounded domain. 
Indeed, the usual comparison relies on enclosing a ball or cube in a set of the other type having comparable size, but the enclosing set need not remain inside $\Omega$.

Equivalence does hold under suitable additional assumptions. 
In particular, in the radial case $\varphi=\varphi(r)$, condition \eqref{eq:Hardy-intrinsic} ensures that balls and cubes yield equivalent quasi-norms on bounded Lipschitz domains; see Corollary~\ref{cor:interior-balls-cubes}. 
\end{remark}

To prove the necessity part of our embedding result, we need some control over the dependence of $\varphi(x,r)$ on the center variable $x$. 
We therefore introduce the following local comparability conditions.
Throughout this section, $A\Subset B$ means that $\overline A$ is compact and $\overline A\subset B$.

\begin{definition}\label{def:L(R)C}
Let $U\subset\mathbb R^d$ be open and 
$\varphi:U\times(0,\operatorname{diam}(\Omega)]\to(0,\infty)$.

\begin{enumerate}[\bfseries\upshape (i)]
    \item We say that $\varphi$ is \emph{locally comparable} on $U$ if for all $A\ge1$ there exists a constant $C_A>0$ such that
    \begin{equation}\label{eq:def-LC}
        C_A^{-1}\varphi(y,r)
        \le \varphi(x,r)
        \le C_A\varphi(y,r)
    \end{equation}
    for all $x,y\in U$, $0<r\le \operatorname{diam}(\Omega)$, and $|x-y|\le Ar$.

    \item Let $0<p<\infty$. We say that $\varphi$ is
    \emph{locally radially comparable with respect to $p$} on a cube
    $Q_\ast=Q(x_\ast,R_\ast)\Subset U$ if there exist a function
    $\psi\in\mathcal G_p$ and a constant $C\ge 1$ such that
    \begin{equation}\label{eq:def-LRC}
        C^{-1}\psi(r) \le \varphi(x,r)\le C\,\psi(r),
    \end{equation}
    for all $x\in Q_\ast$ and $0<r\le R_\ast$.
\end{enumerate}
\end{definition}

\subsection{Embeddings in the non-radial setting}\label{ssec:NRc}

We start with analyzing the norm $\|\chi_R\mid M_\varphi^p(\Omega)\|$.

\begin{lemma}\label{lem:1Q-Omega}
Let $0<p<\infty$ and $\varphi\in\mathcal G_p$.
Assume that $\varphi$ satisfies \eqref{eq:def-LC} on $\Omega$.
Then, 
\[
\|\chi_R\mid M_\varphi^p(\Omega)\|
\sim
\varphi(c(R),\ell(R)),
\]
for all cubes $R\subset\Omega$.
\end{lemma}

\begin{proof}
It is easily seen that 
\[
\|\chi_R\mid M_\varphi^p(\Omega)\|
\ge \varphi(c(R),\ell(R)).
\]

For the converse estimate, let $Q\subset\Omega$ be a cube. If
$Q\cap R=\emptyset$, the corresponding expression vanishes.
Suppose therefore that $Q\cap R\ne\emptyset$.

If $\ell(Q)\le\ell(R)$, then
\[
\varphi(c(Q),\ell(Q))
\lesssim\varphi(c(Q),\ell(R))
\lesssim\varphi(c(R),\ell(R)),
\]
where the first inequality follows from $\varphi\in\mathcal G_p$ is almost-increasing 
and the second from \eqref{eq:def-LC}. Consequently,
\[
\varphi(c(Q),\ell(Q))
\Bigl(\frac{|Q\cap R|}{|Q|}\Bigr)^{1/p}
\lesssim\varphi(c(R),\ell(R)).
\]

If $\ell(Q)>\ell(R)$, then \eqref{eq:def-LC} and
$Q\cap R\ne\emptyset$ give $\varphi(c(Q),\ell(Q))\lesssim\varphi(c(R),\ell(Q))$.
Moreover, since $r^{-d/p}\varphi(c(R),r)$ is almost decreasing,
\[
\varphi(c(R),\ell(Q))
({\ell(R)}/{\ell(Q)})^{d/p}
\lesssim\varphi(c(R),\ell(R)).
\]
Using $|Q\cap R|\le |R|$, we obtain
\[
\varphi(c(Q),\ell(Q))
\Bigl(\frac{|Q\cap R|}{|Q|}\Bigr)^{1/p}
\lesssim\varphi(c(R),\ell(R)).
\]
Taking the supremum over $Q\subset\Omega$ completes the proof.
\end{proof}

\begin{remark}\label{rem:1Q-rad}
In the radial case $\varphi=\varphi(r)$, condition \eqref{eq:def-LC} holds automatically. Hence, if $0<p<\infty$ and $\varphi=\varphi(r)\in\mathcal G_p$, then
\[
\|\chi_R\mid M_\varphi^p(\Omega)\|\sim \varphi(\ell(R))
\]
for all cubes $R\subset\Omega$. 
\end{remark}

We now present the main results of this section.

\begin{theorem}\label{thm:suff-Omega}
Let $\Omega\subset\mathbb R^d$ be a bounded domain,
$0<p_i<\infty$, and $\varphi_i:
\Omega\times(0,\operatorname{diam}(\Omega)]
\to(0,\infty)$, $i=1,2$.
If
\[
p_2\le p_1\qquad\text{and}\qquad
\sup_{Q(x,r)\subset\Omega}
\frac{\varphi_2(x,r)}{\varphi_1(x,r)}
<\infty,
\]
then
\[
M^{p_1}_{\varphi_1}(\Omega)
\hookrightarrow
M^{p_2}_{\varphi_2}(\Omega).
\]
\end{theorem}

\begin{proof}
This directly follows from H\"older's inequality. 
\end{proof}

\begin{theorem}\label{thm:nece-Omega}
Let $\Omega\subset\mathbb R^d$ be a bounded domain, 
$0<p_1,p_2<\infty$, and assume that
$\varphi_i\in\mathcal G_{p_i}$, $i=1,2$, both satisfying 
\eqref{eq:def-LC} on $\Omega$. 

Assume in addition that there exists a cube $Q_\ast=Q(x_\ast,R_\ast)\Subset\Omega$
such that $\varphi_1$ is locally radially comparable with respect to
$p_1$ on $Q_\ast$, with a function $\psi\in\mathcal G_1$ satisfying 
\begin{equation}\label{eq:LRC-norad}
\lim_{r\to0^+}\psi(r)=0.
\end{equation}

If the embedding 
\begin{equation}\label{eb:Mfa-Mfb-Omega}
    M^{p_1}_{\varphi_1}(\Omega)
\hookrightarrow
M^{p_2}_{\varphi_2}(\Omega)
\end{equation}
holds, then 
\[
p_2\le p_1
\qquad\text{and}\qquad
\sup_{
Q(x,r)\subset\Omega}
\frac{\varphi_2(x,r)}{\varphi_1(x,r)}
<\infty.
\]
\end{theorem}

\begin{proof}
\emph{Step }1. Let $Q(x,r)\subset\Omega$. Applying the embedding \eqref{eb:Mfa-Mfb-Omega} to the characteristic function $\chi_{Q(x,r)}$ and using Lemma~\ref{lem:1Q-Omega}, we obtain
\[
\varphi_2(x,r)\sim \|\chi_{Q(x,r)}\mid M^{p_2}_{\varphi_2}(\Omega)\|
\lesssim \|\chi_{Q(x,r)}\mid M^{p_1}_{\varphi_1}(\Omega)\|\sim \varphi_1(x,r).
\]

\emph{Step }2. 
Suppose, to the contrary, that $p_2>p_1$. 
We shall construct a family
$(g_r)_{r>0}$ such that $\|g_r\mid M^{p_1}_{\varphi_1}(\Omega)\|$ is uniformly bounded for all sufficiently small $r>0$, but 
$\|g_r\mid M^{p_2}_{\varphi_2}(\Omega)\|\to\infty$ as $r\to 0^+$.

 \emph{Step }2.1. Choose $Q'=Q(x_\ast,R_\ast/2)\Subset Q_\ast$
and set
\[
s_r:=c_0r\psi(r)^{-p_1/d}.
\]
Since $\rho^{-d/p_1}\psi(\rho)$ is almost decreasing, $r\psi(r)^{-p_1/d}
\lesssim R_\ast\psi(R_\ast)^{-p_1/d}$ for $0<r\le R_\ast$.
Hence $c_0>0$ can be chosen so that $s_r\le R_\ast/4$.
Moreover,
\[
\frac{s_r}{r}=c_0\psi(r)^{-p_1/d}\longrightarrow\infty\qquad\text{as $r\to0^+$}.
\]

Let
\[
J_r:=\{j\in\mathbb Z^d:x_\ast+s_rj\in Q'\},
\qquad
x_j:=x_\ast+s_rj.
\]
Then the points $x_j$ are $s_r$-separated in the
$\ell^\infty$-metric and
\[
\#J_r\sim\Bigl(\frac{R_\ast}{s_r}\Bigr)^d
\sim R_\ast^dr^{-d}\psi(r)^{p_1}.
\]
For all sufficiently small $r$, we have $r\le s_r/4$, and the
cubes $Q(x_j,r)$ are pairwise disjoint and compactly contained in
$Q_\ast$.
 
 We now define the test function
 \[
 g_r(x):=\sum_{j\in J_r}\psi(r)^{-1}\chi_{Q(x_j,r)}(x).
 \]

\emph{Step }2.2. We prove that $g_r$ is uniformly bounded in
$M^{p_1}_{\varphi_1}(\Omega)$ for all sufficiently small $r>0$. Let $Q=Q(x,t)\subset\Omega$ be arbitrary and set
\[
I(Q):=
\varphi_1(x,t)
\Bigl(\frac1{|Q|}\int_Q |g_r(y)|^{p_1}\,dy\Bigr)^{1/p_1}.
\]
If $Q\cap\operatorname{supp}g_r=\emptyset$, then $I(Q)=0$. Hence we assume that
$Q\cap\operatorname{supp}g_r\neq\emptyset$ and  consider four cases.

Case (i): $0<t\le2r$.
Since $|g_r|\le\psi(r)^{-1}$, we have
\[
\Bigl(\frac1{|Q|}\int_Q |g_r|^{p_1}\Bigr)^{1/p_1}
\le \psi(r)^{-1}.
\]
If $Q(x,t)\cap Q(x_j,r)\ne\emptyset$, then $|x-x_j|_\infty\le \frac{t+r}{2}\le \frac{3r}{2}$.
Since $x_j\in Q'\Subset Q_\ast$ and $r\le R_\ast/16$, it follows
that $x\in Q_\ast$.
Hence \eqref{eq:def-LC} and \eqref{eq:def-LRC} are applicable. 
Now using the fact that $|x-x_j|\lesssim r$, $\varphi_1\in\mathcal G_{p_1}$, \eqref{eq:def-LC}, \eqref{eq:def-LRC} and $\psi\in\mathcal G_{p_1}$, we obtain
\[
\varphi_1(x,t)
\lesssim \varphi_1(x,2r)
\lesssim \varphi_1(x_j,2r)
\lesssim \psi(2r)
\lesssim \psi(r).
\]
Therefore, 
\[
I(Q)\lesssim \psi(r)\psi(r)^{-1}\lesssim1.
\]

Case (ii): $2r<t<s_r$. 
The cube $Q(x,t)$ intersects at most $C_d$ of the cubes $Q(x_j,r)$. 
Indeed, assume that 
$Q(x,t)\cap Q(x_j,r)\neq\emptyset$, 
then $x_j\in Q(x,t+2r)\subset Q(x,2t) \subset Q(x,2s_r)$. 
Since the centers $x_j$ are $s_r$-separated, only a bounded number of them can lie in a cube of side length $2s_r$. 
More precisely, by placing disjoint cubes $Q(x_j,s_r/2)$ around these centers, all these  cubes are contained in a cube of side length comparable to $s_r$. 
Therefore the number of such centers is bounded by a constant $C_d$, depending only on the
dimension. Consequently, 
\[
\int_Q |g_r(y)|^{p_1}\,dy
\lesssim r^d\psi(r)^{-p_1}.
\]
Moreover, if $Q$ intersects $Q(x_j,r)$, then $|x-x_j|\lesssim t$. Thus, by \eqref{eq:def-LC} and \eqref{eq:def-LRC}, 
\[
\varphi_1(x,t)\lesssim \varphi_1(x_j,t)\lesssim \psi(t),
\]
since $t<s_r\le R_\ast/4<R_\ast$. Therefore
\[
I(Q)
\lesssim
\psi(t)
\Bigl(\frac{r^d\psi(r)^{-p_1}}{t^d}\Bigr)^{1/p_1}
=
\frac{\psi(t)t^{-d/p_1}}{\psi(r)r^{-d/p_1}}.
\]
Since $t>r$ and $\rho^{-d/p_1}\psi(\rho)$ is almost decreasing, 
it follows that $I(Q)\lesssim1$.

Case (iii): $s_r\le t\le R_\ast$.
Similarly to Case (ii), we claim that the cube $Q(x,t)$ intersects at most $C(t/s_r)^d$ of the cubes $Q(x_j,r)$. 
Indeed, suppose that $Q(x,t)\cap Q(x_j,r)\neq\emptyset$, then $x_j\in Q(x,t+2r) \subset Q(x,3t)$. 
Let $N$ be the number of such centers $x_j$. Since the points $x_j$ are
$s_r$-separated, the cubes $Q(x_j,s_r/2)$ 
are pairwise disjoint, and are contained in
$Q(x,3t+s_r)$. Therefore, $N s_r^d
    \le c |Q(x,3t+s_r)|
    \le c(t+s_r)^d$.
Since $t\ge s_r$, we have $N\le c(1+{t}/{s_r})^d \le C({t}/{s_r})^d$.
Consequently, 
\[
\int_Q |g_r(y)|^{p_1}\,dy
\lesssim
\Bigl(\frac{t}{s_r}\Bigr)^d r^d\psi(r)^{-p_1}.
\]
Since $s_r^d\sim r^d\psi(r)^{-p_1}$, we obtain
\[
\Bigl(\frac1{|Q|}\int_Q |g_r|^{p_1}\Bigr)^{1/p_1}\lesssim1.
\]
Choose $j$ such that $Q(x,t)\cap Q(x_j,r)\ne\emptyset$. Then $|x-x_j|\lesssim t$, and hence
\[
\varphi_1(x,t)\lesssim \varphi_1(x_j,t)\lesssim\psi(t)\lesssim\psi(R_\ast).
\]
Thus $I(Q)\lesssim1$.

Case (iv): $t>R_\ast$.
Since $g_r$ is supported in $Q_\ast$, we have
\[
\Bigl(\frac1{|Q|}\int_Q |g_r|^{p_1}\Bigr)^{1/p_1}
\le t^{-d/p_1}\Bigl(\int_\Omega |g_r|^{p_1}\Bigr)^{1/p_1}=t^{-d/p_1}\bigl(\# J_r r^d \psi(r)^{-p_1}\bigr)^{1/p_1}\sim
t^{-d/p_1}R_\ast^{d/p_1}.
\]
Again, since $Q$ intersects the support of $g_r$, there exists $x_j\in Q_\ast$
such that $|x-x_j|\lesssim t$. By \eqref{eq:def-LC}, the almost decreasing property of
$\rho\mapsto \rho^{-d/p_1}\varphi_1(x_j,\rho)$, and
\eqref{eq:def-LRC}, we have
\[
\begin{aligned}
t^{-d/p_1}\varphi_1(x,t)
\lesssim t^{-d/p_1}\varphi_1(x_j,t)
\lesssim R_\ast^{-d/p_1}\varphi_1(x_j,R_\ast)
\lesssim R_\ast^{-d/p_1}\psi(R_\ast).
\end{aligned}
\]
It follows that
\[
I(Q)
\lesssim
R_\ast^{d/p_1}t^{-d/p_1}\varphi_1(x,t)
\lesssim
\psi(R_\ast).
\]
Hence we obtain $I(Q)\lesssim 1$.

Combining the four cases, we obtain
\[
\|g_r\mid M^{p_1}_{\varphi_1}(\Omega)\|\le C
\]
with $C$ independent of $r$.

\emph{Step }2.3. It remains to estimate the target norm. Testing on the fixed cube $Q_\ast$, we have
\[
\|g_r\mid M^{p_2}_{\varphi_2}(\Omega)\|
\ge
\varphi_2(x_\ast,R_\ast)
\Bigl(
\frac1{|Q_\ast|}\int_{Q_\ast}|g_r(y)|^{p_2}\,dy
\Bigr)^{1/p_2}.
\]
Note that the cubes $Q(x_j,r)$ are pairwise disjoint,
\[
\int_{Q_\ast}|g_r(y)|^{p_2}\,dy
=
\#J_r\,r^d\psi(r)^{-p_2}.
\]
Thus, by $\#J_r\sim R_\ast^d r^{-d}\psi(r)^{p_1}$,
we obtain
\[
\int_{Q_\ast}|g_r(y)|^{p_2}\,dy
\sim
R_\ast^d\psi(r)^{p_1-p_2}.
\]
Since $p_2>p_1$ and \eqref{eq:LRC-norad}, we have  
\[
\|g_r\mid M^{p_2}_{\varphi_2}(\Omega)\|\gtrsim \varphi_2(x_\ast,R_\ast)\psi(r)^{p_1/p_2-1}\to\infty.
\]
This contradicts the continuous embedding \eqref{eb:Mfa-Mfb-Omega} and completes
the proof.
\end{proof}

\begin{remark}\label{rem:vanishing-psi}
The condition \eqref{eq:LRC-norad} is needed only to obtain
the necessary of $p_2\leq p_1$, but it cannot in general be
omitted. Indeed, if $\varphi_1=\varphi_2\equiv1$, then
$M_1^p(\Omega)=L^\infty(\Omega)$ for $0<p<\infty$. 
Thus the embedding $M_1^{p_1}(\Omega)\hookrightarrow M_1^{p_2}(\Omega)$ holds 
for arbitrary $p_1,p_2>0$. 
Consequently, without the vanishing condition, the
embedding need not imply $p_2\leq p_1$. In the proof of
Theorem~\ref{thm:nece-Omega}, this condition enters precisely in the
block-packing argument used to exclude $p_2>p_1$.
\end{remark}

We summarize the above two main results at the end of this subsection.

\begin{corollary}\label{cor:main-Omega}
Let $\Omega\subset\mathbb R^d$ be a bounded domain, 
$0<p_1,p_2<\infty$, and assume that
$\varphi_i\in\mathcal G_{p_i}$, $i=1,2$, both satisfying 
\eqref{eq:def-LC} on $\Omega$. 
Assume in addition that there exists a cube $Q_\ast=Q(x_\ast,R_\ast)\Subset\Omega$
such that $\varphi_1$ is locally radially comparable with respect to
$p_1$ on $Q_\ast$, with a function $\psi\in\mathcal G_1$ satisfying $\lim_{r\to0^+}\psi(r)=0$. 
Then 
\[
M_{\varphi_1}^{p_1}(\Omega)
\hookrightarrow
M_{\varphi_2}^{p_2}(\Omega)
\]
if and only if
\[
p_2\leq p_1
\qquad\text{and}\qquad
\sup_{Q(x,r)\subset\Omega}
\frac{\varphi_2(x,r)}{\varphi_1(x,r)}
<\infty.
\]
\end{corollary}

\begin{proof}
The sufficiency follows from Theorem~\ref{thm:suff-Omega}, while the
necessity follows from Theorem~\ref{thm:nece-Omega}.
\end{proof}

\subsection{The radial case and two definitions on bounded domains}\label{ssec:Rc}

We now consider the radial case $\varphi=\varphi(r)$ throughout this subsection.

\begin{proposition}\label{pro:rad-Omega}
Let $\Omega\subset\mathbb R^d$ be a bounded domain,
$0<p_i<\infty$, and
$\varphi_i\in\mathcal G_{p_i}$, $i=1,2$.
Assume that
\begin{equation}\label{eq:LRC-rad}
\lim_{r\to0^+}\varphi_1(r)=0.
\end{equation}
Then
\[
M^{p_1}_{\varphi_1}(\Omega)
\hookrightarrow
M^{p_2}_{\varphi_2}(\Omega)
\]
if and only if
\[
p_2\le p_1\qquad\text{and}\qquad
\sup_{0<r\le\operatorname{diam}(\Omega)}
\frac{\varphi_2(r)}{\varphi_1(r)}
<\infty.
\]
\end{proposition}

\begin{proof}
In the radial case, local comparability is automatic. 
Moreover, on all cubes
$Q_\ast\Subset\Omega$, the function $\varphi_1(x,r)=\varphi_1(r)$
is locally radially comparable with comparison function
$\psi=\varphi_1$. 
Therefore, all the hypotheses of Corollary~\ref{cor:main-Omega} are satisfied, and the conclusion follows.
\end{proof}

\begin{remark}\label{rem:compare-HMS}
In \cite[Corollary~5.13]{HMS26}, the authors study embeddings between
generalized Morrey spaces on bounded $C^\infty$ domains in the radial
setting. More precisely, for
$\varphi_i=\varphi_i(r)\in\mathcal G_{p_i}$, $i=1,2$, with
$\lim_{r\to0^+}\varphi_1(r)=0$, they show that
\[
p_2\le p_1
\qquad\text{and}\qquad
\varphi_2(r)\lesssim\varphi_1(r),\quad 0<r\le1,
\]
are sufficient for
\[
M^{p_1}_{\varphi_1}(\Omega)
\hookrightarrow
M^{p_2}_{\varphi_2}(\Omega).
\]
These conditions are also necessary when $1<p_1,p_2<\infty$.
As noted in \cite[Remark~5.14]{HMS26}, the authors conjecture that the same
characterization remains valid for $0<p_1,p_2\le1$, although their argument
does not apply in this range.

Proposition~\ref{pro:rad-Omega} proves the necessity of
$p_2\leq p_1$ for the full range $0<p_1,p_2<\infty$. In particular,
it settles the case $0<p_1,p_2\leq1$ left open in \cite{HMS26}. The
proof uses the geometric block-packing argument developed in
Theorem~\ref{thm:nece-Omega} and avoids duality and other methods
restricted to the Banach range.
Moreover, Corollary~\ref{cor:main-Omega} shows the embedding characterization also holds in the non-radial setting, allowing $\varphi(x,r)$ to depend on both the location and the scale of a cube $Q(x,r)$, provided that \eqref{eq:def-LC}--\eqref{eq:LRC-norad} hold.
\end{remark}

Following \cite{HSS18}, we compare two definitions of generalized Morrey spaces on a bounded domain. 
The space $M_\varphi^p(\Omega)$, introduced in Definition~\ref{def:morrey-Omega}(ii), is defined using cubes contained in $\Omega$. 
It therefore controls local averages over interior cubes, but does not directly account for averaging sets that meet the boundary $\partial\Omega$. 
This motivates the following alternative definition, in which the averaging sets are allowed to intersect $\partial\Omega$. 
We first determine conditions under which these two spaces coincide up to equivalence of quasi-norms, and then characterize their restrictions in terms of generalized Morrey spaces on $\mathbb R^d$.

\begin{definition}\label{def:Morrey-Omega-cap}
Let $\Omega\subset\mathbb R^d$ be a bounded domain, $0<p<\infty$, and
$\varphi:(0,\infty)\to(0,\infty)$ be measurable. The \emph{generalized
Morrey space} $\mathcal M_\varphi^p(\Omega)$ is defined to be the set of all
measurable functions $f$ on $\Omega$ such that
\[
\|f\mid\mathcal M_\varphi^p(\Omega)\|
:=
\sup_{\substack{x\in\Omega\\0<r\le \operatorname{diam}(\Omega)}}
\varphi(r)
\Bigl(
\frac{1}{|\Omega\cap Q(x,r)|}
\int_{\Omega\cap Q(x,r)}|f(y)|^p\,dy
\Bigr)^{1/p}
<\infty.
\]
\end{definition}

\begin{remark}\label{rem:cubes-balls-Omega}
Here we adopt the notation
$\mathcal M$ from \cite{HSS18} to compare with the classical cases; see Remark~\ref{rem:intrinsic-classical}. 
This notation is used only within the present subsection and should not be confused with
the notation for the Komori--Shirai-type weighted Morrey spaces considered
elsewhere in this paper.

If $\Omega$ is a bounded Lipschitz domain and
$\varphi\in\mathcal G_p$, then the quasi-norm in
Definition~\ref{def:Morrey-Omega-cap} may equivalently be defined using
balls $B(x,r)$ in place of cubes $Q(x,r)$. Indeed, the measure
density property of bounded Lipschitz domains gives
\[
|\Omega\cap B(x,r)|\sim r^d
\quad\text{and}\quad
|\Omega\cap Q(x,r)|\sim r^d
\]
uniformly for $x\in\Omega$ and $0<r\le\operatorname{diam}(\Omega)$;
see \cite{HKT08}. The equivalence follows from these estimates, the standard inclusions
between concentric cubes and balls, and the fact that
$\varphi(cr)\sim\varphi(r)$ for all fixed $c\geq1$, which follows
from $\varphi\in\mathcal G_p$.
\end{remark}

We first establish the immediate embedding.

\begin{proposition}\label{prop:cap-nor}
Let $\Omega\subset\mathbb R^d$ be a bounded domain, $0<p<\infty$ and $\varphi:(0,\infty)\to(0,\infty)$ be measurable.
Then
\[
\mathcal M_\varphi^p(\Omega)
\hookrightarrow
M_\varphi^p(\Omega).
\]
\end{proposition}

\begin{proof}
For all cubes $Q(x,r)\subset\Omega$, we have
$\Omega\cap Q(x,r)=Q(x,r)$. Hence
\[
\begin{aligned}
&\varphi(r)
\Bigl(
\frac{1}{|Q(x,r)|}
\int_{Q(x,r)}|f(y)|^p\,dy
\Bigr)^{1/p}
\\
&\qquad=
\varphi(r)
\Bigl(
\frac{1}{|\Omega\cap Q(x,r)|}
\int_{\Omega\cap Q(x,r)}|f(y)|^p\,dy
\Bigr)^{1/p}
\le
\|f\mid\mathcal M_\varphi^p(\Omega)\|.
\end{aligned}
\]
Taking the supremum over all cubes $Q(x,r)\subset\Omega$ proves the
assertion.
\end{proof}

To show the converse embedding, we need two auxiliary lemmas.

\begin{lemma}\label{lem:interior-ball-estimate}
Let $\Omega\subset\mathbb R^d$ be a bounded domain, $0<p<\infty$, and
$\varphi\in\mathcal G_p$. Let $\lambda>1$ be fixed. Then, for all 
balls $B(x,r)$ satisfying $B(x,\lambda r)\subset\Omega$ and all $f\in M_\varphi^p(\Omega)$, 
\[
\int_{B(x,r)}|f(y)|^p\,dy
\lesssim
r^d\varphi(r)^{-p}
\|f\mid M_\varphi^p(\Omega)\|^p.
\]
The implicit constant depends only on $d$, $p$, $\lambda$, and the
constants in the definition of $\mathcal G_p$.
\end{lemma}

\begin{proof}
Fix $\lambda>1$ and set
\[
c:=\frac12\min\Bigl\{1,\frac{\lambda-1}{\sqrt d}\Bigr\}.
\]
Consider the standard grid of cubes of side length $cr$, and let
$Q_1,\ldots,Q_N$ be those grid cubes that intersect $B(x,r)$.
These cubes have pairwise disjoint interiors, and $B(x,r)\subset\bigcup_{k=1}^N Q_k$.
For $z\in Q_k$, choose $v\in Q_k\cap B(x,r)$. Then $|z-x|
\le |z-v|+|v-x|
<\sqrt d\,cr+r
<\lambda r$.
Hence
\[
Q_k\subset B(x,\lambda r)\subset\Omega,\qquad k=1,\ldots,N.
\]
Using the disjointness of their interiors, we obtain $N(cr)^d
=
|\bigcup_{k=1}^N Q_k|
\le |B(x,\lambda r)|$, and hence  
$N\le C_{d,\lambda}$.

Since all the cubes have side length $cr$, we have $\ell(Q_k)\sim r$.
By the definition of $M_\varphi^p(\Omega)$,
\[
\varphi(\ell(Q_k))
\Bigl(
\frac{1}{|Q_k|}
\int_{Q_k}|f(y)|^p\,dy
\Bigr)^{1/p}
\le
\|f\mid M_\varphi^p(\Omega)\|.
\]
Therefore, by $|Q_k|\sim r^d$ and $\varphi\in\mathcal G_p$,  
\[
\begin{aligned}
\int_{Q_k}|f(y)|^p\,dy
\le
|Q_k|\,\varphi(\ell(Q_k))^{-p}
\|f\mid M_\varphi^p(\Omega)\|^p
\lesssim
r^d\varphi(r)^{-p}
\|f\mid M_\varphi^p(\Omega)\|^p.
\end{aligned}
\]
Summing over the uniformly bounded number of cubes yields
\[
\begin{aligned}
\int_{B(x,r)}|f(y)|^p\,dy
\le
\sum_{k=1}^N\int_{Q_k}|f(y)|^p\,dy\lesssim
r^d\varphi(r)^{-p}
\|f\mid M_\varphi^p(\Omega)\|^p.
\end{aligned}
\]
This completes the proof.
\end{proof}

Inspired by the proof of \cite[Theorem~2.7]{HSS18}, we obtain the following lemma.

\begin{lemma}\label{lem:boundary-covering-Omega}
Let $\Omega\subset\mathbb R^d$ be a bounded Lipschitz domain and set $R_0:=\operatorname{diam}(\Omega)$.
For $J\in\mathbb N_0$, define
\[
S_J:=
\bigl\{
y\in\Omega:
2^{-J-1}R_0
<
\operatorname{dist}(y,\partial\Omega)
\le
2^{-J+1}R_0
\bigr\}.
\]
Then there exist constants $0<c_1<c_2$ and $C>0$, depending only on
$\Omega$ and $d$, such that, for all $x\in\Omega$, all
$j\in\mathbb N_0$, and all $J\ge j$, the set $S_J\cap B(x,2^{-j}R_0)$
can be covered by balls $B(y_\nu,c_1 2^{-J}R_0)$, $\nu=1,\ldots,N$,
where $y_\nu\in S_J\cap B(x,2^{-j}R_0)$, $N\le C2^{(J-j)(d-1)}$,
and $B(y_\nu,c_2 2^{-J}R_0)\subset\Omega$, $\nu=1,\ldots,N$.
\end{lemma}

\begin{proof}
\emph{Step }1. 
We first claim that, for fixed $a,b>0$, $x\in\mathbb R^d$ and $0<t\le R\le R_0$,
\begin{equation}\label{eq:tubular-neighbourhood-estimate}
\bigl|
\bigl\{
z\in\mathbb R^d:
\operatorname{dist}(z,\partial\Omega)<at
\bigr\}
\cap B(x,bR)
\bigr|
\le C R^{d-1}t.
\end{equation}
Here and below, $C$ may depend on $a$, $b$, $\Omega$ and $d$. 

Indeed, since $\Omega$ is a bounded Lipschitz domain,
$\partial\Omega$ is compact and can be covered by finitely many
coordinate patches. After a rigid change of coordinates in each
patch, the corresponding part of $\partial\Omega$ is contained in a
Lipschitz graph $\Gamma_k=\{(u,g_k(u)):u\in\mathbb R^{d-1}\}$, $k=1,\ldots,K$,
where the locally defined Lipschitz functions have been extended to
all of $\mathbb R^{d-1}$ without increasing their Lipschitz constants.
Thus,
\[
\partial\Omega\subset\bigcup_{k=1}^K\Gamma_k.
\]

It therefore suffices to prove that, for all Lipschitz graphs
$\Gamma=\{(u,g(u)):u\in\mathbb R^{d-1}\}$
with Lipschitz constant $L$ and all $s>0$,
\begin{equation}\label{eq:graph-tubular-estimate}
\bigl|
\bigl\{
z\in\mathbb R^d:
\operatorname{dist}(z,\Gamma)<s
\bigr\}
\cap B(x,bR)
\bigr|
\le C_{b,d}(1+L)R^{d-1}s.
\end{equation}
To prove this, let $z=(z',z_d)$. If
$\operatorname{dist}(z,\Gamma)<s$, then there exists
$u\in\mathbb R^{d-1}$ such that $|z'-u|<s$ and $|z_d-g(u)|<s$.
Consequently,
\[
\begin{aligned}
|z_d-g(z')|
\le |z_d-g(u)|+|g(u)-g(z')|<s+L|u-z'| <(1+L)s.
\end{aligned}
\]
On the other hand, if $z\in B(x,bR)$, then $z'$ belongs to a
$(d-1)$-dimensional ball of radius $bR$. Hence, by Fubini's theorem,
\[
\begin{aligned}
\bigl|
\bigl\{
z\in\mathbb R^d:
\operatorname{dist}(z,\Gamma)<s
\bigr\}
\cap B(x,bR)
\bigr| \le
\int_{B_{\mathbb R^{d-1}}(x',bR)}
2(1+L)s\,dz'
\le
C_{b,d}(1+L)R^{d-1}s,
\end{aligned}
\]
which proves \eqref{eq:graph-tubular-estimate}. 
Taking $s=at$, using the finite covering of $\partial\Omega$ by the
graphs $\Gamma_k$, and summing
\eqref{eq:graph-tubular-estimate} over $k$, we obtain
\eqref{eq:tubular-neighbourhood-estimate}.

\emph{Step }2. 
Let $x\in\Omega$, $j\in\mathbb N_0$, and $J\ge j$. Set $R:=2^{-j}R_0$, $t:=2^{-J}R_0$.
Then $0<t\le R$. Let 
\[
E:=S_J\cap B(x,R).
\]
If $E=\emptyset$, there is nothing to prove. So we assume that $E\neq\emptyset$.

Fix constants $0<c_1<c_2<1/2$. Choose a maximal
$c_1t$-separated family $\{y_1,\ldots,y_N\}\subset E$.
By maximality, $E\subset\bigcup_{\nu=1}^N B(y_\nu,c_1t)$.
Moreover, the balls $B(y_\nu,\frac{c_1t}{2})$, $\nu=1,\ldots,N$,
are pairwise disjoint. 
Since $y_\nu\in B(x,R)$ and $t\le R$, we have $B(y_\nu,\frac{c_1t}{2})\subset B(x,2R)$.
In addition, $y_\nu\in S_J$ implies $\operatorname{dist}(y_\nu,\partial\Omega)\le 2t$.
Therefore, if
$z\in B(y_\nu,\frac{c_1t}{2})$, then $\operatorname{dist}(z,\partial\Omega)
\le
\operatorname{dist}(y_\nu,\partial\Omega)+|z-y_\nu|
<
(2+\frac{c_1}{2})t
<3t$. 
Consequently,
\[
\bigcup_{\nu=1}^N
B\Bigl(y_\nu,\frac{c_1t}{2}\Bigr)
\subset
\{
z\in\mathbb R^d:
\operatorname{dist}(z,\partial\Omega)<3t
\}
\cap B(x,2R).
\]
Using the disjointness of these balls and
\eqref{eq:tubular-neighbourhood-estimate} with $a=3$ and $b=2$, we
obtain
\[
\begin{aligned}
N\Bigl|B\Bigl(0,\frac{c_1t}{2}\Bigr)\Bigr|
\le
\bigl|
\bigl\{
z\in\mathbb R^d:
\operatorname{dist}(z,\partial\Omega)<3t
\bigr\}
\cap B(x,2R)
\bigr| \lesssim R^{d-1}t.
\end{aligned}
\]
It follows that
\[
N
\lesssim
\frac{R^{d-1}t}{t^d}
=
\Bigl(\frac Rt\Bigr)^{d-1}
=
2^{(J-j)(d-1)}.
\]

Finally, $y_\nu\in S_J$ yields
$\operatorname{dist}(y_\nu,\partial\Omega)>t/2$. Since
$c_2<1/2$, $c_2t<\operatorname{dist}(y_\nu,\partial\Omega)$. Hence
$B(y_\nu,c_2t)\subset\Omega$.
Substituting $t=2^{-J}R_0$ then proves the lemma.
\end{proof}

\begin{proposition}\label{prop:nor-cap}
Let $\Omega\subset\mathbb R^d$ be a bounded Lipschitz domain,
$0<p<\infty$, $\varphi\in \mathcal G_p$, and set $R_0:=\operatorname{diam}(\Omega)$. 
Assume that
\begin{equation}\label{eq:Hardy-intrinsic}
\sup_{0<r\le R_0}
\frac{\varphi(r)^p}{r}
\sum_{\substack{J\in\mathbb N_0\\2^{-J}R_0\le r}}
\frac{2^{-J}R_0}
{\varphi(2^{-J}R_0)^p}
<\infty.
\end{equation}
Then
\[
M_\varphi^p(\Omega)
\hookrightarrow
\mathcal M_\varphi^p(\Omega).
\]
\end{proposition}

\begin{proof}
Let $f\in M_\varphi^p(\Omega)$.
By Remark~\ref{rem:cubes-balls-Omega}, it suffices to estimate
the quasi-norm using balls. 

\emph{Step }1. Let $x\in\Omega$ and $0<r\le R_0$. If
$B(x,2r)\subset\Omega$, then Lemma~\ref{lem:interior-ball-estimate} with $\lambda=2$ gives
\[
\varphi(r)
\Bigl(
\frac{1}{|\Omega\cap B(x,r)|}
\int_{\Omega\cap B(x,r)}|f(y)|^p\,dy
\Bigr)^{1/p}
\lesssim \|f\mid M_\varphi^p(\Omega)\|.
\]

\emph{Step }2. We now assume that $B(x,2r)\not\subset\Omega$. Choose $j\in\mathbb N_0$ such that
$2^{-j-1}R_0<r\le2^{-j}R_0$, and set
$r_j:=2^{-j}R_0$. Then $r\sim r_j$. Since
$\varphi\in \mathcal G_p$ is almost increasing, we have $\varphi(r)\lesssim\varphi(r_j)$, and
$B(x,r)\subset B(x,r_j)$. Moreover, the measure density property of the bounded Lipschitz domain gives 
$|\Omega\cap B(x,r)|\gtrsim r^d\sim r_j^d$, cf. \cite{HKT08}. Consequently,
\[
\begin{aligned}
\varphi(r)
\Bigl(
\frac{1}{|\Omega\cap B(x,r)|}
\int_{\Omega\cap B(x,r)}|f(y)|^p\,dy
\Bigr)^{1/p}\lesssim
\varphi(r_j)r_j^{-d/p}
\Bigl(
\int_{\Omega\cap B(x,r_j)}|f(y)|^p\,dy
\Bigr)^{1/p}.
\end{aligned}
\]
Since $B(x,2r_j)\not\subset\Omega$, we have
$\operatorname{dist}(x,\partial\Omega)<2r_j$. 
Hence all
$y\in B(x,r_j)$ satisfy
$\operatorname{dist}(y,\partial\Omega)<3r_j$. 
Therefore 
\[
\Omega\cap B(x,r_j)
\subset
\bigcup_{J\ge \max\{j-1, 0\}}
\Bigl(S_J\cap B(x,r_j)\Bigr),
\]

\emph{Step }3. We split the resulting sum into two parts: 
$\max\{j-1, 0\}\le J<j$ and $J\ge j$.

\emph{Step }3.1. For $J\ge j$, Lemma~\ref{lem:boundary-covering-Omega} shows that
$S_J\cap B(x,r_j)$ can be covered by at most
$C2^{(J-j)(d-1)}$ balls of radius $c_1 2^{-J}R_0$ contained in $\Omega$.
Each covering ball is contained in $\Omega$, and its concentric
dilation by the fixed factor $\lambda=c_2/c_1>1$ is also contained
in $\Omega$.
Hence Lemma~\ref{lem:interior-ball-estimate} applies and 
\[
\begin{aligned}
\int_{S_J\cap B(x,r_j)}|f(y)|^p\,dy
&\lesssim
2^{(J-j)(d-1)}
(2^{-J}R_0)^d
\varphi(2^{-J}R_0)^{-p}\|f\mid M_\varphi^p(\Omega)\|^p.
\end{aligned}
\]
Since
$2^{-J}R_0=2^{-(J-j)}r_j$, we have $2^{(J-j)(d-1)}(2^{-J}R_0)^d=r_j^{d-1}(2^{-J}R_0)$.
Thus
\begin{equation}\label{eq:boundary-layer-main}
\sum_{J\ge j}
\int_{S_J\cap B(x,r_j)}|f(y)|^p\,dy
\lesssim
\|f\mid M_\varphi^p(\Omega)\|^p r_j^{d-1}
\sum_{J\ge j}
\frac{2^{-J}R_0}
{\varphi(2^{-J}R_0)^p}.
\end{equation}

\emph{Step 3.2.}
We next consider the finitely many indices $\max\{0,j-1\}\le J<j$.
Since $J<j$, we have $r_j=2^{-j}R_0\le 2^{-J}R_0$,
and hence
\[
S_J\cap B(x,r_j)
\subset
S_J\cap B(x,2^{-J}R_0).
\]
Apply Lemma~\ref{lem:boundary-covering-Omega} with both scale
indices equal to $J$. We obtain a covering of
$S_J\cap B(x,2^{-J}R_0)$ by balls $B(z_{J,k},c_1 2^{-J}R_0)$, $1\le k\le N_J$, such that 
$N_J\lesssim 1$ and $B(z_{J,k},c_2 2^{-J}R_0)\subset\Omega$,
where $c_2>c_1>0$ are the constants from Lemma \ref{lem:boundary-covering-Omega}.
Thus Lemma~\ref{lem:interior-ball-estimate}, with the fixed
factor $\lambda=c_2/c_1>1$, yields
\[
\begin{aligned}
\int_{S_J\cap B(x,r_j)}|f(y)|^p\,dy
\le
\sum_{k=1}^{N_J}
\int_{B(z_{J,k},c_1 2^{-J}R_0)}
|f(y)|^p\,dy \lesssim
\|f\mid M_\varphi^p(\Omega)\|^p
\frac{(2^{-J}R_0)^d}
     {\varphi(2^{-J}R_0)^p}.
\end{aligned}
\]
Since $0<j-J\le 1$, $2^{-J}R_0\sim r_j$.
Therefore $(2^{-J}R_0)^d
\lesssim
r_j^{d-1}\,2^{-J}R_0$,
and consequently
\[
\sum_{\max\{j-1,0\}\le J <j}\int_{S_J\cap B(x,r_j)}|f(y)|^p\,dy
\lesssim
\|f\mid M_\varphi^p(\Omega)\|^p
r_j^{d-1}
\sum_{\max\{j-1,0\}\le J <j}\frac{2^{-J}R_0}{\varphi(2^{-J}R_0)^p}.
\]

Combining this estimate with \eqref{eq:boundary-layer-main},
we obtain
\[
\int_{\Omega\cap B(x,r_j)}|f(y)|^p\,dy
\lesssim
\|f\mid M_\varphi^p(\Omega)\|^p
r_j^{d-1}
\sum_{J\ge\max\{j-1,0\}}
\frac{2^{-J}R_0}{\varphi(2^{-J}R_0)^p}.
\]
For $\max\{j-1,0\}\le J<j$, we have $r_j\le 2^{-J}R_0\le 2^{1}r_j$.
The almost increasing property of $\varphi$ therefore gives
$\frac{2^{-J}R_0}{\varphi(2^{-J}R_0)^p}
\lesssim
\frac{r_j}{\varphi(r_j)^p}$.
Since there are at most one such indices and
$r_j/\varphi(r_j)^p$ is the term corresponding to $J=j$,
it follows that
\[
\sum_{J\ge\max\{j-1,0\}}
\frac{2^{-J}R_0}{\varphi(2^{-J}R_0)^p}
\lesssim
\sum_{J\ge j}
\frac{2^{-J}R_0}{\varphi(2^{-J}R_0)^p}.
\]
Consequently,
\[
\int_{\Omega\cap B(x,r_j)}|f(y)|^p\,dy
\lesssim
\|f\mid M_\varphi^p(\Omega)\|^p
r_j^{d-1}
\sum_{J\ge j}
\frac{2^{-J}R_0}{\varphi(2^{-J}R_0)^p}.
\]
Using the estimate from Step~2 and applying
\eqref{eq:Hardy-intrinsic} at $r=r_j$, we conclude that
\[
\begin{aligned}
\frac{\varphi(r)^p}{|\Omega\cap B(x,r)|}
\int_{\Omega\cap B(x,r)}|f(y)|^p\,dy
\lesssim
\|f\mid M_\varphi^p(\Omega)\|^p
\frac{\varphi(r_j)^p}{r_j}
\sum_{J\ge j}
\frac{2^{-J}R_0}{\varphi(2^{-J}R_0)^p}
\lesssim
\|f\mid M_\varphi^p(\Omega)\|^p.
\end{aligned}
\]
Together with Step~1, taking the supremum over
$x\in\Omega$ and $0<r\le R_0$ and using
Remark~\ref{rem:cubes-balls-Omega} yields
\[
\|f\mid\mathcal M_\varphi^p(\Omega)\|
\lesssim
\|f\mid M_\varphi^p(\Omega)\|.
\]
This completes the proof.
\end{proof}

Combining Propositions~\ref{prop:cap-nor} and~\ref{prop:nor-cap}
immediately gives the following equivalence.

\begin{corollary}\label{cor:cap=nor}
    Let $\Omega\subset\mathbb R^d$ be a bounded Lipschitz domain,
$0<p<\infty$ and $\varphi\in \mathcal G_p$. 
Assume that condition \eqref{eq:Hardy-intrinsic} holds. 
Then
\[
\mathcal M_\varphi^p(\Omega)
=
M_\varphi^p(\Omega).
\]
\end{corollary}

We next investigate the relation between the spaces $M_\varphi^p(\Omega)$ defined by cubes and balls. 
Define the ball-based quasi-norm by
\[
\|f\mid M_{\varphi,\mathrm{ball}}^p(\Omega)\|
:=
\sup_{B(x,r)\subset\Omega}
\varphi(r)
\Bigl(
r^{-d}\int_{B(x,r)}|f(y)|^p\,dy
\Bigr)^{1/p}.
\]

\begin{corollary}\label{cor:interior-balls-cubes}
Let $\Omega\subset\mathbb R^d$ be a bounded Lipschitz domain,
$0<p<\infty$, and $\varphi\in\mathcal G_p$.
Assume that \eqref{eq:Hardy-intrinsic} holds.
Then
\[
M_{\varphi,\mathrm{ball}}^p(\Omega)
=
M_\varphi^p(\Omega)
=
\mathcal M_\varphi^p(\Omega). 
\]
\end{corollary}

\begin{proof}
By Corollary~\ref{cor:cap=nor}, it remains to compare
$M_{\varphi,\mathrm{ball}}^p(\Omega)$ with
$\mathcal M_\varphi^p(\Omega)$. Set $R_0=\operatorname{diam}(\Omega)$. 

\emph{Step }1. Remark~\ref{rem:cubes-balls-Omega} gives
\[
\|f\mid\mathcal M_\varphi^p(\Omega)\|
\sim
\sup_{\substack{x\in\Omega\\0<r\le R_0}}
\varphi(r)
\Bigl(
\frac{1}{|B(x,r)\cap\Omega|}
\int_{B(x,r)\cap\Omega}|f(y)|^p\,dy
\Bigr)^{1/p}.
\]
Taking only balls contained in $\Omega$ on the right-hand
side, and using $|B(x,r)|\sim r^d$, we obtain
\[
\|f\mid M_{\varphi,\mathrm{ball}}^p(\Omega)\|
\lesssim
\|f\mid\mathcal M_\varphi^p(\Omega)\|.
\]

\emph{Step }2. For the reverse inequality, we repeat the proof of
Proposition~\ref{prop:nor-cap}, with the ball-based
quasi-norm in place of the cube-based quasi-norm.
We give some details.

For all balls $B(z,\rho)\subset\Omega$, 
\[
\int_{B(z,\rho)}|f(y)|^p\,dy
\le
\|f\mid M_{\varphi,\mathrm{ball}}^p(\Omega)\|^p
\frac{\rho^d}{\varphi(\rho)^p}.
\]
Thus all estimates for balls that were obtained there from Lemma~\ref{lem:interior-ball-estimate} now follow directly from the source quasi-norm.

In particular, all balls in the covering provided by Lemma~\ref{lem:boundary-covering-Omega} are contained in $\Omega$. 
The integral over each such ball can therefore be estimated by the preceding inequality. 
Since its radius is a fixed multiple of $2^{-J}R_0$, the properties of $\varphi\in\mathcal G_p$ give the same scale factor ${(2^{-J}R_0)^d}/{\varphi(2^{-J}R_0)^p}$ as in the proof of Proposition~\ref{prop:nor-cap}.
The covering estimates and the subsequent summation are unchanged. 
Applying \eqref{eq:Hardy-intrinsic}, we obtain
\[
\|f\mid\mathcal M_\varphi^p(\Omega)\|
\lesssim
\|f\mid M_{\varphi,\mathrm{ball}}^p(\Omega)\|.
\]
Together with Corollary~\ref{cor:cap=nor}, this proves the asserted equivalences.
\end{proof}

In the classical case, the above results recover the corresponding Morrey spaces on bounded
Lipschitz domains.

\begin{remark}\label{rem:intrinsic-classical}
    Let $\Omega\subset\mathbb R^d$ be a bounded Lipschitz domain and
$0<p\le u<\infty$. Set $\varphi(r):=r^{d/u}$.
By Remark~\ref{rem:cubes-balls-Omega},
\[
\begin{aligned}
\|f\mid\mathcal M_\varphi^p(\Omega)\|
&\sim
\sup_{\substack{x\in\Omega\\0<r\le R_0}}
|\Omega\cap B(x,r)|^{1/u-1/p}
\Bigl(
\int_{\Omega\cap B(x,r)}
|f(y)|^p\,dy
\Bigr)^{1/p}
\\
&=
\|f\mid\mathcal M_{u,p}(\Omega)\|,
\end{aligned}
\]
and, by Corollary~\ref{cor:interior-balls-cubes},
\[
\begin{aligned}
\|f\mid M_\varphi^p(\Omega)\|
&\sim
\sup_{B(x,r)\subset\Omega}
r^{d/u-d/p}
\Bigl(
\int_{B(x,r)}|f(y)|^p\,dy
\Bigr)^{1/p}
\\
&=
\|f\mid\mathbb M_{u,p}(\Omega)\|,
\end{aligned}
\]
where $\mathcal M_{u,p}(\Omega)$ and
$\mathbb M_{u,p}(\Omega)$ were introduced in \cite[Definition~2.1]{HSS18} and we follow their notation. 

It remains to verify \eqref{eq:Hardy-intrinsic}. Set
\[
\theta:=1-\frac{dp}{u}.
\]
If $d/u<1/p$, namely $\theta>0$, then ${\varphi(r)^p}/{r}=r^{-\theta}$ and 
${2^{-J}R_0}/{\varphi(2^{-J}R_0)^p}
=(2^{-J}R_0)^\theta$.
Therefore,
\[
\frac{\varphi(r)^p}{r}
\sum_{\substack{J\in\mathbb N_0\\2^{-J}R_0\le r}}
\frac{2^{-J}R_0}{\varphi(2^{-J}R_0)^p}
=
r^{-\theta}
\sum_{\substack{J\in\mathbb N_0\\2^{-J}R_0\le r}}
(2^{-J}R_0)^\theta
\lesssim1.
\]
Thus Corollaries~\ref{cor:cap=nor}
and~\ref{cor:interior-balls-cubes} recover the coincidence of
$\mathcal M_{u,p}(\Omega)$ and $\mathbb M_{u,p}(\Omega)$ in the
range $d/u<1/p$ obtained in \cite[Theorem~2.7]{HSS18}.

When $\theta\le0$, the series in
\eqref{eq:Hardy-intrinsic} diverges. Hence the sufficient condition
used above does not apply in that range.
\end{remark}

Finally, we obtain a restriction characterization of the space $\mathcal M_\varphi^p(\Omega)$.

\begin{proposition}\label{prop:whole-space-restriction}
Let $\Omega\subset\mathbb R^d$ be a bounded domain such that there exists
$c_\Omega>0$ with
\[
|\Omega\cap Q(x,r)|\ge c_\Omega r^d,
\qquad x\in\Omega,\quad
0<r\le\operatorname{diam}(\Omega).
\]
Let $0<p<\infty$ and $\varphi\in \mathcal G_p$. Then, for all measurable
functions $f$ on $\Omega$,
\[
\|f\mid\mathcal M_\varphi^p(\Omega)\|
\sim
\inf_{\substack{
g\in\mathcal M_\varphi^p(\mathbb R^d)\\
g|_\Omega=f\text{ a.e. on }\Omega.}}
\|g\mid\mathcal M_\varphi^p(\mathbb R^d)\|.
\]
In particular, $\mathcal M_\varphi^p(\Omega)$ is the restriction of
$\mathcal M_\varphi^p(\mathbb R^d)$ to $\Omega$, with equivalent quasi-norms.
\end{proposition}

\begin{proof}
We adapt some ideas from the proof of \cite[Theorem~2.3]{HSS18} and set $R_0:=\operatorname{diam}(\Omega)$.

\emph{Step }1. We first prove
\[
\|f\mid\mathcal M_\varphi^p(\Omega)\|
\lesssim
\|g\mid\mathcal M_\varphi^p(\mathbb R^d)\|
\]
for all $g\in\mathcal M_\varphi^p(\mathbb R^d)$ satisfying
$g|_\Omega=f$. Let $x\in\Omega$ and $0<r\le R_0$. By the assumption  
$|\Omega\cap Q(x,r)|\ge c_\Omega|Q(x,r)|$, 
\[
\begin{aligned}
&\varphi(r)
\Bigl(
\frac{1}{|\Omega\cap Q(x,r)|}
\int_{\Omega\cap Q(x,r)}|f(y)|^p\,dy
\Bigr)^{1/p}
\\
&\qquad\lesssim
\varphi(r)
\Bigl(
\frac{1}{|Q(x,r)|}
\int_{Q(x,r)}|g(y)|^p\,dy
\Bigr)^{1/p}
\le
\|g\mid\mathcal M_\varphi^p(\mathbb R^d)\|.
\end{aligned}
\]
Taking the supremum gives the desired estimate.

\emph{Step }2. For the converse estimate, let
\[
\widetilde f(x)=
\begin{cases}
    f(x), &\quad x\in \Omega\\
    0, &\quad x\in \mathbb R^d\setminus \Omega.
\end{cases}
\]
We show that
\begin{equation}
\|\widetilde f\mid\mathcal M_\varphi^p(\mathbb R^d)\|
\lesssim \|f\mid\mathcal M_\varphi^p(\Omega)\|.
\label{eb:rest-wf-f}
\end{equation}

\emph{Step }2.1. Let $Q(z,r)$ be an arbitrary cube in $\mathbb R^d$.
We first assume that $0<r\le R_0/3$. If
$Q(z,r)\cap\Omega=\emptyset$, then
$\int_{Q(z,r)}|\widetilde f(y)|^p\,dy=0$. Otherwise choose
$x\in Q(z,r)\cap\Omega$. Then $Q(z,r)\subset Q(x,3r)$. 
Since $3r\le R_0$, the definition of
$\mathcal M_\varphi^p(\Omega)$ gives
\[
\int_{Q(z,r)}|\widetilde f(y)|^p\,dy
\le
|\Omega\cap Q(x,3r)|\,\varphi(3r)^{-p}\|f\mid\mathcal M_\varphi^p(\Omega)\|^p.
\]
Using $|\Omega\cap Q(x,3r)|\lesssim r^d$ and
$\varphi(r)\lesssim\varphi(3r)$, we obtain
\begin{equation}
\varphi(r)
\Bigl(
\frac{1}{|Q(z,r)|}
\int_{Q(z,r)}|\widetilde f(y)|^p\,dy
\Bigr)^{1/p}
\lesssim \|f\mid\mathcal M_\varphi^p(\Omega)\|.
\label{eq:S2.1}
\end{equation}

\emph{Step }2.2. It remains to consider $r>R_0/3$. Since $\Omega$ is bounded, it can be
covered by finitely many cubes
$Q(x_k,R_0)$, $x_k\in\Omega$, $k=1,\ldots,N_0$, where $N_0$ depends only
on $\Omega$ and $d$. Hence
\[
\int_\Omega |f(y)|^p\,dy
\lesssim
R_0^d\varphi(R_0)^{-p}\|f\mid\mathcal M_\varphi^p(\Omega)\|^p.
\]
By $\varphi\in \mathcal G_p$, $\varphi(r)r^{-d/p}
\lesssim (R_0/3)^{-d/p}\varphi(R_0/3)\lesssim
\varphi(R_0)R_0^{-d/p}$ for $r>R_0/3$.
Using $\int_{Q(z,r)}|\widetilde f(y)|^p\,dy
\le
\int_\Omega|f(y)|^p\,dy$, 
we obtain
\begin{equation}
\varphi(r)
\Bigl(
\frac{1}{|Q(z,r)|}
\int_{Q(z,r)}|\widetilde f(y)|^p\,dy
\Bigr)^{1/p}
\lesssim \|f\mid\mathcal M_\varphi^p(\Omega)\|.
\label{eq:S2.2}
\end{equation}

Combining \eqref{eq:S2.1} and \eqref{eq:S2.2}, and taking the supremum over all
cubes $Q(z,r)\subset\mathbb R^d$, yields \eqref{eb:rest-wf-f}. Hence the proof is completed.
\end{proof}

\begin{remark}\label{rem:hss-restriction}
Proposition~\ref{prop:whole-space-restriction} is a generalized analogue of \cite[Theorem~2.3]{HSS18}. 
In \cite{HSS18}, the corresponding result is
stated for type A domains. Rather than imposing the type A condition, we
formulate the geometric assumption explicitly through the measure density
condition $|\Omega\cap Q(x,r)|\ge c_\Omega r^d$, $x\in\Omega$, $0<r\le \operatorname{diam}(\Omega)$, 
adapted to the cube-based formulation used here. This formulation makes
the precise geometric assumption transparent; in particular, it applies
to every bounded Lipschitz domain.
\end{remark}

\begin{corollary}\label{cor:M-res}
Let $\Omega\subset\mathbb R^d$ be a bounded Lipschitz domain,
$0<p<\infty$, and $\varphi\in\mathcal G_p$.
Assume that condition \eqref{eq:Hardy-intrinsic} holds.
Then, for all measurable functions $f$ on $\Omega$,
\[
\|f\mid M_\varphi^p(\Omega)\|
\sim
\inf_{\substack{
g\in M_\varphi^p(\mathbb R^d)\\
g|_\Omega=f \text{ a.e. on } \Omega}}
\|g\mid M_\varphi^p(\mathbb R^d)\|.
\]
In particular, $M_\varphi^p(\Omega)$ is the restriction of
$M_\varphi^p(\mathbb R^d)$ to $\Omega$, with equivalent quasi-norms.
\end{corollary}

\begin{proof}
This directly follows from Corollary~\ref{cor:cap=nor} and Proposition~\ref{prop:whole-space-restriction}.
\end{proof}


\section{Embeddings of Komori--Shirai-type spaces}\label{sec:KS}

In this section, we investigate embeddings for Komori--Shirai-type generalized 
weighted Morrey spaces $\mathcal{M}_p^\varphi(\mathbb{R}^d;w)$. 
More precisely, we consider embeddings of the form
\begin{equation}\label{eb:main-KS}
    \mathcal{M}_{p_1}^{\varphi_1}(\mathbb{R}^d;w)
    \hookrightarrow M^{p_2}_{\varphi_2}(\mathbb{R}^d),
\end{equation}
where $M_\varphi^p(\mathbb R^d)$ denotes the unweighted generalized Morrey spaces as in Definition~\ref{def:Mf}. 
We first consider piecewise power weights and then general Muckenhoupt weights.
The corresponding embedding results for Komori--Shirai-type weighted Morrey 
spaces $\mathcal{M}_p^u(\mathbb{R}^d;w)$ were established in \cite{HS17}.

We carefully distinguish between the assumptions on $\varphi_1$ needed to prove 
sufficiency for the embedding \eqref{eb:main-KS} and those needed to prove 
necessity. The sufficiency part requires only that $\varphi_1$ be positive and measurable, whereas the necessity part establishes that $\varphi_1$ not only belongs to $\mathcal{G}_{p_1}^{w_{\alpha,\beta}}$ but also satisfies the conditions \eqref{eq:def-LC}, \eqref{eq:def-LRC} and \eqref{eq:LRC-norad}.

\subsection{Embeddings with piecewise power weights}
\label{ssec:KS-pw}

Throughout this subsection, we set
\[
D_k:=Q(0,2^k), \qquad E_k:=D_k\setminus D_{k-1}, \qquad k\in\mathbb Z.
\]
Then the sets $E_k$ are pairwise disjoint and $D_k=\bigcup_{j=-\infty}^k E_j$
up to a null set.

We begin with a useful observation concerning piecewise power weights.

\begin{lemma}\label{lem:wab-cube-annuli}
Let $w_{\alpha,\beta}$ be given by \eqref{eq:def-wab}, where $\alpha,\beta>-d$.
Then 
\begin{equation}\label{eq:weight-cube-annulus}
w_{\alpha,\beta}(x)
\sim
\frac{w_{\alpha,\beta}(D_k)}{|D_k|},
\qquad x\in E_k,
\end{equation}
uniformly in $k\in\mathbb Z$.
\end{lemma}

\begin{proof}
Since $x\in E_k$, we have $2^{k-2}\le |x|\le \sqrt d\,2^{k-1}$, 
so $|x|\sim 2^k$ on $E_k$. 
If $E_k$ lies in $\{|x|<1\}$, then $w_{\alpha,\beta}(x)\sim2^{j\alpha}$ and $w_{\alpha,\beta}(D_j)\sim2^{j(d+\alpha)}$. If $E_k$ lies in $\{|x|>1\}$, then $w_{\alpha,\beta}(x)\sim2^{j\beta}$ and $w_{\alpha,\beta}(D_j)\sim2^{j(d+\beta)}$.
Only finitely many annuli meet $\{|x|=1\}$, and on each of them both sides are comparable to a positive constant. This proves the assertion.
\end{proof}

We now derive sufficient conditions for the embedding \eqref{eb:main-KS} with  $w=w_{\alpha,\beta}$.

\begin{theorem}\label{thm:Mfwwab-suff}
Let $0<p_2\le p_1<\infty$, $\varphi_1$ be a positive measurable
function on $\mathbb R^d\times(0,\infty)$, 
$\varphi_2\in\mathcal G_{p_2}$ satisfy \eqref{eq:def-C}, and $w_{\alpha,\beta}$ be given by
\eqref{eq:def-wab}, where $\alpha,\beta>-d$.

Then the embedding
\begin{equation}\label{eb:Mfwwab-Mfb}
\mathcal M^{p_1}_{\varphi_1}
(\mathbb R^d;w_{\alpha,\beta})
\hookrightarrow
M^{p_2}_{\varphi_2}(\mathbb R^d)
\end{equation}
holds if
\begin{equation}\label{eq:cond1}
\sup_{Q\in\mathcal Q}
\frac{\varphi_2(Q)}{\varphi_1(Q)}
<\infty
\end{equation}
and
\begin{equation}\label{eq:cond2}
\sup_{k\in\mathbb Z}
\frac{\varphi_2(D_k)}{|D_k|^{1/p_2}}
\Bigl(
\sum_{j=-\infty}^k
\frac{|D_j|}{\varphi_1(D_j)^{p_2}}
\Bigr)^{1/p_2}
<\infty.
\end{equation}
\end{theorem}

\begin{proof}
Let $f\in \mathcal M^{p_1}_{\varphi_1} (\mathbb R^d;w_{\alpha,\beta})$
and fix an arbitrary cube $Q_0=Q(x_0,r)$, where $r=\ell(Q_0)$.
We distinguish two cases according to the position of $Q_0$ relative
to the origin.

\emph{Step }1.
Assume that $|x_0|>C_d r$, 
where $C_d>0$ is a sufficiently large constant depending only on the dimension. 
Since $|x-x_0|\le \frac{\sqrt d}{2}r$, $x\in Q_0$,
we have $|x|\sim|x_0|$ on $Q_0$.
By a similar argument used in the proof of Lemma~\ref{lem:wab-cube-annuli},
\begin{equation}
w_{\alpha,\beta}(x)
\sim
\frac{w_{\alpha,\beta}(Q_0)}{|Q_0|},
\qquad x\in Q_0.
\label{eq:weight-away-origin}
\end{equation}

Since $p_2\le p_1$, H\"older's inequality and
\eqref{eq:weight-away-origin} give
\[
\begin{aligned}
\Bigl(
\int_{Q_0}|f(y)|^{p_2}\,dy
\Bigr)^{1/p_2}
&\le
|Q_0|^{1/p_2-1/p_1}
\Bigl(
\int_{Q_0}|f(y)|^{p_1}\,dy
\Bigr)^{1/p_1}
\\
&\lesssim
|Q_0|^{1/p_2-1/p_1}
\Bigl(
\frac{|Q_0|}{w_{\alpha,\beta}(Q_0)}
\Bigr)^{1/p_1}
\Bigl(
\int_{Q_0}|f(y)|^{p_1}
w_{\alpha,\beta}(y)\,dy
\Bigr)^{1/p_1}
\\
&\le
\frac{|Q_0|^{1/p_2}}{\varphi_1(Q_0)}
\|f\mid
\mathcal M^{p_1}_{\varphi_1}
(\mathbb R^d;w_{\alpha,\beta})\|.
\end{aligned}
\]
It follows from \eqref{eq:cond1} that
\[
\begin{aligned}
\varphi_2(Q_0)
\Bigl(
\frac{1}{|Q_0|}
\int_{Q_0}|f(y)|^{p_2}\,dy
\Bigr)^{1/p_2}
&\lesssim
\frac{\varphi_2(Q_0)}{\varphi_1(Q_0)}
\|f\mid
\mathcal M^{p_1}_{\varphi_1}
(\mathbb R^d;w_{\alpha,\beta})\|
\\
&\lesssim
\|f\mid
\mathcal M^{p_1}_{\varphi_1}
(\mathbb R^d;w_{\alpha,\beta})\|.
\end{aligned}
\]

\emph{Step }2.
Assume now that $|x_0|\le C_d r$.
Let $k\in\mathbb Z$ be the smallest integer such that
$Q_0\subset D_k$.
Since $|x_0|\lesssim r$, the cube $Q_0$ is contained in an
origin-centered cube with side length comparable to $r$. By the
minimality of $k$, it follows that
$2^k\sim r$ and hence $|Q_0|\sim|D_k|$. 
Moreover, the centers of $Q_0$ and $D_k$ are at distance comparable
to their side lengths. 
Therefore, \eqref{eq:def-C} and $\varphi_2\in\mathcal G_{p_2}$ yield
\begin{equation}
\frac{\varphi_2(Q_0)}{|Q_0|^{1/p_2}}
\lesssim
\frac{\varphi_2(D_k)}{|D_k|^{1/p_2}}.
\label{eq:qm-Q0toDk}
\end{equation}

Note that the definition of
$\mathcal M^{p_1}_{\varphi_1}(\mathbb R^d;w_{\alpha,\beta})$
implies, for all $j\in\mathbb Z$,
\begin{equation}\label{eq:dyadic-cube-source}
    \Bigl(
\int_{D_j}|f(x)|^{p_1}w_{\alpha,\beta}(x)\,dx
\Bigr)^{1/p_1}
\le
\frac{w_{\alpha,\beta}(D_j)^{1/p_1}}{\varphi_1(D_j)}
\|f\mid
\mathcal M^{p_1}_{\varphi_1}
(\mathbb R^d;w_{\alpha,\beta})\|.
\end{equation}
We now use the decomposition
$D_k=\bigcup_{j=-\infty}^k E_j$ up to a null set. 
For $j\le k$, H\"older's inequality,
\eqref{eq:weight-cube-annulus}, and
\eqref{eq:dyadic-cube-source} give
\[
\begin{aligned}
\Bigl(
\int_{E_j}|f(y)|^{p_2}\,dy
\Bigr)^{1/p_2}
&\le
|E_j|^{1/p_2-1/p_1}
\Bigl(
\int_{E_j}|f(y)|^{p_1}\,dy
\Bigr)^{1/p_1}
\\
&\lesssim
|D_j|^{1/p_2-1/p_1}
\Bigl(
\frac{|D_j|}{w_{\alpha,\beta}(D_j)}
\Bigr)^{1/p_1}
\Bigl(
\int_{E_j}|f(y)|^{p_1}
w_{\alpha,\beta}(y)\,dy
\Bigr)^{1/p_1}
\\
&\le
|D_j|^{1/p_2-1/p_1}
\Bigl(
\frac{|D_j|}{w_{\alpha,\beta}(D_j)}
\Bigr)^{1/p_1}
\Bigl(
\int_{D_j}|f(y)|^{p_1}
w_{\alpha,\beta}(y)\,dy
\Bigr)^{1/p_1}
\\
&\le
\frac{|D_j|^{1/p_2}}{\varphi_1(D_j)}
\|f\mid
\mathcal M^{p_1}_{\varphi_1}
(\mathbb R^d;w_{\alpha,\beta})\|.
\end{aligned}
\]
Thus,
\begin{equation}
\Bigl(
\int_{D_k}|f(y)|^{p_2}\,dy
\Bigr)^{1/p_2}
\lesssim
\|f\mid
\mathcal M^{p_1}_{\varphi_1}
(\mathbb R^d;w_{\alpha,\beta})\|
\Bigl(
\sum_{j=-\infty}^k
\frac{|D_j|}{\varphi_1(D_j)^{p_2}}
\Bigr)^{1/p_2}.
\label{eq:Dk-target-estimate}
\end{equation}

Since $Q_0\subset D_k$, it follows from
\eqref{eq:qm-Q0toDk}, \eqref{eq:Dk-target-estimate}, and
\eqref{eq:cond2} that
\[
\begin{aligned}
\frac{\varphi_2(Q_0)}{|Q_0|^{1/p_2}}
\Bigl(
\int_{Q_0}|f(y)|^{p_2}\,dy
\Bigr)^{1/p_2}
&\lesssim
\frac{\varphi_2(D_k)}{|D_k|^{1/p_2}}
\Bigl(
\int_{D_k}|f(y)|^{p_2}\,dy
\Bigr)^{1/p_2}
\\
&\lesssim
\|f\mid
\mathcal M^{p_1}_{\varphi_1}
(\mathbb R^d;w_{\alpha,\beta})\|
\frac{\varphi_2(D_k)}{|D_k|^{1/p_2}}
\Bigl(
\sum_{j=-\infty}^k
\frac{|D_j|}{\varphi_1(D_j)^{p_2}}
\Bigr)^{1/p_2}
\\
&\lesssim
\|f\mid
\mathcal M^{p_1}_{\varphi_1}
(\mathbb R^d;w_{\alpha,\beta})\|.
\end{aligned}
\]
Taking the supremum over all arbitrary cubes
$Q_0\in\mathcal Q$ proves \eqref{eb:Mfwwab-Mfb}. Hence we finish the proof. 
\end{proof}

\begin{remark}\label{rem:asymmetry-regularity}
The assumptions on $\varphi_1$ and $\varphi_2$ in
Theorem~\ref{thm:Mfwwab-suff} play different roles. The proof does not
require any comparison between the values of $\varphi_1$ on different
cubes. Indeed, each source estimate is applied directly on a cube $D_j$
appearing in the definition of the source Morrey quasi-norm; see \eqref{eq:dyadic-cube-source}.

In the near-origin case, however, an arbitrary cube $Q$ occurring in
the target quasi-norm is replaced by an origin-centered dyadic cube
$D_k$ with $\ell(D_k)\sim\ell(Q)$. 
Thus the comparison between
$\varphi_2(Q)$ and $\varphi_2(D_k)$ follows from the almost increasing property
of $\varphi_2\in \mathcal G_{p_2}$ together with condition~\eqref{eq:def-C}.
These assumptions are therefore used for $\varphi_2$, whereas no
corresponding comparison is needed for $\varphi_1$.
\end{remark}

We now turn to the necessity part.

\begin{theorem}\label{thm:Mfwwab-nece}
Let $0<p_1,p_2<\infty$, $\varphi_1\in\mathcal G_{p_1}^{w_{\alpha,\beta}}$ satisfy \eqref{eq:def-LC}--\eqref{eq:LRC-norad}, $\varphi_2\in\mathcal G_{p_2}$ satisfy \eqref{eq:def-C}, and $w_{\alpha,\beta}$ be
given by \eqref{eq:def-wab}, where $\alpha,\beta>-d$.

If the embedding
\[
\mathcal M^{p_1}_{\varphi_1}
(\mathbb R^d;w_{\alpha,\beta})
\hookrightarrow
M^{p_2}_{\varphi_2}(\mathbb R^d)
\]
holds, then
\[
p_2\le p_1
\]
and
\begin{equation}\label{eq:cond1'}
\sup_{Q\in\mathcal Q}
\frac{\varphi_2(Q)}{\varphi_1(Q)}
<\infty.
\end{equation}
\end{theorem}

\begin{proof}
The necessity of $p_2\le p_1$ follows from Remark~\ref{rem:regularity-cubes-applications} and Theorem~\ref{thm:nece-Omega}. 
Testing the embedding \eqref{eb:Mfwwab-Mfb} on characteristic functions
and using Lemma~\ref{lem:1Qww}, we obtain $\varphi_2(Q)
\lesssim
\varphi_1(Q)$, $Q\in\mathcal Q$.
Taking the supremum over $Q\in\mathcal Q$ proves
\eqref{eq:cond1'}.
\end{proof}

\begin{remark}\label{gap}
The necessary condition \eqref{eq:cond1'} and the sufficient condition
\eqref{eq:cond2} arise from different arguments. Testing the embedding
with the characteristic function of a cube gives
\eqref{eq:cond1'}, which is an estimate for each cube separately. In
the proof of Theorem~\ref{thm:Mfwwab-suff}, however, a general function
is decomposed over the dyadic cube annuli $E_j$. For a fixed $D_k$, the
estimates for all $E_j$, $j\le k$, must be added. The requirement that
this sum be finite gives \eqref{eq:cond2}.

Under the present hypotheses, testing with characteristic functions does not show that \eqref{eq:cond2} is necessary. 
Thus, \eqref{eq:cond2} is stated only as a sufficient condition. Its divergence does not in general imply failure of the embedding; it only shows that the argument used in the proof of Theorem~\ref{thm:Mfwwab-suff} does not apply.

In special cases, the gap may disappear. This happens, for example,
in the classical Morrey setting considered in
Remark~\ref{rem:classical-HS17}, where \eqref{eq:cond2} reduces to the
sharp parameter conditions in \cite[Theorem~3.12]{HS17}.
\end{remark}

\begin{remark}\label{rem:classical-HS17}
Let $0<p_1\le u_1<\infty$ and $0<p_2<u_2<\infty$, and set
\[
\varphi_1(Q)=w_{\alpha,\beta}(Q)^{1/u_1},
\qquad
\varphi_2(Q)=|Q|^{1/u_2},
\qquad Q\in\mathcal Q.
\]
Then the embedding \eqref{eb:Mfwwab-Mfb} becomes 
\begin{equation}\label{eb:Mwwab-Mb}
\mathcal M_{p_1}^{u_1}
(\mathbb R^d;w_{\alpha,\beta})
\hookrightarrow
M_{p_2}^{u_2}(\mathbb R^d).
\end{equation}

We first examine the necessary condition \eqref{eq:cond1}, which now
reads
\begin{equation}\label{eq:classical-local-condition}
\sup_{Q\in\mathcal Q}
\frac{|Q|^{1/u_2}}
     {w_{\alpha,\beta}(Q)^{1/u_1}}
<\infty.
\end{equation}
Let $x_0\ne0$ be fixed and take cubes $Q_0$ centered at $x_0$ whose
side lengths tend to zero. By Remark~\ref{rem:regularity-cubes-applications}, $w_{\alpha,\beta}(Q_0)\sim |Q_0|$.
Consequently, \eqref{eq:classical-local-condition} implies
\[
|Q_0|^{1/u_2-1/u_1}\lesssim1
\qquad\text{as }|Q_0|\to0,
\]
and hence $u_1\ge u_2$.

Next, apply \eqref{eq:classical-local-condition} to the
origin-centered dyadic cubes $D_k$. Since $|D_k|\sim2^{kd}$ 
and
\[
w_{\alpha,\beta}(D_k)\sim
\begin{cases}
2^{k(d+\alpha)},&k< 0,\\
2^{k(d+\beta)},&k\ge0,
\end{cases}
\]
up to constants independent of $k$, the cases $k\to-\infty$ and
$k\to+\infty$ give, respectively,
\[
\frac{d}{u_2}-\frac{d+\alpha}{u_1}\ge0
\qquad\text{and}\qquad
\frac{d}{u_2}-\frac{d+\beta}{u_1}\le0.
\]
Equivalently,
\begin{equation}\label{eq:classical-ratio}
1+\frac{\alpha}{d}
\le
\frac{u_1}{u_2}
\le
1+\frac{\beta}{d}.
\end{equation}

We now verify that the sufficient condition \eqref{eq:cond2} gives
exactly the two inequalities in \eqref{eq:classical-ratio}. In the
present setting, \eqref{eq:cond2} becomes
\begin{equation}\label{eq:classical-multiscale}
\sup_{k\in\mathbb Z}
2^{kd(1/u_2-1/p_2)}
\Bigl(
\sum_{j=-\infty}^{k}
\frac{2^{jd}}
     {w_{\alpha,\beta}(D_j)^{p_2/u_1}}
\Bigr)^{1/p_2}
<\infty.
\end{equation}
Set
\[
a_\alpha
:=
d-\frac{(d+\alpha)p_2}{u_1},
\qquad
a_\beta
:=
d-\frac{(d+\beta)p_2}{u_1},
\]
and
\[
b_\alpha
:=
\frac{d}{u_2}-\frac{d+\alpha}{u_1},
\qquad
b_\beta
:=
\frac{d}{u_2}-\frac{d+\beta}{u_1}.
\]

For $k< 0$, the sum in \eqref{eq:classical-multiscale} is comparable
to $\sum_{j=-\infty}^{k}2^{j a_\alpha}$. 
This series is finite if and only if $a_\alpha>0$, and in that case
$\sum_{j=-\infty}^{k}2^{j a_\alpha}\sim 2^{k a_\alpha}$. 
It follows that the expression in
\eqref{eq:classical-multiscale} is comparable to 
$2^{k(d(1/u_2-1/p_2)+a_\alpha/p_2)} = 2^{k b_\alpha}$.
Since $k\to-\infty$, this is uniformly bounded if and only if
$b_\alpha\ge0$. Moreover, since $p_2<u_2$, 
$a_\alpha = p_2[ d(\frac1{p_2}-\frac1{u_2})+b_\alpha]$, 
so $b_\alpha\ge0$ already implies $a_\alpha>0$. 
Thus the part of \eqref{eq:classical-multiscale} corresponding to $k\le0$ is equivalent to
\[
\frac{d}{u_2}-\frac{d+\alpha}{u_1}\ge0.
\]

For $k\ge0$, the terms with $j\le0$ give a fixed finite quantity, and
the remaining sum satisfies
\[
\sum_{j=0}^{k}2^{j a_\beta}
\sim
\begin{cases}
2^{k a_\beta},&a_\beta>0,\\
1+k,&a_\beta=0,\\
1,&a_\beta<0.
\end{cases}
\]
If $a_\beta>0$, the expression in
\eqref{eq:classical-multiscale} is comparable to 
$2^{k(d(1/u_2-1/p_2)+a_\beta/p_2)} = 2^{k b_\beta}$, 
so it is uniformly bounded as $k\to+\infty$ if and only if
$b_\beta\le0$.
If $a_\beta=0$, the corresponding expression is comparable to
$2^{kd(1/u_2-1/p_2)}(1+k)^{1/p_2}$, 
which is bounded since $p_2<u_2$. 
In this case, $b_\beta=d(\frac1{u_2}-\frac1{p_2})<0$. 
Finally, if $a_\beta<0$, the inner sum remains bounded and the outer
factor tends to zero, again since $p_2<u_2$. In this case,
$b_\beta = d(\frac1{u_2}-\frac1{p_2}) + \frac{a_\beta}{p_2} < 0$.
Therefore, in all cases, the part of
\eqref{eq:classical-multiscale} corresponding to $k\ge0$ is equivalent
to
\[
\frac{d}{u_2}-\frac{d+\beta}{u_1}\le0.
\]

Hence \eqref{eq:cond2} is equivalent to the two inequalities in
\eqref{eq:classical-ratio}. Combining this with the condition
$p_1\ge p_2$ and the necessary condition $u_1\ge u_2$, we conclude
that \eqref{eb:Mwwab-Mb} holds if and only if
\[
p_1\ge p_2,\qquad
u_1\ge u_2,\qquad
1+\frac{\alpha}{d}
\le
\frac{u_1}{u_2}
\le
1+\frac{\beta}{d}.
\]
These are precisely the conditions obtained in
\cite[Theorem~3.12]{HS17}.
\end{remark}

We finally consider the case where the target space is a Lebesgue
space.

\begin{proposition}\label{prop:Lebesgue-target-wab}
Let $w_{\alpha,\beta}$ be given by \eqref{eq:def-wab}, where $\alpha,\beta>-d$.

\begin{enumerate}[\bfseries\upshape (i)]
\item
Let $0<p_2\le p_1<\infty$, and let $\varphi_1$ be a positive measurable function. If 
\begin{equation}\label{eq:Leb-cond}
\sum_{j=-\infty}^{\infty}
\frac{|D_j|}{\varphi_1(D_j)^{p_2}}
<\infty,
\end{equation}
then
\begin{equation}\label{eb:Mfvwab-Lp}
\mathcal M^{p_1}_{\varphi_1}
(\mathbb R^d;w_{\alpha,\beta})
\hookrightarrow
L_{p_2}(\mathbb R^d).
\end{equation}

\item
Let $0<p_1,p_2<\infty$, and let 
$\varphi_1\in\mathcal G_{p_1}^{w_{\alpha,\beta}}$ satisfy \eqref{eq:def-LC}--\eqref{eq:LRC-norad}. If the embedding 
\eqref{eb:Mfvwab-Lp} holds, then
\[
p_2\le p_1\qquad\text{and}\qquad \sup_{Q\in\mathcal Q}
\frac{|Q|^{1/p_2}}{\varphi_1(Q)}
<\infty.
\]
\end{enumerate}
\end{proposition}

\begin{proof}
We first prove part~(i). Let $f\in
\mathcal M^{p_1}_{\varphi_1}
(\mathbb R^d;w_{\alpha,\beta})$. 
Using the cube annuli $E_j=D_j\setminus D_{j-1}$ and the same estimate
as in the proof of Theorem~\ref{thm:Mfwwab-suff}, we obtain
\[
\Bigl(
\int_{E_j}|f(y)|^{p_2}\,dy
\Bigr)^{1/p_2}
\lesssim
\frac{|D_j|^{1/p_2}}{\varphi_1(D_j)}
\|f\mid
\mathcal M^{p_1}_{\varphi_1}
(\mathbb R^d;w_{\alpha,\beta})\|.
\]
Therefore,
\[
\begin{aligned}
\int_{\mathbb R^d}|f(y)|^{p_2}\,dy
=
\sum_{j=-\infty}^{\infty}
\int_{E_j}|f(y)|^{p_2}\,dy
\lesssim
\|f\mid
\mathcal M^{p_1}_{\varphi_1}
(\mathbb R^d;w_{\alpha,\beta})\|^{p_2}
\sum_{j=-\infty}^{\infty}
\frac{|D_j|}{\varphi_1(D_j)^{p_2}}.
\end{aligned}
\]
Condition \eqref{eq:Leb-cond} now proves
\eqref{eb:Mfvwab-Lp}.

To prove part~(ii), assume that \eqref{eb:Mfvwab-Lp} holds and define $\varphi_2(Q):=|Q|^{1/p_2}$,
$Q\in\mathcal Q$. 
For all $f\in L_{p_2}(\mathbb R^d)$,
\[
\begin{aligned}
\|f\mid M^{p_2}_{\varphi_2}(\mathbb R^d)\|
=
\sup_{Q\in\mathcal Q}
|Q|^{1/p_2}
\Bigl(
\frac{1}{|Q|}
\int_Q|f(x)|^{p_2}\,dx
\Bigr)^{1/p_2}
=
\|f\mid L_{p_2}(\mathbb R^d)\|.
\end{aligned}
\]
Hence the assumed embedding can be written as
\[
\mathcal M^{p_1}_{\varphi_1}
(\mathbb R^d;w_{\alpha,\beta})
\hookrightarrow
M^{p_2}_{\varphi_2}(\mathbb R^d).
\]
The conclusion now follows from Theorem~\ref{thm:Mfwwab-nece}.
\end{proof}

\subsection{Embeddings with general Muckenhoupt weights}\label{ssec:KS-w}

In this subsection we consider embeddings for general Muckenhoupt weights.
Some of the arguments are inspired by the techniques used in \cite[Section~4]{HS17}.

\begin{theorem}\label{thm:wMfa-Mfb} 
Let $0<p_1,p_2<\infty$, $\varphi_1\in\mathcal G_{p_1}^w$, $\varphi_2\in\mathcal G_{p_2}$ and  $w\in\mathcal A_\infty$. 
\begin{enumerate}[\bfseries\upshape (i)] 
\item  If 
\begin{equation}\label{eb:wMfa-Mfb} 
\mathcal M^{p_1}_{\varphi_1}(\mathbb R^d;w) \hookrightarrow M^{p_2}_{\varphi_2}(\mathbb R^d), \end{equation} 
then
\begin{equation}\label{eq:radio-wMfab} 
\sup_{Q\in\mathcal Q} \frac{\varphi_2(Q)}{\varphi_1(Q)} <\infty. 
\end{equation} 
If, in addition, the hypotheses of Proposition~\ref{pro:regularity-cubes} and 
Theorem~\ref{thm:nece-Omega} are satisfied, then $p_2\le p_1$.

\item Assume that $p_2\le p_1$. If $w\in\mathcal A_{p_1/p_2}$ and \eqref{eq:radio-wMfab} holds, then the embedding \eqref{eb:wMfa-Mfb} holds. 
\end{enumerate} 
\end{theorem}

\begin{proof}
\emph{Step }1. We first prove (i). Since $w\in\mathcal A_\infty$, Proposition~\ref{pro:regularity-cubes}
implies that there exists a cube $Q_0\in\mathcal Q$ on which $w$ is equivalent
to a positive constant. Thus
\[
\|f\mid \mathcal M^{p_1}_{\varphi_1}(Q_0;w)\|
\sim
\|f\mid \mathcal M^{p_1}_{\varphi_1}(Q_0)\|.
\]
Therefore, applying Theorem~\ref{thm:nece-Omega} locally on $Q_0$, we obtain $p_1 \ge p_2$.

Next, testing the embedding \eqref{eb:wMfa-Mfb} on characteristic functions and
using Lemma~\ref{lem:1Qww}, we obtain
\[
\varphi_2(Q)\lesssim \varphi_1(Q),
\qquad Q\in\mathcal Q,
\]
and hence \eqref{eq:radio-wMfab} follows.

\emph{Step }2. We now prove (ii). Since $w\in\mathcal A_{p_1/p_2}$, there exists a constant
$C>0$ such that, for all cubes $Q$ and all measurable functions $f$,
\begin{equation}\label{eq:Ap1p2}
    |Q|^{-1/p_2}
    \Bigl(\int_Q |f(x)|^{p_2}\,dx\Bigr)^{1/p_2}
    \le
    C\,w(Q)^{-1/p_1}
    \Bigl(\int_Q |f(x)|^{p_1}w(x)\,dx\Bigr)^{1/p_1}.
\end{equation}
cf. \cite[pp.~195--197]{Ste93}.
Multiplying \eqref{eq:Ap1p2} by $\varphi_2(Q)$ and using
\eqref{eq:radio-wMfab}, we obtain
\begin{align*}
\varphi_2(Q)
\Bigl(\frac1{|Q|}\int_Q |f(x)|^{p_2}\,dx\Bigr)^{1/p_2}
&\lesssim
\varphi_2(Q)
\Bigl(\frac1{w(Q)}\int_Q |f(x)|^{p_1}w(x)\,dx\Bigr)^{1/p_1}
\\
&\le
\frac{\varphi_2(Q)}{\varphi_1(Q)}
\|f\mid \mathcal M^{p_1}_{\varphi_1}(\mathbb R^d;w)\|
\\
&\lesssim
\|f\mid \mathcal M^{p_1}_{\varphi_1}(\mathbb R^d;w)\|.
\end{align*}
Taking the supremum over all cubes $Q$ gives \eqref{eb:wMfa-Mfb}.
\end{proof}

\begin{remark}\label{rem:w-in-Ap1/p2}
    Note that if we assume that $w_{\alpha,\beta}\in \mathcal A_{p_1/p_2}$, then, by Theorem \ref{thm:wMfa-Mfb}, the embedding \eqref{eb:Mfwwab-Mfb} holds if and only if $p_2\le p_1$ and $\varphi_2(Q)\lesssim \varphi_1(Q)$ for all $Q\in\mathcal Q$.
    Thus the gap between Theorems \ref{thm:Mfwwab-suff} and \ref{thm:Mfwwab-nece} disappears.
\end{remark}

\begin{remark}\label{rem:wMfa-Mfb}
Let $0<p_2<p_1<\infty$. Assume that
\begin{equation}\label{eq:add-wMfa-Mfb}
\sup_{Q\in\mathcal Q}
\frac{\varphi_2(Q)}{\varphi_1(Q)}
w(Q)^{1/p_1}
|Q|^{-1/p_2}
\Bigl(
\int_Q w(x)^{-\frac{p_2}{p_1-p_2}}\,dx
\Bigr)^{\frac1{p_2}-\frac1{p_1}}
<\infty.
\end{equation}
Then the embedding \eqref{eb:wMfa-Mfb} holds.
Indeed, by H\"older's inequality,
\[
\Bigl(\int_Q |f(x)|^{p_2}\,dx\Bigr)^{1/p_2}
\le
\Bigl(\int_Q |f(x)|^{p_1}w(x)\,dx\Bigr)^{1/p_1}
\Bigl(
\int_Q w(x)^{-\frac{p_2}{p_1-p_2}}\,dx
\Bigr)^{\frac1{p_2}-\frac1{p_1}}.
\]
Multiplying this estimate by
$\varphi_2(Q)|Q|^{-1/p_2}$ and using
\eqref{eq:add-wMfa-Mfb} proves the assertion.

This sufficient condition does not require $w\in\mathcal A_\infty$.
It applies to every positive locally integrable weight for which the
quantity in \eqref{eq:add-wMfa-Mfb} is finite.
\end{remark}

Now we have the non-radial counterpart of \eqref{eb:Mfa-Mfb}.

\begin{corollary}\label{cor:Mfab-nr}
Let $0<p_i<\infty$ and $\varphi_i\in\mathcal G_{p_i}$,
$i=1,2$. Assume that $\varphi_1$ satisfies \eqref{eq:def-LC}--\eqref{eq:LRC-norad}. Then
\[
M^{p_1}_{\varphi_1}(\mathbb R^d)
\hookrightarrow
M^{p_2}_{\varphi_2}(\mathbb R^d)
\]
if and only if
\[
p_2\le p_1
\quad\text{and}\quad
\sup_{Q\in\mathcal Q}
\frac{\varphi_2(Q)}{\varphi_1(Q)}
<\infty.
\]
\end{corollary}

Next we consider the classical Morrey settings.

\begin{example}\label{ex:classical}
Let $0<p_i\le u_i<\infty$, $i=1,2$, and
\[
\varphi_1(Q)=w(Q)^{1/u_1},
\qquad
\varphi_2(Q)=|Q|^{1/u_2},
\qquad Q\in\mathcal Q.
\]
Then \eqref{eb:wMfa-Mfb} becomes
\[
\mathcal M_{p_1}^{u_1}(\mathbb R^d;w)
\hookrightarrow
M_{p_2}^{u_2}(\mathbb R^d).
\]
Condition \eqref{eq:radio-wMfab} becomes
\[
\sup_{Q\in\mathcal Q}
|Q|^{1/u_2}w(Q)^{-1/u_1}
<\infty.
\]
Thus Theorem \ref{thm:wMfa-Mfb} recovers Proposition~4.1 in \cite{HS17}.
\end{example}

Building on the preceding discussion, we obtain the following two-weight
sufficient condition.

\begin{corollary}\label{cor:Mfww-two-w}
Let $0<p_2<p_1<\infty$, $\varphi_i\in\mathcal G_{p_i}^{w_i}$ and $w_i\in\mathcal A_\infty$, $i=1,2$.
Assume that, for some $\gamma>1$ with conjugate exponent $\gamma'$, one has
\begin{align}\label{eq:two-weight-suff}
\sup_{Q\in\mathcal Q}
&\frac{\varphi_2(Q)w_1(Q)^{1/p_1}}
{\varphi_1(Q)w_2(Q)^{1/p_2}}
\Bigl(
\int_Q w_1(x)^{-\gamma\frac{p_2}{p_1-p_2}}\,dx
\Bigr)^{\frac1\gamma(\frac1{p_2}-\frac1{p_1})}
\nonumber\\
&\hspace{35mm}\times
\Bigl(
\int_Q w_2(x)^{\gamma'\frac{p_1}{p_1-p_2}}\,dx
\Bigr)^{\frac1{\gamma'}(\frac1{p_2}-\frac1{p_1})}
<\infty.
\end{align}
Then the embedding
\begin{equation*}
    \mathcal M^{p_1}_{\varphi_1}(\mathbb R^d;w_1)
    \hookrightarrow
    \mathcal M^{p_2}_{\varphi_2}(\mathbb R^d;w_2)
\end{equation*}
is continuous.
\end{corollary}

\begin{proof}
For a cube $Q$, by H\"older's inequality,
\[
\Bigl(\int_Q |f(x)|^{p_2}w_2(x)\,dx\Bigr)^{1/p_2}
\le
\Bigl(\int_Q |f(x)|^{p_1}w_1(x)\,dx\Bigr)^{1/p_1}
\Bigl(
\int_Q
w_1(x)^{-\frac{p_2}{p_1-p_2}}
w_2(x)^{\frac{p_1}{p_1-p_2}}
\,dx
\Bigr)^{\frac1{p_2}-\frac1{p_1}}.
\]
Applying H\"older's inequality once more to the last integral, with exponents
$\gamma$ and $\gamma'$, gives the two factors in
\eqref{eq:two-weight-suff}. Multiplying by $\varphi_2(Q)w_2(Q)^{-1/p_2}$ 
and using \eqref{eq:two-weight-suff}, we obtain
\[
\varphi_2(Q)
\Bigl(
\frac1{w_2(Q)}
\int_Q |f(x)|^{p_2}w_2(x)\,dx
\Bigr)^{1/p_2}
\lesssim
\|f\mid \mathcal M^{p_1}_{\varphi_1}(\mathbb R^d;w_1)\|.
\]
Taking the supremum over all cubes $Q$ proves the assertion.
\end{proof}

\begin{remark}\label{rem:Mfww-two-w}
Corollary~\ref{cor:Mfww-two-w} is the generalized Morrey-function analogue of the
two-weight sufficient conditions discussed in \cite[Remark~4.5]{HS17}.
In the classical case $\varphi_i(Q)\sim w_i(Q)^{1/u_i}$ with $p_i\le u_i$, $i=1,2$, 
condition \eqref{eq:two-weight-suff} reduces to the corresponding weighted
Morrey scale condition in \cite[p.~1082]{HS17}.
\end{remark}

We next consider the special case where the same weight appears on both sides.

\begin{corollary}\label{cor:same-weight}
Let $0<p_i<\infty$ and
$\varphi_i\in\mathcal G_{p_i}^{w}$, $i=1,2$. Let
$w\in\mathcal A_\infty$.

If
\[
p_2\le p_1
\qquad\text{and}\qquad
\sup_{Q\in\mathcal Q}
\frac{\varphi_2(Q)}{\varphi_1(Q)}
<\infty,
\]
then
\[
\mathcal M^{p_1}_{\varphi_1}(\mathbb R^d;w)
\hookrightarrow
\mathcal M^{p_2}_{\varphi_2}(\mathbb R^d;w).
\]

Conversely, if this embedding holds, then $\sup_{Q\in\mathcal Q}\frac{\varphi_2(Q)}{\varphi_1(Q)}<\infty$. 
If, in addition, the hypotheses of Proposition~\ref{pro:regularity-cubes} and 
Theorem~\ref{thm:nece-Omega} are satisfied, then $p_2\le p_1$.
\end{corollary}

\begin{proof}
The sufficiency follows directly from H\"older's inequality. 
For necessity, the condition $p_1\ge p_2$ follows by applying
Theorem~\ref{thm:nece-Omega} on a regularity cube of the weight $w$, whose
existence follows from Proposition~\ref{pro:regularity-cubes}. Finally, testing the
embedding on characteristic functions and using Lemma~\ref{lem:1Qww} yields $\varphi_2(Q)\lesssim \varphi_1(Q)$, $Q\in\mathcal Q$.
\end{proof}

Finally, we focus on the smallest Muckenhoupt weight class $\mathcal A_1$.

\begin{theorem}\label{thm:wMfa-Mfb'}
Let $0<p_1,p_2<\infty$, 
$\varphi_1\in\mathcal G_{p_1}^w$, 
$\varphi_2\in\mathcal G_{p_2}$, and $w\in\mathcal A_1$.

If $p_2\le p_1$ and 
\begin{equation}\label{eq:ratio-wMfab'}
\sup_{Q\in\mathcal Q}
\frac{\varphi_2(Q)}{\varphi_1(Q)}
<\infty,
\end{equation}
then
\[
\mathcal M^{p_1}_{\varphi_1}(\mathbb R^d;w)
\hookrightarrow
M^{p_2}_{\varphi_2}(\mathbb R^d).
\]

Conversely, if this embedding holds, then \eqref{eq:ratio-wMfab'} is necessary. 
If, in addition, the hypotheses of Proposition~\ref{pro:regularity-cubes} and Theorem~\ref{thm:nece-Omega} are satisfied, then $p_2\le p_1$.
\end{theorem}

\begin{proof}
We first prove sufficiency. Since $w\in\mathcal A_1$, there exists a constant
$C>0$ such that, for all cubes $Q$ and all nonnegative measurable functions 
$g$,
\begin{equation}\label{eq:pr-A1}
    \frac1{|Q|}\int_Q g(x)\,dx
    \le
    \frac{C}{w(Q)}
    \int_Q g(x)w(x)\,dx,
\end{equation}
cf. \cite[p.~197]{Ste93}. Applying the above estimate with $g=|f|^{p_1}$ we obtain
\[
\Bigl(\frac{1}{|Q|}\int_Q |f(x)|^{p_1}\,dx\Bigr)^{1/p_1}
\lesssim
\Bigl(\frac{1}{w(Q)}\int_Q |f(x)|^{p_1}w(x)\,dx\Bigr)^{1/p_1}.
\] 
Since $p_2\le p_1$, by H\"older's inequality we have
\[
\Bigl(
\frac1{|Q|}\int_Q |f(x)|^{p_2}\,dx
\Bigr)^{1/p_2}
\le
\Bigl(
\frac1{|Q|}\int_Q |f(x)|^{p_1}\,dx
\Bigr)^{1/p_1}.
\]
Combining the last two estimates, we obtain
\begin{align*}
\varphi_2(Q)
\Bigl(
\frac1{|Q|}\int_Q |f(x)|^{p_2}\,dx
\Bigr)^{1/p_2}
&\lesssim
\varphi_2(Q)
\Bigl(
\frac1{w(Q)}\int_Q |f(x)|^{p_1}w(x)\,dx
\Bigr)^{1/p_1}
\\
&\le
\frac{\varphi_2(Q)}{\varphi_1(Q)}
\|f\mid\mathcal M^{p_1}_{\varphi_1}(\mathbb R^d;w)\|
\\
&\lesssim
\|f\mid\mathcal M^{p_1}_{\varphi_1}(\mathbb R^d;w)\|.
\end{align*}
Taking the supremum over all cubes $Q$ proves the sufficiency.

We next prove necessity. By Lemma~\ref{lem:1Qww},
\[
\|\chi_Q\mid \mathcal M^{p_1}_{\varphi_1}(\mathbb R^d;w)\|
\sim
\varphi_1(Q),
\qquad
\|\chi_Q\mid M^{p_2}_{\varphi_2}(\mathbb R^d)\|
\sim
\varphi_2(Q),
\]
and hence the embedding implies
\[
\varphi_2(Q)\lesssim \varphi_1(Q),
\qquad Q\in\mathcal Q.
\]
Finally, the necessity of $p_1\ge p_2$ follows from
Proposition~\ref{pro:regularity-cubes} and Theorem~\ref{thm:nece-Omega}, applied on a
cube on which $w$ is equivalent to a positive constant.
\end{proof}

We conclude this section with two useful examples.

\begin{example}\label{ex:wMfab'}
\begin{enumerate}[\bfseries\upshape (i)]

\item
If $w\equiv1$, then $w\in\mathcal A_1$. Hence 
Theorem~\ref{thm:wMfa-Mfb'} reduces to the full unweighted
embedding criterion \eqref{eb:Mfa-Mfb}.

\item Let $0<p_i\le u_i<\infty$, $i=1,2$, and let
$\varphi_1(Q)= w(Q)^{1/u_1}$, $\varphi_2(Q)= |Q|^{1/u_2}$, $Q\in\mathcal Q$. 
Then condition \eqref{eq:ratio-wMfab'} becomes
\[
\sup_{Q\in\mathcal Q}
|Q|^{1/u_2}w(Q)^{-1/u_1}
<\infty.
\]
Thus Theorem~\ref{thm:wMfa-Mfb'} recovers 
Theorem~4.7 in \cite{HS17}.
\end{enumerate}
\end{example}

\section{Embeddings of Samko-type spaces}\label{sec:S}

In this section we study embeddings for Samko-type generalized weighted Morrey spaces $\mathfrak M_\varphi^p(\mathbb R^d;w)$, where the Morrey function $\varphi=\varphi(r)$ depends only on the scale. We first consider piecewise power weights and then general Muckenhoupt weights. The classical Samko-type weighted Morrey spaces are obtained by taking $\varphi(r)=r^{d/u}$, $0<p\le u<\infty$.

Since the Morrey functions considered in this section depend only on the scale, we can use balls to define the corresponding Morrey quasi-norms and the class $\mathcal B_{\varphi,p}$. 
The equivalence with the definitions using cubes follows from the fact that balls and cubes can contain one another up to constants, together with $\varphi\in\mathcal G_p$.

\subsection{Embeddings with piecewise power weights}\label{ssec:S-pw}

We start with the following criterion for the membership of the weight $w_{\alpha,\beta}$ in the class $\mathcal B_{\varphi,p}$.

\begin{lemma}\label{lem:wab-Bphip}
Let $0<p<\infty$, $\varphi\in\mathcal G_p$ and 
$w_{\alpha,\beta}$ be given by \eqref{eq:def-wab}, where $\alpha,\beta>-d$. Then
\[
w_{\alpha,\beta}\in\mathcal B_{\varphi,p}
\]
if and only if
\begin{equation}\label{eq:wab-in-Bphip}
\begin{cases}
\varphi(R_0)R_0^{\alpha/p}
\lesssim
\varphi(R_1)R_1^{\alpha/p},
& 0<R_0\le R_1<1,\\[1mm]
\varphi(R_0)R_0^{\alpha/p}
\lesssim
\varphi(R_1)R_1^{\beta/p},
& 0<R_0<1\le R_1,\\[1mm]
\varphi(R_0)R_0^{\beta/p}
\lesssim
\varphi(R_1)R_1^{\beta/p},
& 1\le R_0\le R_1.
\end{cases}
\end{equation}
The implicit constants are independent of $R_0$ and $R_1$.
\end{lemma}

\begin{proof}
For a ball $B=B(a,R)$, set
\[
\Phi(B)
:=
\varphi(R)
\Bigl(\frac{w_{\alpha,\beta}(B)}{|B|}\Bigr)^{1/p}.
\]
By definition, $w_{\alpha,\beta}\in\mathcal B_{\varphi,p}$ if and only if
$\Phi(B_0)\lesssim\Phi(B_1)$, $B_0\subset B_1$.

Note that the cases $0<R_0\le R_1<1$ and $1\le R_0\le R_1$
follow from the same arguments as those given in
\cite[Example~2.3]{Nak16}. 
It remains to consider the middle case $0<R_0<1\le R_1$.

\emph{Step }1. We first prove the sufficiency. 
Consider balls $B_i=B(a_i,R_i)$,  $i=0,1$, such that $B_0\subset B_1$ and $0<R_0<1\le R_1$. 

\emph{Step }1.1. We first consider the case $|a_1|>2R_1$. Since $R_1\ge1$, we have
$|x|\sim|a_1|>1$ for all $x\in B_1$.
So, $w_{\alpha,\beta}(x)\sim |a_1|^\beta$ on $B_1$ and also on $B_0$. 
Thus 
\[
\frac{w_{\alpha,\beta}(B_0)}{|B_0|}
\sim
\frac{w_{\alpha,\beta}(B_1)}{|B_1|}
\sim |a_1|^\beta.
\]
Since $\varphi\in\mathcal G_p$ is almost increasing and
$R_0\le R_1$, we obtain
\[
\Phi(B_0)\lesssim\Phi(B_1).
\]

\emph{Step }1.2. We next assume that $|a_1|\le2R_1$. In this case, $B_1$ contains a
ball $E$ of radius comparable to $R_1$ on which 
$|x|\sim R_1$. 
Therefore, $|E|\sim R_1^d$ and $w_{\alpha,\beta}(x)\sim R_1^\beta$ for $x\in E$.
It follows that
\[
w_{\alpha,\beta}(B_1)
\ge
\int_E w_{\alpha,\beta}(x)\,dx
\gtrsim
R_1^{d+\beta}.
\]
Since $|B_1|\sim R_1^d$, we conclude that
\[
\frac{w_{\alpha,\beta}(B_1)}{|B_1|}
\gtrsim R_1^\beta,
\]
and hence
\begin{equation}\label{eq:Phi-B1-lower}
\Phi(B_1)
\gtrsim
\varphi(R_1)R_1^{\beta/p}.
\end{equation}

\emph{Step }1.3. For the smaller ball $B_0$, standard estimates for the piecewise power weight give
\begin{equation}\label{eq:wab-B0-upper}
\frac{w_{\alpha,\beta}(B_0)}{|B_0|}
\lesssim
\begin{cases}
R_0^\alpha,
& |a_0|\le2R_0,\\[1mm]
|a_0|^\alpha,
& 2R_0<|a_0|<1,\\[1mm]
|a_0|^\beta,
& |a_0|\ge1.
\end{cases}
\end{equation}
We treat these three cases separately.

Case (i): $|a_0|\le2R_0$. By \eqref{eq:wab-B0-upper}, the second
condition in \eqref{eq:wab-in-Bphip} and
\eqref{eq:Phi-B1-lower}, 
\[
\Phi(B_0)
\lesssim
\varphi(R_0)R_0^{\alpha/p}
\lesssim
\varphi(R_1)R_1^{\beta/p}
\lesssim
\Phi(B_1).
\]

Case (ii): $2R_0<|a_0|<1$.
If $\alpha<0$, then $|a_0|^\alpha\lesssim R_0^\alpha$, 
and therefore
\[
\Phi(B_0)
\lesssim
\varphi(R_0)R_0^{\alpha/p}
\lesssim
\varphi(R_1)R_1^{\beta/p}
\lesssim
\Phi(B_1).
\]
If $\alpha\ge0$, then $|a_0|^\alpha\le1$. Since $\varphi$ is almost
increasing and $R_0<1$, the third condition in
\eqref{eq:wab-in-Bphip}, applied with the radii $1$ and $R_1$, gives
\[
\Phi(B_0)
\lesssim
\varphi(R_0)
\lesssim
\varphi(1)
\lesssim
\varphi(R_1)R_1^{\beta/p}
\lesssim
\Phi(B_1).
\]

Case (iii): $|a_0|\ge1$. 
Since $B_0\subset B_1$ and $|a_1|\le2R_1$, we have
\[
|a_0|
\le |a_1|+R_1
\le3R_1.
\]
If $\beta\ge0$, then $|a_0|^\beta\lesssim R_1^\beta$. 
Using the almost increasing property of $\varphi$, we obtain
\[
\Phi(B_0)
\lesssim
\varphi(R_0)|a_0|^{\beta/p}
\lesssim
\varphi(R_1)R_1^{\beta/p}
\lesssim
\Phi(B_1).
\]
If $\beta<0$, then $|a_0|^\beta\le1$. Hence, by the
almost increasing property of $\varphi$ and the third condition in
\eqref{eq:wab-in-Bphip},
\[
\Phi(B_0)
\lesssim
\varphi(R_0)
\lesssim
\varphi(1)
\lesssim
\varphi(R_1)R_1^{\beta/p}
\lesssim
\Phi(B_1).
\]
This proves
\[
\Phi(B_0)\lesssim\Phi(B_1)
\]
Thus $w_{\alpha,\beta}\in\mathcal B_{\varphi,p}$.

\emph{Step }2. We next prove the necessity. Let $w_{\alpha,\beta}\in\mathcal B_{\varphi,p}$.
For $R>0$, the averages over balls centered at the origin satisfy
\begin{equation}\label{eq:wab-centered-average}
\frac{w_{\alpha,\beta}(B(0,R))}{|B(0,R)|}
\sim
\begin{cases}
R^\alpha, & 0<R<1,\\
R^\beta, & R\ge1.
\end{cases}
\end{equation}
Indeed, the first estimate follows directly by integrating
$|x|^\alpha$ over $B(0,R)$, while the second follows from
$\alpha,\beta>-d$ and the definition of $w_{\alpha,\beta}$.

Let $B_0=B(0,R_0)$ and $B_1=B(0,R_1)$, 
where $0<R_0\le R_1$. Since $B_0\subset B_1$, the defining
almost-monotonicity of $\mathcal B_{\varphi,p}$ gives $\Phi(B_0)\lesssim\Phi(B_1)$.
If $0<R_0<1\le R_1$, then \eqref{eq:wab-centered-average} yields
\[
\varphi(R_0)R_0^{\alpha/p}
\lesssim
\varphi(R_1)R_1^{\beta/p},
\]
and the proof is complete.
\end{proof}

\begin{remark}\label{rem:Bphip-a=b}
Let 
$\alpha=\beta$, then the weight becomes the usual power weight
$w(x)=|x|^\alpha$, and
\eqref{eq:wab-in-Bphip} reduces to
\[
\varphi(R_0)R_0^{\alpha/p}
\lesssim
\varphi(R_1)R_1^{\alpha/p},
\qquad 0<R_0\le R_1,
\]
which is precisely \cite[Example~2.3]{Nak16}, cf. also Example~\ref{ex:Bp}.
\end{remark}

For the classical case $\varphi(t)=t^{d/u}$ with $p\le u$, 
we obtain a simple criterion similar to Example~\ref{ex:Bp}.

\begin{corollary}\label{cor:wab-Bup}
Let $0<p\le u<\infty$, and $w_{\alpha,\beta}$ be given by \eqref{eq:def-wab}, where $\alpha,\beta>-d$. Then 
\[
w_{\alpha,\beta}\in\mathcal B_{u,p}\quad\text{if and only if }\quad
\alpha,\beta\ge -\frac{dp}{u}.
\]
\end{corollary}

We now derive a necessary condition for the desired embedding.

\begin{theorem}\label{thm:Mf1wab-nece}
Let $0<p_i<\infty$, $\varphi_i\in\mathcal G_{p_i}$, $i=1,2$, and $w_{\alpha,\beta}\in\mathcal B_{\varphi_1,p_1}$ is given by \eqref{eq:def-wab}, where $\alpha,\beta>-d$.
Assume that $\varphi_1$ satisfies \eqref{eq:LRC-rad}.
If
\begin{equation}\label{eb:Mf1wab-Mfb}
\mathfrak M^{p_1}_{\varphi_1}(\mathbb R^d;w_{\alpha,\beta})
\hookrightarrow
M^{p_2}_{\varphi_2}(\mathbb R^d),
\end{equation}
then $p_2\le p_1$, $\beta\ge 0$, and
\begin{equation}\label{eq:nece-ball-cond}
\sup_{B=B(a,R)}
\frac{\varphi_2(R)}{\varphi_1(R)}
\Bigl(
\frac{w_{\alpha,\beta}(B)}{|B|}
\Bigr)^{-1/p_1}
<\infty.
\end{equation}
\end{theorem}

\begin{proof}
The necessity of $p_2\le p_1$ follows from Remark~\ref{rem:regularity-cubes-applications} and Proposition~\ref{pro:rad-Omega}. 

By Lemma~\ref{lem:1Q1w}, the assumption
$w_{\alpha,\beta}\in\mathcal B_{\varphi_1,p_1}$ gives
\[
\|\chi_B\mid
\mathfrak M^{p_1}_{\varphi_1}(\mathbb R^d;w_{\alpha,\beta})
\|
\sim
\varphi_1(R)
\Bigl(
\frac{w_{\alpha,\beta}(B)}{|B|}
\Bigr)^{1/p_1}.
\]
On the other hand,
\[
\|\chi_B\mid M^{p_2}_{\varphi_2}(\mathbb R^d)\|
\sim
\varphi_2(R).
\]
Testing \eqref{eb:Mf1wab-Mfb} on $\chi_B$ gives
\eqref{eq:nece-ball-cond}.

To derive $\beta\ge0$, take $B_a=B(a,1)$, $|a|\to\infty$.
Then $\frac{w_{\alpha,\beta}(B_a)}{|B_a|} \sim |a|^\beta$.
The condition \eqref{eq:nece-ball-cond} gives 
\[ 
\frac{\varphi_2(1)}{\varphi_1(1)} |a|^{-\beta/p_1}\lesssim1. 
\] 
Since $\varphi_1(1),\varphi_2(1)>0$, this is impossible when $\beta<0$. 
\end{proof}

We next prove the sufficient result. 

\begin{proposition}\label{pro:gen-suff-S}
Let $0<p_2<p_1<\infty$, 
$\varphi_1\in\mathcal G_{p_1}$, $\varphi_2\in\mathcal G_{p_2}$ and $w$ be a weight.
Assume that $\sigma=\frac{p_2}{p_1-p_2}$ and 
\begin{equation}\label{eq:gen-holder-cond}
\sup_{B=B(a,R)}
\frac{\varphi_2(R)}{\varphi_1(R)}
\Bigl(
\frac1{|B|}
\int_B w(x)^{-\sigma}\,dx
\Bigr)^{\frac1{p_2}-\frac1{p_1}}
<\infty.
\end{equation}
Then
\[
\mathfrak M^{p_1}_{\varphi_1}(\mathbb R^d;w)
\hookrightarrow
M^{p_2}_{\varphi_2}(\mathbb R^d).
\]
\end{proposition}

\begin{proof}
By H\"older's inequality and \eqref{eq:gen-holder-cond},
\begin{align*}
\varphi_2(R)
\Bigl(
\frac1{|B|}
\int_B |f(x)|^{p_2}\,dx
\Bigr)^{1/p_2}
&\le
\frac{\varphi_2(R)}{\varphi_1(R)}
\Bigl(
\frac1{|B|}
\int_B w(x)^{-\sigma}\,dx
\Bigr)^{\frac1{p_2}-\frac1{p_1}}
\\
&\quad\times
\varphi_1(R)
\Bigl(
\frac1{|B|}
\int_B |f(x)|^{p_1}w(x)\,dx
\Bigr)^{1/p_1}\\
&\lesssim \|f\mid \mathfrak M^{p_1}_{\varphi_1}(\mathbb R^d; w)\|.
\end{align*}
Taking the supremum over all balls gives the desired result.
\end{proof}

To obtain the second main result of this subsection, we need to estimate the averages of $w_{\alpha,\beta}^{-\sigma}$ on a ball $B$.

\begin{lemma}\label{lem:negative-power}
Let $\sigma=\frac{p_2}{p_1-p_2}>0$, $\sigma\alpha < d$ and $\beta\ge 0$. 
Then there exists a constant $C>0$ such that for all balls $B=B(a,R)$,
\begin{align}\label{eq:neg-ave-small}
    \frac1{|B|}
\int_B w_{\alpha,\beta}(x)^{-\sigma}\,dx
\le
C R^{-\sigma \alpha_+},
\qquad 0<R\le1,
\end{align}
and 
\begin{align}\label{eq:neg-ave-large}
    \frac1{|B|}
\int_B w_{\alpha,\beta}(x)^{-\sigma}\,dx
\le
C A_{\beta,\sigma}(R),
\qquad R\ge1,
\end{align}
where
\begin{equation}\label{eq:def-A-beta-s}
A_{\beta,\sigma}(R)
:=
\begin{cases}
R^{-\sigma \beta},&\sigma \beta < d,\\
R^{-d}\log(e+R),&\sigma \beta = d,\\
R^{-d},&\sigma \beta > d.
\end{cases}
\end{equation}
\end{lemma}

\begin{proof}
\emph{Step }1. First assume $0<R\le1$. If $|a|\le 2R$, then
$B(a,R)\subset B(0,3R)$. Hence
\[
\int_{B(a,R)} w_{\alpha,\beta}(x)^{-\sigma}\,dx
\le
\int_{B(0,3R)\cap\{|x|< 1\}} |x|^{-\sigma \alpha}\,dx
+\int_{B(0,3R)\cap\{|x|\ge 1\}} |x|^{-\sigma \beta}\,dx.
\]
Since $\beta\ge 0$, the second integral is bounded by $CR^d$.
If $\alpha\ge0$, the assumption $\sigma \alpha<d$ yields
\[
\int_{B(0,3R)\cap\{|x|< 1\}} |x|^{-\sigma \alpha}\,dx
\lesssim R^{d-\sigma \alpha}<\infty.
\]
If $-d<\alpha<0$, then $-\sigma\alpha>0$. Since $R\le 1$, we have $|x|^{-\sigma \alpha}\le (3R)^{-\sigma\alpha}\lesssim 1$ on
$B(0,3R)$, so the integral is bounded by $R^d$. Dividing by
$|B|\sim R^d$, we obtain
\[
\frac1{|B|}\int_B w_{\alpha,\beta}(x)^{-\sigma}\,dx
\lesssim R^{-\sigma \alpha_+}.
\]

If $|a|>2R$, then $|x|\sim |a|$ on $B(a,R)$. If $|a|<1$, then
$w_{\alpha,\beta}(x)^{-\sigma}\sim |a|^{-\sigma\alpha}$. For $\alpha\ge0$, since
$|a|>2R$, this is bounded by $R^{-\sigma \alpha}$. For $-d<\alpha<0$, it is
bounded by $1$. If $|a|\ge1$, then either $B(a,R)\subset \{|x|\ge 1\}$ implies $w_{\alpha,\beta}(x)^{-\sigma}\sim |a|^{-\sigma\beta}\le1$ by $\beta\ge0$, or $B(a,R)$ intersects the unit sphere and $|x|\sim 1$ on $B(a,R)$ yield $w_{\alpha,\beta}(x)^{-\sigma}\lesssim 1$. This proves \eqref{eq:neg-ave-small}.

\emph{Step }2. Now assume $R\ge1$. If $|a|\le2R$, then
$B(a,R)\subset B(0,3R)$. Therefore
\[
\int_B w_{\alpha,\beta}(x)^{-\sigma}\,dx
\le
\int_{B(0,1)} |x|^{-\sigma\alpha}\,dx
+
\int_{1\le |x|<3R} |x|^{-\sigma\beta}\,dx.
\]
The first integral is finite by $\sigma \alpha<d$. For the second integral,
\[
\int_1^{3R} \rho^{d-1-\sigma\beta}\,d\rho
\lesssim
\begin{cases}
R^{d-\sigma \beta},&\sigma \beta<d,\\
\log(e+R),&\sigma \beta=d,\\
1,&\sigma \beta>d.
\end{cases}
\]
Dividing by $|B|\sim R^d$, we obtain
\eqref{eq:neg-ave-large}.

If $|a|>2R$, then $|x|\sim |a|$ on $B(a,R)$, and since
$R\ge1$, we have $|a|>2$. Thus
\[
w_{\alpha,\beta}(x)^{-\sigma}\sim |a|^{-\sigma\beta}\le R^{-\sigma\beta}.
\]
If $\sigma\beta<d$, this gives $R^{-\sigma\beta}$. If $\sigma\beta\ge d$, then
$R^{-\sigma\beta}\le R^{-d}$, which is dominated by the corresponding right-hand
side in \eqref{eq:def-A-beta-s}. The proof is complete.
\end{proof}

\begin{remark}\label{rem:beta<0}
    The assumption $\beta\ge 0$ is essential. Indeed, if $\beta<0$, then for fixed $R>0$ and $|a|\to\infty$, 
    \[ 
    \frac1{|B(a,R)|} \int_{B(a,R)}w_{\alpha,\beta}(x)^{-\sigma}\,dx \sim |a|^{-\sigma\beta}\to\infty. \] 
    Therefore, no estimate depending only on the radius $R$ can hold uniformly over all centers $a$. 
\end{remark}

We can now state the main sufficient condition.

\begin{theorem}\label{thm:Mf1wab-suff}
Let $0<p_2<p_1<\infty$, 
$\varphi_1\in\mathcal G_{p_1}$ and  
$\varphi_2\in\mathcal G_{p_2}$. 
Let $\sigma=p_2/(p_1-p_2)$ and $w_{\alpha,\beta}$ be
given by \eqref{eq:def-wab}, where $\alpha>-d$ and $\beta\ge 0$. 
Assume that $\sigma\alpha < d$,
\begin{align}
    &\sup_{0<R\le1}
\frac{\varphi_2(R)}{\varphi_1(R)}
R^{-{\alpha_+}/{p_1}}
<\infty, \label{eq:suff-small-cond}\\ 
\text{and}\qquad
&\sup_{R\ge1}
\frac{\varphi_2(R)}{\varphi_1(R)}
A_{\beta,\sigma}(R)^{1/{p_2}-1/{p_1}}
<\infty, \label{eq:suff-large-cond}
\end{align}
where $A_{\beta,\sigma}$ is defined by \eqref{eq:def-A-beta-s}.
Then
\[
\mathfrak M^{p_1}_{\varphi_1}(\mathbb R^d;w_{\alpha,\beta})
\hookrightarrow
M^{p_2}_{\varphi_2}(\mathbb R^d).
\]
\end{theorem}

\begin{proof}
Applying Proposition~\ref{pro:gen-suff-S}, Lemma~\ref{lem:negative-power} and the fact that $\sigma(\frac{1}{p_2}-\frac{1}{p_1})=\frac{1}{p_1}$, we obtain Theorem~\ref{thm:Mf1wab-suff}.
\end{proof}

\begin{remark}\label{rem:gap-Mf1wab}
The difference between the necessary and sufficient conditions above comes
from the different weight quantities detected by the two arguments.
Testing the embedding on characteristic functions yields the necessary
condition
\[
\frac{\varphi_2(R)}{\varphi_1(R)}
\Bigl(\frac{w(B)}{|B|}\Bigr)^{-1/p_1}\lesssim1,
\]
which involves the average of $w$ over $B$. On the other hand, the
H\"older argument used for sufficiency requires
\[
\frac{\varphi_2(R)}{\varphi_1(R)}
\Bigl(\frac1{|B|}\int_Bw(x)^{-\sigma}\,dx\Bigr)^{
1/p_2-1/p_1}
\lesssim1.
\]
Thus the sufficient condition depends on the average of a negative power of the weight, whereas the necessary condition depends on the average of the weight itself. These two quantities need not be comparable unless an additional condition is imposed on the weight. This explains the gap between the necessary and sufficient conditions.
\end{remark}

\begin{remark}\label{rem:Mf1wab-a=b}

Let $\alpha=\beta$. Then $w_{\alpha,\beta}(x)=|x|^\alpha$. Under the assumptions of
Theorem~\ref{thm:Mf1wab-suff}, we then have
$0\le\alpha<d/\sigma$. Consequently,
\[
A_{\alpha,\sigma}(R)=R^{-\sigma\alpha},
\]
and conditions \eqref{eq:suff-small-cond} and
\eqref{eq:suff-large-cond} combine into
\[
\sup_{R>0}
\frac{\varphi_2(R)}{\varphi_1(R)}
R^{-\alpha/p_1}<\infty.
\]
This also agrees with the necessary condition in Theorem~\ref{thm:Mf1wab-nece}.

In particular, when $\alpha=\beta=0$, one has
$w_{\alpha,\beta}\equiv 1$ and therefore $\mathfrak M_\varphi^p(\mathbb R^d;1)
    =
    M_\varphi^p(\mathbb R^d)$.
Thus, for $p_2<p_1$, the necessary condition
\eqref{eq:nece-ball-cond} and the sufficient conditions
\eqref{eq:suff-small-cond}--\eqref{eq:suff-large-cond} reduce to the same
condition
\[
    \sup_{R>0}\frac{\varphi_2(R)}{\varphi_1(R)}<\infty,
\]
while the endpoint case $p_2=p_1$ is also covered by the
unweighted embedding criterion \eqref{eb:Mfa-Mfb}.
\end{remark}

\begin{remark}\label{rem:p_1=p_2}
A direct estimate at the endpoint $p_1=p_2=p$ yields an $L^\infty$-type condition on $w^{-1}$ on each ball.
That is, 
\[
\Bigl(\int_B |f(x)|^p\,dx\Bigr)^{1/p}
\le
\Bigl(\operatorname*{ess\,sup}_{x\in B} w(x)^{-1}\Bigr)^{1/p}
\Bigl(\int_B |f(x)|^p w(x)\,dx\Bigr)^{1/p}.
\]
Hence a sufficient condition for
\[
\mathfrak M_p^{\varphi_1}(\mathbb R^d;w)
\hookrightarrow
M_p^{\varphi_2}(\mathbb R^d)
\]
is
\[
\sup_{B=B(a,r)}
\frac{\varphi_2(r)}{\varphi_1(r)}
\Bigl(
\operatorname*{ess\,sup}_{x\in B} w(x)^{-1}
\Bigr)^{1/p}
<\infty.
\]
For the weight $w_{\alpha,\beta}$, this requires in particular $\alpha\le0\le\beta$.

If $p_2>p_1$, the embedding cannot hold in general. Indeed, on every ball
away from the origin the weight $w_{\alpha,\beta}$ is equivalent to a
positive constant, and the embedding would imply the local embedding
$L^{p_1}\hookrightarrow L^{p_2}$, which is false for $p_2>p_1$.
\end{remark}

Next we consider two classical consequences of the preceding results. 

\begin{example}\label{ex:Lwaba-Lb}
    Let $\varphi_i(R)=R^{d/p_i}$, $i=1,2$. Then the embedding \eqref{eb:Mf1wab-Mfb} reads as
    \begin{equation}\label{eb:Lwaba-Lb}
    L_{p_1}(\mathbb R^d; w_{\alpha,\beta})
    \hookrightarrow
    L_{p_2}(\mathbb R^d).
    \end{equation}
    Conditions~\eqref{eq:nece-ball-cond},~\eqref{eq:suff-small-cond},~\eqref{eq:suff-large-cond} and Remark~\ref{rem:p_1=p_2} imply
    \[
    p_1\ge p_2 \quad \text{and} \quad
    \begin{cases}
        -d<\alpha \le 0 \le \beta, & \text{if $p_1=p_2$,}\\ 
        1+\frac{\alpha}{d}<\frac{p_1}{p_2}<1+\frac{\beta}{d},&\text{if $p_1>p_2$.}
    \end{cases}
    \]
    This agrees with the weighted Lebesgue
embedding result in \cite[Lemma~3.1]{HS17}. Note that for the endpoint
$p_1=p_2$ the embedding \eqref{eb:Lwaba-Lb} holds if and only if  $w_{\alpha,\beta}^{-1} \in L^\infty(\mathbb R^d)$, or equivalently, $\operatorname*{ess\,inf}_{x\in\mathbb R^d}w_{\alpha,\beta}(x)>0$. This implies $\alpha \le 0 \le \beta$, cf. also Remark~\ref{rem:p_1=p_2}.
\end{example}

Now we consider the classical Morrey spaces. Recall that $w_{\alpha,\beta} \in \mathcal B_{u,p}$ if and only if  $\alpha,\beta\ge -dp/u$, cf. Corollary~\ref{cor:wab-Bup}. Note that if $p<u$ one always has $\alpha,\beta\ge -dp/u > -d$.

\begin{proposition}
\label{pro:M1wab-nece}
Let $0<p_i< u_i<\infty$, $i=1,2$, and $w_{\alpha,\beta}$ be given by \eqref{eq:def-wab}, where $\alpha,\beta\ge -dp_1/u_1$.
If
\[
\mathfrak M_{p_1}^{u_1}(\mathbb R^d;w_{\alpha,\beta})
\hookrightarrow
M_{p_2}^{u_2}(\mathbb R^d),
\]
then
\[
p_2\le p_1
\quad\text{and}\quad
\begin{cases}
  -\frac{p_1d}{u_1}\le  \alpha
\le
p_1d\bigl(\frac1{u_2}-\frac1{u_1}\bigr),\\[2mm] 
\beta
\ge
\max\Bigl\{
0,\,
p_1 d\bigl(\frac1{u_2}-\frac1{u_1}\bigr)
\Bigr\}.
\end{cases}
\]
\end{proposition}

\begin{proof}
Let $\varphi_i(R)=R^{d/u_i}$ for $p_i< u_i$ and $i=1,2$. 
Set
\[
\Delta:=d\Bigl(\frac1{u_2}-\frac1{u_1}\Bigr).
\]
Then ${\varphi_2(R)}/{\varphi_1(R)}
=
R^\Delta$.
By Theorem~\ref{thm:Mf1wab-nece}, we have $p_2\le p_1$ and
\[
\sup_{B=B(a,R)}
R^\Delta
\Bigl(
\frac{w_{\alpha,\beta}(B)}{|B|}
\Bigr)^{-1/p_1}
<\infty.
\]
If $0<R\le1$, $R^{\Delta-\alpha/p_1}\lesssim1$; if $R\ge1$, $R^{\Delta-\beta/p_1}\lesssim1$. 
Hence
\[
\alpha\le p_1\Delta,
\qquad
\beta\ge p_1\Delta.
\]
Since $w_{\alpha,\beta}\in \mathcal B_{u_1,p_1}$ and $\beta\ge 0$, we have $-p_1d/u_1\le \alpha \le p_1\Delta$ and $\beta\ge \max\{0,p_1\Delta\}$.
\end{proof}

\begin{proposition}
\label{pro:M1wab-suff}
Let $0<p_i< u_i<\infty$, $i=1,2$, and $p_2<p_1$. 
Set
\[
\Delta:=d\Bigl(\frac1{u_2}-\frac1{u_1}\Bigr),
\qquad
\Lambda:=d\Bigl(\frac1{p_2}-\frac1{p_1}\Bigr),
\qquad
\sigma:=\frac{p_2}{p_1-p_2}.
\]
Let $w_{\alpha,\beta}$ be given by \eqref{eq:def-wab} with
$\alpha\ge -p_1d/u_1$ and $\beta\ge 0$.
Assume that $\sigma\alpha<d$, $\alpha_+\le p_1\Delta$ and 
\[
\begin{cases}
    \text{either}\qquad& \Delta<\Lambda
\quad\text{and}\quad
\beta\ge p_1\Delta,\\ 
\text{or}\qquad
& \Delta=\Lambda
\quad\text{and}\quad
\beta>p_1\Delta.
\end{cases}
\]
Then
\[
\mathfrak M_{p_1}^{u_1}(\mathbb R^d;w_{\alpha,\beta})
\hookrightarrow
M_{p_2}^{u_2}(\mathbb R^d).
\]
\end{proposition}

\begin{proof}
We apply Theorem~\ref{thm:Mf1wab-suff} with $\varphi_i(R)=R^{d/u_i}$, $p_i<u_i$, $i=1,2$.
Then ${\varphi_2(R)}/{\varphi_1(R)} = R^\Delta$.
Condition \eqref{eq:suff-small-cond} becomes
\[
\sup_{0<R\le1}
R^{\Delta-\alpha_+/p_1}
<\infty,
\]
which is equivalent to
\[
\alpha_+\le p_1\Delta.
\]

For the condition \eqref{eq:suff-large-cond}, 
\[
\sup_{R\ge1}
R^\Delta A_{\beta,\sigma}(R)^{1/p_2-1/p_1}
<\infty.
\]
Consider three cases in $A_{\beta,\sigma}$ defined by \eqref{eq:def-A-beta-s}.
If $\beta \sigma<d$, then $\sup_{R\ge1}R^{\Delta-\beta/p_1}<\infty$, namely, $\beta\ge p_1\Delta$. 
If $\beta \sigma=d$, then $\beta=p_1\Lambda$, and $\sup_{R\ge1}
R^{\Delta-\Lambda}
\bigl(\log(e+R)\bigr)^{1/p_2-1/p_1}
<\infty$.
This holds precisely when $\Delta<\Lambda$.
If $\beta \sigma>d$, then
$
\sup_{R\ge1}R^{\Delta-\Lambda}<\infty
$,
that is, $\Delta\le\Lambda$.

Thus we conclude that 
if $\Delta<\Lambda$, then $\beta\ge p_1\Delta$ is admissible, while if
$\Delta=\Lambda$, the logarithmic borderline is excluded and one needs
$\beta>p_1\Delta$. This proves the assertion.
\end{proof}

\begin{remark}\label{rem:u=p}
The Lebesgue case corresponds to $u_i=p_i$. Then
\[
\Delta=\Lambda=d\Bigl(\frac1{p_2}-\frac1{p_1}\Bigr).
\]
Thus Proposition~\ref{pro:M1wab-suff} implies, for
$p_2<p_1$,
\[
\sigma\alpha<d,
\qquad
\beta>p_1\Delta
=
d\Bigl(\frac{p_1}{p_2}-1\Bigr),
\]
which is equivalent to 
\[
1+\frac{\alpha}{d}
<
\frac{p_1}{p_2}
<
1+\frac{\beta}{d}.
\]
The endpoint $p_1=p_2$ remains separate and is described in
Example~\ref{ex:Lwaba-Lb}.
\end{remark}

\subsection{Embeddings with general Muckenhoupt weights}\label{ssec:S-w}

We now replace the specific weight $w_{\alpha,\beta}$ by a general Muckenhoupt weight in the source space. 

\begin{theorem}\label{thm:1Mfa-Mfb}
Let $0<p_i<\infty$ and $\varphi_i\in\mathcal G_{p_i}$, $i=1,2$. Consider the embedding 
\begin{align}\label{eb:1Mfa-Mfb}
    \mathfrak M^{p_1}_{\varphi_1}(\mathbb R^d; w)\hookrightarrow
    M^{p_2}_{\varphi_2}(\mathbb R^d).
\end{align}
Then the following assertions hold.
\begin{enumerate}[\bfseries\upshape  (i)]
\item Assume that $\varphi_1$ satisfies \eqref{eq:LRC-rad} and $w\in\mathcal A_\infty$ satisfies the assumptions of Proposition~\ref{pro:regularity-cubes}.  
If $w\in\mathcal{A}_\infty\cap\mathcal{B}_{\varphi_1,p_1}$ and the embedding \eqref{eb:1Mfa-Mfb} holds, then $p_1\ge p_2$ and 
\begin{equation}\label{eq:1Mfab-con}
    \sup_{Q\in\mathcal Q}
    \frac{\varphi_2(\ell(Q))}{\varphi_1(\ell(Q))}
    \Big(\frac{w(Q)}{|Q|}\Big)^{-1/p_1}<\infty.
\end{equation}

\item If $p_2\le p_1$, $w\in\mathcal A_{p_1/p_2}$, and \eqref{eq:1Mfab-con} holds, then the embedding \eqref{eb:1Mfa-Mfb} is valid.

\item If $w\in\mathcal A_1\cap\mathcal{B}_{\varphi_1,p_1}$, then the embedding \eqref{eb:1Mfa-Mfb} holds if and only if $p_2\le p_1$ and \eqref{eq:1Mfab-con} holds.
\end{enumerate}
\end{theorem}

\begin{proof}
We first show (i). The necessity of \eqref{eq:1Mfab-con} follows from Lemma~\ref{lem:1Q1w}. Meanwhile, by Propositions~\ref{pro:regularity-cubes} and~\ref{pro:rad-Omega}, we conclude that the embedding \eqref{eb:1Mfa-Mfb} implies $p_1\geq p_2$.

For (ii), assume that $p_2\le p_1$, $w\in\mathcal A_{p_1/p_2}$, and \eqref{eq:1Mfab-con} holds. Then by \eqref{eq:Ap1p2}, for all cubes $Q\in \mathcal Q$,
\begin{align*}
\frac{\varphi_2(\ell(Q))}{|Q|^{1/p_2}}
\Big(\int_Q |f(x)|^{p_2}\,dx\Big)^{1/p_2}
&\lesssim
\frac{\varphi_2(\ell(Q))}{\varphi_1(\ell(Q))}
\Big(\frac{w(Q)}{|Q|}\Big)^{-1/p_1}
\frac{\varphi_1(\ell(Q))}{|Q|^{1/p_1}}
\Big(\int_Q |f(x)|^{p_1}w(x)\,dx\Big)^{1/p_1}\\
&\lesssim \|f\mid \mathfrak M^{p_1}_{\varphi_1}(\mathbb R^d; w)\|.
\end{align*}
Taking the supremum over all cubes $Q$ yields the desired embedding.

    To prove (iii), it suffices to prove the sufficiency of \eqref{eq:1Mfab-con}, since its necessity has already been established in (i). 
    Note that $w\in\mathcal A_1$ implies \eqref{eq:pr-A1}, thus we obtain
    \begin{align*}
        \frac{\varphi_2(\ell(Q))}{|Q|^{1/p_2}}\Big(\int_Q |f(x)|^{p_2}\,d x\Big)^{1/p_2}
        &\lesssim{\varphi_2(\ell(Q))}\Big(\frac{1}{w(Q)}\int_Q |f(x)|^{p_1}w(x)\,d x\Big)^{1/p_1}\\
        &= \frac{\varphi_2(\ell(Q))}{\varphi_1(\ell(Q))}\Big(\frac{w(Q)}{|Q|}\Big)^{-1/p_1}\frac{\varphi_1(\ell(Q))}{|Q|^{1/p_1}}\Big(\int_Q |f(x)|^{p_1}w(x)\,d x\Big)^{1/p_1}\\
        &\lesssim\|f\mid \mathfrak M^{p_1}_{\varphi_1}(\mathbb R^d; w)\|.
    \end{align*}
    Hence we complete the proof.
\end{proof}

\begin{remark}\label{rem:1w1w2}
Let $0<p_2<p_1<\infty$, $\varphi_i\in\mathcal G_{p_i}$, and
 $w_i$ be weights, $i=1,2$. Assume that there exists
$\gamma\in(1,\infty)$, with $\gamma'=\gamma/(\gamma-1)$, such that
\[
\begin{split}
&\sup_{Q\in\mathcal Q}
\frac{\varphi_2(\ell(Q))}{\varphi_1(\ell(Q))}
|Q|^{\frac1{p_1}-\frac1{p_2}}
\Bigl(\int_Q
w_1(x)^{-\gamma\frac{p_2}{p_1-p_2}}\,dx\Bigr)^{
\frac1\gamma(\frac1{p_2}-\frac1{p_1})}
\\
&\qquad\times
\Bigl(\int_Q
w_2(x)^{\gamma'\frac{p_1}{p_1-p_2}}\,dx\Bigr)^{
\frac1{\gamma'}(\frac1{p_2}-\frac1{p_1})}
<\infty.
\end{split}
\]
Then applying H\"older's inequality twice yields 
\[
\mathfrak M^{p_1}_{\varphi_1}(\mathbb R^d;w_1)
\hookrightarrow
\mathfrak M^{p_2}_{\varphi_2}(\mathbb R^d;w_2).
\]
\end{remark}

Next we turn to the spaces $\mathfrak M_u^p(\mathbb R^d; w)$. In the class $\mathcal A_1$, we obtain a rigidity phenomenon.

\begin{proposition}\label{pro:1Ma-Mb}
Let $0<p_i\le u_i<\infty$, $i=1,2$. Consider the embedding 
\begin{align}\label{M1wa-Mb} 
    \mathfrak M_{p_1}^{u_1}(\mathbb R^d; w)\hookrightarrow\mathfrak M_{p_2}^{u_2}(\mathbb R^d).
\end{align}
Then the following assertions hold.
\begin{enumerate}[\bfseries\upshape  (i)]
    \item If $w\in \mathcal{A}_\infty\cap \mathcal{B}_{u_1,p_1}$ and \eqref{M1wa-Mb} holds, then 
    \begin{align}
        &p_1\ge p_2,\qquad u_1\ge u_2, \label{eq:1Mab-cond1}\\
        &\sup_{Q\in\mathcal Q}|Q|^{\frac{1}{u_2}+\frac{1}{p_1}-\frac{1}{u_1}}w(Q)^{-1/p_1}<\infty, \label{eq:1Mab-cond2}\\
        &\frac{p_1}{u_2}-\frac{p_1}{u_1}+1\le q_w. \label{eq:1Mab-cond3}
    \end{align}

    \item If $p_2\le p_1$, $w\in\mathcal A_{p_1/p_2}$ and \eqref{eq:1Mab-cond2} holds, then the embedding \eqref{M1wa-Mb} holds.

    \item If $w\in\mathcal A_1\cap \mathcal{B}_{u_1,p_1}$, then the embedding \eqref{M1wa-Mb} holds if and only if 
    \begin{align}\label{eq:1Mab-c}
        u_1=u_2,\quad p_1\ge p_2,\quad \inf_{Q\in\mathcal Q:\,|Q|=1}w(Q)>0.
    \end{align}
\end{enumerate}
\end{proposition}

\begin{proof}
\emph{Step }1. 
The assertions $p_1\ge p_2$ and \eqref{eq:1Mab-cond2} in (i), as well as the sufficiency in (ii) follow from Theorem~\ref{thm:1Mfa-Mfb} by choosing $\varphi_i(t)=t^{d/u_i}$, $i=1,2$.

To show the necessity of $u_1\ge u_2$. Choose a Lebesgue point $x_0$ of $w$ such that 
$0<w(x_0)<\infty$. For cubes $Q_r$ centered at $x_0$ with side length $r$, $\frac{w(Q_r)}{|Q_r|}\to w(x_0)$, $r\downarrow 0$. Condition~\eqref{eq:1Mab-cond2} can be written as 
\[
|Q_r|^{1/u_2-1/u_1}\Bigl(\frac{w(Q_r)}{|Q_r|}\Bigr)^{-1/p_1}\lesssim 1.
\]
Since the second factor in the above estimate converges to a finite positive number, letting $r\downarrow 0$, we obtain $\frac{1}{u_2}-\frac{1}{u_1}\ge 0$, which is equivalent to $u_1\ge u_2$.

It remains to show condition~\eqref{eq:1Mab-cond3} in (i) is necessary. Let $w\in\mathcal A_q$ with $q>1$ and $Q\supset E:=[0,1]^d$. Then, by \eqref{eq:EB} and  \eqref{eq:1Mab-cond2}, we have 
\[
|Q|^{\frac{p_1}{u_2}-\frac{p_1}{u_1}+1} \lesssim w(Q) \lesssim |Q|^q.
\]
Taking the limit with $|Q|\rightarrow\infty$ and then the infimum over possible $q$ we obtain $\frac{p_1}{u_2}-\frac{p_1}{u_1}+1\leq q_w$.

\emph{Step }2. Now we prove (iii) and first show sufficiency of \eqref{eq:1Mab-c}. Since $w\in\mathcal A_1$, we have
\[
\frac1{|Q|}\int_Q g(x)\,dx
\lesssim
\frac{1}{w(Q)}\int_Q g(x)w(x)\,dx
\qquad \text{for all } g\ge0.
\]
Applying this with $g=|f|^{p_1}$ yields
\begin{align*}
    \Big(\int_Q |f(x)|^{p_2}\,dx\Big)^{1/p_2}
&\lesssim
|Q|^{1/p_2-1/p_1}
\Big(\int_Q |f(x)|^{p_1}\,dx\Big)^{1/p_1}\\
&\lesssim
|Q|^{1/p_2}w(Q)^{-1/p_1}
\Big(\int_Q |f(x)|^{p_1}w(x)\,dx\Big)^{1/p_1}.
\end{align*}
Hence
\[
|Q|^{\frac1{u_2}-\frac1{p_2}}
\Big(\int_Q |f(x)|^{p_2}\,dx\Big)^{1/p_2}
\lesssim
|Q|^{\frac1{u_2}-\frac1{u_1}+\frac1{p_1}}w(Q)^{-1/p_1}
|Q|^{\frac1{u_1}-\frac1{p_1}}
\Big(\int_Q |f(x)|^{p_1}w(x)\,dx\Big)^{1/p_1}.
\]
Since $u_1=u_2$, it remains to show that ${|Q|}/{w(Q)}$ is bounded uniformly over all cubes $Q$.
Indeed, set
\[
c_0:=\inf_{R\in\mathcal Q:\, |R|=1} w(R)>0.
\]
For all cubes $R\in\mathcal Q$ with $|R|=1$, the $\mathcal A_1$ condition gives
\[
c_0\le w(R)
=
\frac1{|R|}\int_R w(x)\,dx
\le [w]_{\mathcal A_1}\operatorname*{ess\,inf}_{x\in R}w(x),
\]
by \eqref{eq:pro-A1}. 
Hence $\operatorname*{ess\,inf}_{R}w\ge c_0/[w]_{\mathcal A_1}$ for all unit cubes $R$, which further implies that 
$w(x)\ge c_0/[w]_{\mathcal A_1}$ for a.e. $x\in\mathbb R^d$. 
Therefore, for all cubes $Q$,
\[
w(Q)=\int_Q w(x)\, dx\ge \frac{c_0}{[w]_{\mathcal A_1}}\,|Q|,
\]
and consequently
\[
\sup_{Q\in\mathcal Q}\Bigl(\frac{|Q|}{w(Q)}\Bigr)^{1/p_1}\le \Bigl(\frac{[w]_{\mathcal A_1}}{c_0}\Bigr)^{1/p_1}<\infty.
\]
Thus the factor $(|Q|/w(Q))^{1/p_1}$ is uniformly bounded.
Taking the supremum over all cubes yields
\[
\|f\mid \mathfrak M_{p_2}^{u_2}(\mathbb R^d)\|
\lesssim
\|f\mid \mathfrak M_{p_1}^{u_1}(\mathbb R^d;w)\|,
\]
which proves the embedding.

\emph{Step }3. It remains to prove necessity of \eqref{eq:1Mab-c}. Assume that \eqref{M1wa-Mb} holds. Then, by \eqref{eq:norm-1Q1w'}, there exists a constant $C>0$ such that for every cube $Q$,
\begin{equation}\label{eq:assi-4.11}
|Q|^{\frac1{u_2}-\frac1{u_1}+\frac1{p_1}}
\le C\,w(Q)^{1/p_1}.
\end{equation}
Taking $|Q|=1$ yields
\[
\inf_{Q\in\mathcal Q:\,|Q|=1}w(Q)\ge C^{-p_1}>0.
\]

We now show that $u_1=u_2$. By (i), the embedding implies $u_1\ge u_2$ and
\[
1+\frac{p_1}{u_2}-\frac{p_1}{u_1}\le q_w.
\]
Since $w\in\mathcal A_1$, we have $q_w=1$. Hence
\[
\frac1{u_2}-\frac1{u_1}\le0,
\]
which means $u_1\le u_2$. Therefore $u_1=u_2$.

Finally, the necessity of $p_1\ge p_2$ follows from Proposition~\ref{pro:regularity-cubes} and \eqref{eb:Ma-Mb-Omega}. This completes the proof.
\end{proof}

\begin{remark}[Comparison with the Komori--Shirai case]
\label{rem:A1-rigidity-com}
It was proved in \cite[Theorem~4.7]{HS17} that, for
$0<p_i\le u_i<\infty$,
\[
\mathcal M_{p_1}^{u_1}(\mathbb R^d;w)
\hookrightarrow
M_{p_2}^{u_2}(\mathbb R^d),
\qquad w\in\mathcal A_1,
\]
holds if and only if 
\[
u_1=u_2,\qquad p_2\le p_1,\qquad
\inf_{Q\in\mathcal Q:\, |Q|=1} w(Q)>0.
\]
Proposition~\ref{pro:1Ma-Mb} shows that the same rigidity phenomenon holds for the Samko-type spaces. 
Thus, in the $\mathcal A_1$-case, the Morrey parameter is rigid: the embedding forces $u_1=u_2$.
\end{remark}

We conclude with the compactness question. The classical embedding $L^{p_1}(\Omega)\hookrightarrow L^{p_2}(\Omega)$ on a bounded domain is continuous but not compact when $p_1\ge p_2$; see \cite[p.~95]{PKJF13}. This phenomenon extends to suitable weighted Lebesgue spaces \cite{HT94}. Here we show that the same non-compactness persists in the present Morrey-type setting, provided the weight admits a regularity cube; see Proposition~\ref{pro:regularity-cubes} and also \cite[Corollary~2.7]{HS17}.

\begin{proposition}
\label{prop:compact}
Let $0<p_i<\infty$, $\varphi_i=\varphi_i(r)\in\mathcal G_{p_i}$, $i=1,2$, and let
$w$ be a weight which admits a regularity cube as in Proposition~\ref{pro:regularity-cubes}. If
\begin{equation}\label{eb:no-compact}
\mathfrak M^{p_1}_{\varphi_1}(\mathbb R^d;w)
\hookrightarrow
M^{p_2}_{\varphi_2}(\mathbb R^d)
\end{equation}
is continuous, then this embedding is not compact.

The same conclusion holds for the corresponding classical spaces
$\mathfrak M_u^p(\mathbb R^d;w)$, and for the Komori--Shirai-type spaces
$\mathcal M_\varphi^p(\mathbb R^d;w)$ and $\mathcal M_u^p(\mathbb R^d;w)$.
\end{proposition}

\begin{proof}

Let $Q_0$ be a regularity cube for $w$. By Proposition~\ref{pro:regularity-cubes}, $w(x) \sim 1$ for a.e. $x\in Q_0$.
Choose a cube $P\subset Q_0$ with positive measure. On functions supported in $P$, the weighted and unweighted local $L^{p_1}$-quantities are equivalent.

Let $(r_j)_{j\in\mathbb N}$ be a sequence of Rademacher-type functions defined on $P$. 
Specifically, we choose $(r_j)_{j\in\mathbb N}$ such that
\[
|r_j(x)|=1 \quad\text{for a.e. }x\in P,
\]
and for any $j\neq k$, 
\[
|\{x\in P:r_j(x)\neq r_k(x)\}|=\frac{|P|}{2}.
\]

Set
\[
f_j:=r_j\chi_P,\qquad j\in\mathbb N.
\]
Since $|f_j|=\chi_P$, the source norms of all $f_j$ are equal. 
Moreover, $w\sim 1$ a.e. on $P$ yields 
\[
\int_Q\chi_P(x)w(x)\,dx\sim  |Q\cap P|.
\]
for all cubes $Q$. 
Hence
\[
0<
A:=
\|\chi_P\mid
\mathfrak M^{p_1}_{\varphi_1}(\mathbb R^d;w)\|
\sim 
\|\chi_P\mid M^{p_1}_{\varphi_1}(\mathbb R^d)\|
<\infty.
\]
Hence the normalized sequence $g_j:=A^{-1}f_j$
satisfies
\[
\|g_j\mid \mathfrak M^{p_1}_{\varphi_1}(\mathbb R^d;w)\|=1,
\qquad j\in\mathbb N.
\]
Thus $(g_j)_{j\in\mathbb N}$ is bounded in the space $\mathfrak M^{p_1}_{\varphi_1}(\mathbb R^d;w)$.

We next show that it has no convergent subsequence in the space $M^{p_2}_{\varphi_2}(\mathbb R^d)$. 
For $j\neq k$, 
\[
\|g_j-g_k\mid M^{p_2}_{\varphi_2}(\mathbb R^d)\|
\ge
A^{-1}\frac{\varphi_2(\ell(P))}{|P|^{1/p_2}}
\Bigl(\int_P |r_j(x)-r_k(x)|^{p_2}\,dx\Bigr)^{1/p_2}.
\]
Since $|r_j-r_k|=2$ on a set of measure $|P|/2$, we obtain
\[
\Bigl(\int_P |r_j(x)-r_k(x)|^{p_2}\,dx\Bigr)^{1/p_2}
=
\Bigl(2^{p_2}\frac{|P|}{2}\Bigr)^{1/p_2}
=
2^{1-1/p_2}|P|^{1/p_2}.
\]
Consequently,
\[
\|g_j-g_k\mid M^{p_2}_{\varphi_2}(\mathbb R^d)\|
\ge
A^{-1}2^{1-1/p_2}\varphi_2(\ell(P))
=:c>0
\]
for all $j\neq k$. Therefore, $(g_j)_{j\in\mathbb N}$ has no Cauchy subsequence in the target space $M^{p_2}_{\varphi_2}(\mathbb R^d)$, and hence no convergent subsequence.
Thus the image of the bounded sequence $(g_j)_{j\in\mathbb N}$ under the embedding is not relatively compact, so the embedding cannot be compact.

The same argument applies to
$\mathfrak M_u^p(\mathbb R^d;w)$ by taking
$\varphi(t)=t^{d/u}$, and to the Komori--Shirai-type spaces
$\mathcal M_\varphi^p(\mathbb R^d;w)$ and
$\mathcal M_u^p(\mathbb R^d;w)$, since the argument uses only the
local equivalence of the weighted and unweighted quantities on the
regularity cube and a lower estimate for the target quasi-norm on the
fixed cube $P$. For $\mathcal M_u^p(\mathbb R^d;w)$, the conclusion
also follows from \cite[Corollary~4.10]{HS17}.
\end{proof}

\section*{Acknowledgements}

The author wishes to express sincere gratitude to \emph{Key Laboratory of Mathematics and Complex Systems, Ministry of Education, China} for the financial support that enabled the author to visit Professor Dorothee D. Haroske. 
The author is also grateful to Professor Dorothee D. Haroske for the discussions and suggestions concerning Section \ref{ssec:Rc}.

\section*{Declaration on the use of AI}
During the preparation of this manuscript, the author used OpenAI's ChatGPT to assist in developing an initial candidate for Example~\ref{ex:no-regularity-cube} and in improving the language and presentation. 
All AI-generated suggestions were independently checked, revised, and verified by the author. 
The author takes full responsibility for the mathematical content and the final version of the manuscript.

\section*{Conflict of Interest}

The author declares no conflict of interest.

\section*{Data Availability Statement}

Data sharing is not applicable to this article as no datasets were generated or analysed during the current study.


\bigskip~

{\small
\noindent
Zhen Liu\\
Laboratory of Mathematics and Complex Systems (Ministry of Education of China)\\
School of Mathematical Sciences \\
Beijing Normal University\\
100875 Beijing\\
China\\
{\tt zhen.liu@bnu.edu.cn}
}

\end{document}